\documentclass[11pt,a4paper,oneside]{article}
\usepackage[utf8]{inputenc}
\usepackage[T1]{fontenc}
\usepackage{amsmath,amssymb,amsthm}
\usepackage{mathrsfs}
\usepackage{enumerate}
\usepackage{wasysym}
\usepackage{cite}
\usepackage{graphicx}
\usepackage{braket}
\usepackage[dvipsnames]{xcolor}
\usepackage{tikz}
\usepackage{tcolorbox}
\usepackage{hyperref}
\usepackage{geometry}
\usepackage{bookmark}

\usepackage{tikz-cd}
\usetikzlibrary{arrows}

\usetikzlibrary{decorations.pathreplacing}
\usetikzlibrary{quotes,angles}

\tcbuselibrary{theorems}
\tcbuselibrary{breakable}
\tcbuselibrary{skins}
\hypersetup{
    colorlinks,
    citecolor=red,
    filecolor=black,
    linkcolor=blue,
    urlcolor=black
}

\usepackage[T1]{fontenc}
\tcbuselibrary{breakable}
\usepackage{titlesec, blindtext, color}
\definecolor{gray75}{gray}{0.75}
\newcommand{\hsp}{\hspace{20pt}}
\titleformat{\chapter}[hang]{\Huge\bfseries}{\thechapter\hsp\textcolor{gray75}{|}\hsp}{0pt}{\Huge\bfseries}

\newenvironment{claim}[2]
    {\begin{center}
    \begin{tcolorbox}
    [colback=white!20!white,colframe=black!20!white,sharp corners, breakable]
    \begin{minipage}[t]{.15\textwidth}
    \textbf{Claim\# #1: }
    \end{minipage}
    \begin{minipage}[t]{.85\textwidth}
    #2
    \end{minipage}
\ \\
 \textit{Proof:}
    }
    {\flushright{\checkmark}
    \end{tcolorbox}
    \end{center}
    }

\newtheorem{theorem}{{Theorem}}[section]
\newtheorem{lemma}[theorem]{{Lemma}}

\newtheorem{korollar}[theorem]{{Corollary}}

\newcommand{\Sp}{\mathbb S}

\newcommand{\N}{\mathbb{N}}
\newcommand{\R}{\mathbb{R}}

\newcommand{\tr}{\operatorname{tr}}
\newcommand{\avint}{{-}\hspace{-.4cm}\int} 

\renewcommand{\theequation}{\arabic{section}.\arabic{equation}}
\begin{document}
\author{Jan-Henrik Metsch }
\title{On the area-preserving Willmore flow of small bubbles\\ sliding on a domain's boundary}
\maketitle

\begin{abstract}
We consider the area-preserving Willmore evolution of surfaces $\phi$ that are close to a half-sphere with a small radius, sliding on the boundary $S$ of a domain $\Omega$ while meeting it orthogonally. We prove that the flow exists for all times and keeps a `half-spherical' shape. Additionally, we investigate the asymptotic behaviour of the flow and prove that for large times the barycenter of the surfaces approximately follows an explicit ordinary differential equation. Imposing additional conditions on the mean curvature of $S$, we then establish convergence of the flow. 
\end{abstract}
\textbf{MSC2020 Subject Classification}: Primary: 53E40, Secondary: 35G31, 47J07\\

\section{Introduction}
\setcounter{equation}0
Given an immersion $\phi:\Sp^2_+\rightarrow\R^3$ with mean curvature $H$ the \emph{Willmore energy} is defined as 
$$\mathcal W(\phi):=\frac14\int_{\Sp^2_+}H^2d\mu_\phi.$$
For a sufficiently smooth bounded domain $\Omega\subset\R^3$ with boundary $S:=\partial\Omega$ we consider the class 
$$\mathcal M(S):=\left\{\phi\in C^4(\Sp^2_+,\R^3)\textrm{ immersed}\ \bigg|\ \phi(\partial \Sp^2_+)\subset S, \frac{\partial\phi}{\partial\eta}=N^S\circ\phi\right\}$$
of immersed surfaces meeting $S$ orthogonally along the boundary. Here $\eta$ and $N^S$ denote the interior unit normals of $\Sp^2_+$ and $\Omega$ along their respective boundaries. In this article we study the \emph{area preserving Willmore flow} inside a suitable subclass of $\mathcal M(S)$. That is, denoting the scalar Willmore operator by $W$ and the inner normal of $\phi$ by $\nu$ the equation
\begin{equation}\label{introevlaw}
\left\{\begin{aligned}
\langle\dot\phi(t),\nu\rangle&=- W(\phi(t))+\alpha(t) H(\phi(t))\\
\frac{\partial H}{\partial\eta}&+h^S(\nu,\nu)H=0\\
\phi(t)&\in \mathcal M(S)
\end{aligned}\right.
\end{equation}
Here $\alpha$ is a suitable Lagrange multiplier. The additional third-order Neumann-type boundary condition appearing in (\ref{introevlaw}) arises naturally from the requirement of steepest energy descent (for details see Subsection \ref{flow}). \\

Alessandroni and Kuwert studied the corresponding elliptic problem in \cite{AK} and constructed critical points inside $\mathcal M(S)$ that are (up to rotation and scaling) close to half-spheres attached to $S$. Introducing the class $\mathcal S^{4,\gamma}(\lambda,\theta)$ as the subset of $\mathcal M(S)$ consisting of almost half-spheres with area $2\pi\lambda^2$ where $\theta>0$ measures the $C^{4,\gamma}$-deviation from the round half-sphere (for a detailed definition see Subsection \ref{terminilogy}) we first prove long-time existence of the flow:

\begin{theorem}[Long-time existence]\label{theorem1}\ \\
Let $\Omega\in C^{17}$. Then there exist $\lambda_0>0$ and $\theta_1,\theta_0>0$ such that:
\begin{enumerate}
    \item For  $\lambda\in (0,\lambda_0]$ and $\phi_0\in \mathcal\mathcal S^{4,\gamma}(\lambda,\theta_0)$ there exists a classical solution $\phi:[0,\infty)\rightarrow C^ {4,\gamma}(\Sp^2_+)$ to the area preserving Willmore flow that stays in the class $\mathcal S^{4,\gamma}(\lambda,\theta_1)$ and satisfies $\phi(0)=\phi_0$. Moreover $\phi$ is unique up to reparameterizations.
    \item Given $\phi(t)$ as in part 1, any sequence $t_n\rightarrow\infty$ contains a subsequence $(t_{n_k})$ such that $\phi(t_{n_k})$ converges to a critical point of the elliptic problem (\ref{elliptic}) in $C^{4,\beta}$ for any $\beta<\gamma$.
\end{enumerate} 
\end{theorem}

We prove Theorem \ref{theorem1} by writing $\phi(t)$ as a normal graph over a small half-sphere centered at the surface's Riemannian barycenter $\xi(t):=C[\phi(t)]\in S$ and derive evolution equations for both the graph function $u(t)$ and $\xi(t)$. The proof then follows a two-step approach. First, an arbitrary curve $\xi(t)$ is chosen, rendering the evolution equation for $u$ as a gradient flow with the time-dependent constraint of prescribed barycenter curve. Existence of $u$ and suitably decay estimates can then be established by an application of the implicit function theorem as the additional constraint fixes the kernel of the elliptic operator appearing in the linearized evolution equation for $u$. Afterwards, an appropriate choice for $\xi$ is derived.\\

Due to the application of the implicit function theorem, Theorem \ref{theorem1} is of perturbative nature, which is reflected in the smallness assumption on $\theta_0$. In Riemannian manifolds, similar arguments have been used by Mattuschka \cite{Mattuschka} to study the area-preserving Willmore flow and by Alikakos and Freire \cite{nicholas} in the context of the normalised mean curvature flow. In \cite{simonflow} similar techniques are used to prove exponential convergence of the Willmore flow to a sphere if it is initialised sufficiently close to one. In  \cite{bellettinifusco} Bellettini and Fusco study the volume-preserving mean curvature flow of surfaces close to half-spheres that slide on the boundary of a domain $\Omega$ while meeting it orthogonally and prove results similar to the ones we establish here. It must, however, be noted that in their situation, the orthogonality boundary condition arises naturally while we impose it artificially. Our natural boundary condition is the third-order one appearing in (\ref{introevlaw}).
General gradient flows with time-dependent constraints have been studied in \cite{timedepconst}.\\
The general methodology of these papers is a form of Lyapunov-Schmitt reduction and originates from Ye \cite{Ye}. It has also been employed in the study of related time-independent problems (see e.g. \cite{AK, lamm2009foliations,lamm2009small, pacard, Mondino,malchiodispheres}). An application to Willmore tori in Riemannian manifolds is considered in \cite{malchioditori}. For a general overview, we refer to \cite{malchiodireview}.\\

 The present article establishes similar results to the ones in \cite{bellettinifusco}. In doing so we face the difficulties of dealing with both a fourth-order and an initial boundary value problem. After an extensive search through the literature, it seems that such problems have not been studied using Ye's methodology.\\

Inspired by Theorem 1.6 in \cite{bellettinifusco}, for $K>0$ we introduce the subclass $\mathcal S^-(\lambda, K)\subset \mathcal S^{4,\gamma}(\lambda, K\lambda)$ which imposes an additional smallness requirement on the antisymmetric part of the graph function (see Subsection \ref{terminilogy} for a formal definition). The subclass $\mathcal S^-(\lambda,K)$ arises naturally from the analysis of Alessandroni and Kuwert \cite{AK} as the critical points they construct belong to it (for a proof of this see Appendix \ref{AKcritpoint}). We prove that all flow lines constructed in Theorem \ref{theorem1} eventually must enter and remain in $\mathcal S^-(\lambda, K)$ for some $K>0$ independent of $\lambda$ and use this observation to study the asymptotic behaviour of the associated Riemannian barycenter curves.

\begin{theorem}[Long-time behaviour of the flow]\label{theorem2}\hfill
\begin{enumerate}
    \item Let $\Omega,\lambda_0,\theta_0$ be as in Theorem \ref{theorem1} and $\lambda\leq\lambda_0$. Then there exists a time $T_0=T_0(\lambda,\theta_0)>0$ and $K>0$ independent of $\lambda$ and $\theta_0$ such that given $\phi_0\in \mathcal S(\lambda,\theta_0)$ the flow line $\phi(t)$ from Theorem \ref{theorem1} satisfies $\phi(t)\in \mathcal S^-(\lambda, K)$ for all $t\geq T_0$. 
    \item Let $\Omega\in C^{21}$, $M>0$ and denote the mean curvature of $S=\partial \Omega$ by $H^ S$. There exists $\lambda_0(M)>0$ such that for all $\lambda\leq\lambda_0$  and flow lines satisfying $\phi(t)\in\mathcal S^-(\lambda,M)$ for all times $t\geq 0$ the barycenter curve $\xi(t)$ satisfies the estimate 
$$\sup_{t\geq 0}\left|\dot\xi(t)-\frac3{2\lambda}\nabla H^S(\xi(t))\right|\leq C(M).$$
\end{enumerate}
\end{theorem}

In particular, when introducing the \emph{fast time} $\tau:=\lambda^{-1} t$, and sending $\lambda\rightarrow 0^ +$ the flow collapses to a point that moves according to the ordinary differential equation 
\begin{equation}\label{gradflowh}
\frac {d}{d\tau}\xi(\tau)=\frac32\nabla H^S(\xi(\tau)).
\end{equation}

Let us motivate the structure of evolution equation (\ref{gradflowh}). The critical points $\phi^{a,\lambda}$ Alessandroni and Kuwert construct in \cite{AK} are labelled by their `radius' $\lambda$ and barycenter $a\in S$ (see Appendix \ref{AKcritpoint}). In their paper they derive the expansion 
\begin{equation}\label{introenegryformula}
\mathcal W[\phi^{a,\lambda}]=2\pi-\lambda\pi H^S(a)+\mathcal O(\lambda^2).
\end{equation}
Given a flow line  $\phi(t)\in S^-(\lambda, M)$ with barycenter $\xi(t):=C[\phi(t)]$ let us assume that Equation (\ref{introenegryformula}) holds, at least for large times. Then formally differentiating with respect to time and substituting (\ref{gradflowh}) we recover dissipation of energy:
\begin{equation}\label{Wdecayintro}
\frac d{dt}\mathcal W[\phi(t)]=-\frac {3\pi}2|\nabla H^ S(\xi(t))|^2+\mathcal O(\lambda)
\end{equation}
The structure of the barycenter equation is degenerate in the limit $\lambda\rightarrow 0^+$.  This leads to additional technical difficulties when studying Equation (\ref{introevlaw}) in the limit $\lambda\rightarrow 0^+$ that are absent in both the mean curvature and the Riemannian case (see \cite{Mattuschka,bellettinifusco}). This problem is overcome by altering the method followed in \cite{Mattuschka}.\\

Finally, we investigate the convergence of the flow (\ref{introevlaw}). 

\begin{theorem}[Convergence]\label{theorem3}\ \\
Let $\Omega\in C^{21}$ and assume that the mean curvature $H^S$ of $S$ only has finitely many critical points $p_1,\hdots,p_m$ all of which are nondegenerate. Then for any flow line from Theorem \ref{theorem1} there exists a critical point $\phi_*\in S^-(\lambda,K)$ of the elliptic problem (\ref{elliptic}) such that $\phi(t)\rightarrow\phi_*$ in $C^{4,\beta}$ for $t\rightarrow\infty$ and any $\beta<\gamma$. $\phi_\infty$ is of the form described in Appendix \ref{AKcritpoint}.
\end{theorem}

The proof is based on an idea from \cite{bahri}. Equation (\ref{Wdecayintro}) implies that the barycenter curve $\xi$ of a solution cannot stay far away from critical points for all times. As Equation (\ref{gradflowh}) implies that $H^S(\xi(t))$ is increasing along the flow, the assumption of finitely many critical points then implies a stabilisation of $\xi$ near a critical point of $H^S$. Using a uniqueness Theorem from \cite{AK} (that requires the nondegeneracy assumption) then allows us to establish convergence.

\section{Preliminaries}
\setcounter{equation}0
\subsection{Terminology}\label{terminilogy}
Let $\tilde g$ be a metric on $\R^3$ close to the euclidean metric $\delta$ in $C^ l$ for large enough $l\in\N$, consider a suitably smooth immersion $f:\Sp^2_+\rightarrow (\R^3,\tilde g)$ and put $g:=f^*\tilde g$. We denote the inner normal ($-\omega$ for the round half-sphere in euclidean space) of $f$ with respect to $\tilde g$ by $\tilde\nu$ and its inner conormal ($e_3$ for the round half-sphere in euclidean space) by $\tilde \eta$. We define the mean curvature $H[f,\tilde g]$ of $f$ with the convention that for the round half-sphere in euclidean space $H=2$. Let $h^ 0$ denote the traceless second fundamental form of $f$ and $\operatorname{Ric}_{\tilde g}$ the Ricci tensor of $(\R^3,\tilde g)$. The scalar Willmore gradient is then given by 
$$W[f,\tilde g]:=\frac12\left(\Delta_gH+(|h^ 0|^2+\operatorname{Ric}_{\tilde g}(\tilde\nu,\tilde\nu))H\right).$$
For a proof see Theorem 1 in \cite{AK} and note the following difference in conventions: We have included $1/2$ in the definition of $W$. Given a function $F[f,\tilde g]$ that is invariant under reparameterizations (e.g. the area $A$) we denote the gradient along the normal bundle by $\nabla F[f,\tilde g]$. That is, for $\psi:\Sp^2_+\rightarrow\R$
$$\frac d{ds}\bigg|_{s=0} F[f+s\psi\tilde\nu,\tilde g]=\int_{\Sp^2_+}\nabla F[f,\tilde g]\psi d\mu_{g}.$$
As an example, denoting the inclusion $\Sp^2_+\hookrightarrow\R^3$ by $f_0$ we have $\nabla A[f_0,\delta]=-H[f_0,\delta]=-2$.\\

$\Delta$ always denotes the standard Laplacian on $\Sp^2$.\\

For a bounded domain $\Omega\in \R^3$ of class $C^ n$ with large enough $n$ put $S:=\partial\Omega$ and denote the inner normal (that is pointing into $\Omega$) of $S$ by $N^S$. Let $p\in S$ and choose $C^{n-1}$-vector fields $b_1$ and $b_2$ in a neighbourhood $B_{r_0}(p)\cap S$ with $r_0$ independent of $p$ such that $(b_1(q), b_2(q))$ provides an orthonormal basis of $T_q S$ for each $q\in B_{r_0}(p)\cap S$. Near $p$ we have the local graph representation 
$$f[p,\cdot]: D_{r_1}\rightarrow\R^3,\ f[p,x]:=p+x^1 b_1(p)+x^2 b_2(p)+\varphi[p, x]N^S(p)$$
of $S$ defined on a disk $D_{r_1}:=\set{x\in\R^2\ | |x|<r_1}$ with $r_1>0$ independent of $p$. The map $S\times D_{r_1}\ni (p, x)\mapsto \varphi[p, x]\in\R$ is of class $C^{n-1}$, satisfies $\varphi[p,0]=0$, $D_2\varphi[p,0]=0$ and the estimates 
$$|\varphi[p, x]|\leq C(\Omega)|x|^2,\hspace{.3cm}|D_2\varphi[p, x]|\leq C(\Omega)|x|\hspace{.3cm}\text{and}\hspace{.3cm}\|\varphi\|_{C^{n-1}}\leq C(\Omega).$$
After potentially shrinking $r_1$, we extend the graph representation to a diffeomorphism 
$$F[p,\cdot]:D_{r_1}\times (-r_1, r_1)\rightarrow\operatorname{im}(F[p,\cdot])\subset \R^3,\ F[p,x,z]:=f[p, x]+zN^S(p).$$
For small $\lambda\in(0,\lambda_0(\Omega))$ we now introduce $F^\lambda:S\times Z_2\rightarrow\R^3$ ($Z_2:=\bar D_2\times [-2,2]$) by mapping $F^\lambda[p, x,z]:=F[p,\lambda x,\lambda z]$ and set 
\begin{equation}\label{metricdefinition}
\tilde g^{p,\lambda}:=\frac1{\lambda^2}\left(F^\lambda[p,\cdot]\right)^*\delta
\end{equation}
where $\delta$ denotes the euclidean metric on $\R^3$. We abbreviate $\partial_i\varphi[p,x]:=\partial_{x_i}\varphi[p,x]$. A quick computation (that is carried out in \cite{AK}) gives the concrete formula
\begin{equation}\label{metricformula}
\tilde g^{p,\lambda}=I_3+\begin{bmatrix}
\partial_1\varphi[p,\lambda x]\partial_1\varphi[p,\lambda x] & \partial_1\varphi[p,\lambda x]\partial_2\varphi[p,\lambda x] &  \partial_1\varphi[p,\lambda x]\\
 \partial_1\varphi[p,\lambda x]\partial_2\varphi[p,\lambda x] &  \partial_2\varphi[p,\lambda x]\partial_2\varphi[p,\lambda x] & \partial_2\varphi[p,\lambda x]\\
 \partial_1\varphi[p,\lambda x] & \partial_2\varphi[p,\lambda x] & 0
\end{bmatrix}.
\end{equation}
For $n\geq l+3$, the map $(p,\lambda)\mapsto \tilde g^{p,\lambda}\in C^ l(Z_2, M_3(\R))$ is of class $C^{n-l-3}$ as $\varphi\in C^{n-1}$. For regularity of such \emph{`meta maps'} we point to Appendix \ref{schauder}.

\paragraph{Functional spaces}\ \\
Let $\gamma\in(0,1)$ and $k\in\N$. We consider the Hölder spaces $C^{k,\gamma}(\Sp^2_+)$, $C^{k,\gamma}(\partial\Sp^2_+)$ and denote their norms by $\|\cdot\|_{C^{k,\gamma}}$ as long as the domain is clear from context.\\

Given $T>1$ and denoting the floor function by $[\cdot]^-$, we also consider the parabolic Hölder spaces $C^{k,[k/4]^-,\gamma}([0,T]\times \Sp^2_+)$ and $C^{k,[k/4]^-,\gamma}([0,T]\times \partial\Sp^2_+)$ of functions that are $k$-times continuously differentiable in space, $[k/4]^-$-times continuously differentiable in time and satisfy suitable Hölder conditions. A detailed definition is given in Appendix \ref{schauder}. We denote the norms by $\|\cdot\|_{C^{k,[k/4]^-,\gamma}_T}$ as long as the spatial domain is clear from context.\\

For given $U\subset \R^3$, we consider the spaces $C^{k,\frac\gamma4}([0,T],U)$ and denote their norms by $\|\cdot\|_{C^{1,\frac\gamma4}_T}$ as $U$ will always be clear from context. \\

If $X$ is some functional space, $x\in X$ and $r>0$ we put $X(x;r):=\set{\tilde x\in X\ |\ \|x-\tilde x\|_X\leq r}$. We write $X(r):=X(0;r)$.\\

\paragraph{Summation convention}\ \\
We use the following summation convention. Every repeated index is summed over. If the index is Latin, it takes the values $i=1,2$ and if it is Greek, it takes the values $\mu=1,2,3$ or $\mu=0,1,2$ (which of the two will always be clear from context). In ambiguous cases the summation symbol is included.

\paragraph{Almost half-spheres}\ \\
For $\Omega\subset \R^3$ and $S$ as described above and $\lambda>0$ let 
\begin{align*} 
\mathcal M^{4,\gamma}(S)&:=\set{\phi\in C^{4,\gamma}(\Sp^2_+,\R^3)\ |\ \text{immersed, $\phi(\partial \Sp^2_+)\subset S$ and $\phi\perp S$ along $\partial \Sp^2_+$}},\\
\mathcal M^{4,\gamma}_\lambda(S)&:=\set{\phi\in\mathcal M^{4,\gamma}(S)\ |\ A[\phi]=2\pi\lambda^2}.
\end{align*}
For $u\in C^{4,\gamma}(\Sp^2_+)$ with small enough $C^{4,\gamma}$-norm, $\lambda>0$ and $p\in S$ we consider 
$$f_u:\Sp^2_+\rightarrow\R^3,\ f_u(\omega):=(1+u(\omega))\omega\hspace{.5cm}\text{and}\hspace{.5cm}\phi_{p,u}^\lambda(\omega):=F^\lambda[p, f_u(\omega)].$$
We say that an immersion $\phi\in \mathcal M^{4,\gamma}_\lambda(S)$ is an  \emph{almost half-sphere of radius} $\lambda$ if it can be written as $\phi_{p,u}^\lambda$ and, for given $\theta>0$, put 
$$\mathcal S'(\lambda, \theta):=\set{\phi\in\mathcal M^{4,\gamma}_\lambda(\Sp^2_+)\ |\ \phi\text{ is an almost half-sphere of radius $\lambda$ with $\|u\|_{C^{4,\gamma}}<\theta$}}.$$
For small enough $\theta$ and $\lambda$ we prove in Appendix \ref{Barycenter} that there exists a nonlinear projection $C$ that maps $\phi\in\mathcal S'(\lambda,\theta)$ to a point $C[\phi]\in S$ that we refer to as its (Riemannian) barycenter. For the round half-sphere attached to $\R^2$ this definition coincides with the origin. Similarly an analogue projection $C$ for immersions $f_u:\Sp^2_+\rightarrow(\R^3,\tilde g)$ to $\R^2$ is constructed. These two projections satisfy 
\begin{equation}\label{barycenteridentity}
C[F^\lambda[p, f_u]]=F^\lambda[p, C[f_u,\tilde g^{p,\lambda}]].
\end{equation}
The concept of the Riemannian barycenter is originally due to Karcher \cite{karcher}. We use a slight variant of the local version introduced in \cite{AK}.
Finally, we prove in Appendix \ref{Barycenter} that for small enough $\lambda$ and $\theta_0$ each $\phi\in \mathcal S'(\lambda,\theta_0)$ may be \emph{parameterized over its barycenter}. That is, there exists a parameterization of the form $\phi=F^\lambda[C[\phi], f_u]$. The parameterization depends on the orthonormal frame chosen at $p$ but is unique once a frame is fixed. We define
$$\mathcal S(\lambda,\theta):=\set{\phi\in \mathcal S'(\lambda,\theta_0)\ |\ \text{$\phi$ is parameterized over its barycenter and $\|u\|_{C^{4,\gamma}}<\theta$}}.$$
If we need to make $\gamma$ visible in the notation we write $\mathcal S^{4,\gamma}(\lambda,\theta)$. Given a functional $F$ (e.g. the area $A$ or the Willmore operator $W$ etc.) mapping an immersion $f_u$ and a metric $\tilde g$ to e.g. $\R$ we write $F[u,\tilde g]:=F[f_u,\tilde g]$ and for $F$ invariant under reparameterizations $\nabla F[u,\tilde g]:=\nabla F[f_u,\tilde g]$

\paragraph{Parity}\ \\
We may write points $\omega\in \Sp^2_+$ as $(\vec\omega,\omega^3)$ with $\vec\omega\in \R^2$ and $\omega^3\geq 0$ and introduce the reflection $r:\Sp^2_+\rightarrow \Sp^2_+, (\vec\omega,\omega^3)\mapsto (-\vec\omega,\omega^3)$. A map $u:\Sp^2_+\rightarrow \R$ is said to be \emph{even} if $u\circ r=u$ and is called \emph{odd} if $u\circ r=-u$. It is readily checked that any function $u:\Sp^2_+\rightarrow \R$ possesses a unique decomposition $u=u^++u^-$ into an even part $u^+$ and an odd part $u^-$. For $\lambda$, $K>0$ we define 
$$\mathcal S^-(\lambda, K):=\set{\phi\in \mathcal S(\lambda, K\lambda)\ |\ \|u^-\|_{C^{4,\gamma}}\leq K\lambda^2}.$$

\subsection{The area preserving Willmore flow}\label{flow}
Given $\lambda>0$ and $\phi_0\in \mathcal M^{4,\gamma}_\lambda(S)$ we say that $\phi:[0,\infty)\rightarrow\mathcal M^{4,\gamma}_\lambda(S)$ solves the area preserving Willmore flow with initial value $\phi_0$ if
\begin{equation}\label{orgprob}
\left\{\hspace*{.3cm}\ \begin{aligned}
\langle\dot \phi(t),\nu\rangle&=-P_H^{\perp}\left[W(\phi(t))\right]\hspace{.5cm}&\text{for all $t>0$},\\
\frac{\partial \phi}{\partial\eta}&=N^S\circ \phi\hspace{.5cm}&\text{on $[0,\infty)\times\partial \Sp^2_+$},\\
\frac{\partial H}{\partial\eta}&+h^S(\nu,\nu)H=0\hspace{.5cm}&\text{on $[0,\infty)\times\partial \Sp^2_+$},\\
\phi(0)&=\phi_0.\\
\end{aligned}\right.\end{equation}
Here $\nu$ and $\eta$ denote the normal and conormal of $\phi$ respectively. Also, we  have introduced the projection operator $P_{H[\phi]}^\perp[W[\phi]]$, abbreviated by $P_H^\perp[W(\phi)]$, defined by
\begin{equation}\label{projectionop}
    P_H^{\perp}[W[\phi]]=W[\phi]-\frac{\int_{\Sp^2_+}W[\phi]H[\phi]d\mu_{\phi^*\delta}}{\int_{\Sp^2_+}H[\phi]^2d\mu_{\phi^*\delta}}H[\phi].
\end{equation}
This is precisely the structure appearing in (\ref{introevlaw}) with $\alpha$ chosen to ensure constant $A[\phi(t)]$. The first order boundary condition is a rephrasing of $\phi$ meeting $S$ orthogonal along $\partial \Sp^2_+$. The third order boundary condition is natural for two reasons. First it is automatically satisfied for critical points of $\mathcal W$ (see \cite{AK}). Second it is crucial to establish optimal dissipation of Willmore energy. Indeed, let $t\geq 0$ and consider the variation $\epsilon\mapsto \phi(t+\epsilon)$ of $\phi(t)$. We may decompose the variational vector field as $\vec V=\varphi\nu+D\phi\xi$ where $\varphi$ is  a scalar function $\xi$ is a tangential vector field. As $\phi(t+\epsilon)\in\mathcal M^{4,\gamma}_\lambda(S)$ for all $\epsilon>0$ this variation is \emph{admissible} in the sense described in \cite{AK} (see Equation (1.17) and the proceeding analysis) and must therefore satisfy 
\begin{equation}\label{akadmissible} 
(\phi^*\delta)(\xi,\eta)=0\hspace{.5cm}\text{and}\hspace{.5cm}\frac{\partial\varphi}{\partial\eta}+\varphi h^S(\nu,\nu)=0\hspace{.3cm} \text{on $\partial \Sp^2_+$}.
\end{equation}
Applying Theorem 1 from \cite{AK} now yields
\begin{align*} 
\frac d{d\epsilon}\bigg|_{\epsilon=0}\mathcal W[\phi(t+\epsilon)]&=\int_{\Sp^2_+}W[\phi(t)]\varphi d\mu_\phi+\frac12\int_{\partial \Sp^2_+}\left(\varphi\frac{\partial H}{\partial\eta}-\frac{\partial\varphi}{\partial\eta}H-\frac12 H^2 g(\xi,\eta)\right)dS_\phi\\
&\overset{\substack{(\ref{orgprob})\\ (\ref{akadmissible})}}=-\int_{\Sp^2_+}W[\phi(t)]]P_H^\perp[W[\phi(t)]]d\mu_\phi+\frac12\int_{\partial \Sp^2_+}\varphi\left(\frac{\partial H}{\partial\eta}+Hh^S(\nu,\nu)\right)dS_\phi\\
&\overset{(\ref{orgprob})}=-\int_{\Sp^2_+}|P^\perp_{H[\phi(t)]}(W[\phi(t)])|^2 d\mu_{\phi}.
\end{align*}
Given a suitable $\phi_0\in S(\lambda,\theta)$ we construct a solution $\phi:[0,\infty)\rightarrow \mathcal \mathcal S(\lambda,\theta)$ to (\ref{orgprob}). Once such a solution is constructed it is standard to construct $\phi_*(t)$ that solves the same boundary and initial conditions as well as the evolution equation
\begin{equation}\label{fulleqstar}
\dot\phi_*(t)=-P_H^\perp[W[\phi_*(t)]\nu.
\end{equation}
Indeed, once a solution $\phi$ to (\ref{orgprob}) is constructed one can follow the analysis that is e.g. presented in \cite{stahl} to deduce the existence of $\phi_*$ by considering a reparameterization $\phi_*(t,\omega):=\phi(t,\psi(t,\omega))$. Inserting this an an ansatz into the evolution equation $(\ref{fulleqstar})$ gives the following ordinary differential equation (see e.g. \cite{stahl}) for $\psi$:
\begin{equation}\label{ode}\left\{
\begin{aligned}
\frac{\partial \psi(t,\omega)}{\partial t}&=-\left(D_2\phi(t,\psi(t,\omega))\right)^{-1}\left(D_1\psi(t,\psi(t,\omega))\right)^T,\\
\psi(0,\omega)&=\omega.
\end{aligned}\right.
\end{equation}
Here $T$ denotes component of the velocity vector tangential to $\phi$.\\

\paragraph{The elliptic problem}\ \\
We also state the formulation of the corresponding elliptic problem studied by Alessandroni and Kuwert in \cite{AK} as it is needed for the formulation of Theorem \ref{theorem3}.
\begin{equation}\label{elliptic}
\left\{\begin{aligned}
P_H^{\perp}&\left[W(\phi_0)\right]=0\hspace{.5cm}&\text{in $\Sp^2_+$}\\
\frac{\partial \phi_0}{\partial\eta}&=N^S\circ \phi_0\hspace{.5cm}&\text{on $\partial \Sp^2_+$}\\
\frac{\partial H}{\partial\eta}&+h^S(\nu,\nu)H=0\hspace{.5cm}&\text{on $\partial \Sp^2_+$}\\
\end{aligned}\right.
\end{equation}

\subsection{Flow equation in a moving reference frame}
Suppose that we have a solution $\phi:[0,\infty)\rightarrow\mathcal S(\lambda,\theta)$ to (\ref{orgprob}) with initial barycenter $C[\phi(0)]=p_0\in S$ and have chosen an orthonormal frame $b_i$ in a neighbourhood of $p_0$. Putting $\xi(t):=C[\phi(t)]$ to be the Riemannian barycenter of $\phi(t)$, the expression  $b_i(\xi(t))$ is then well well defined for small times and we may write 
\begin{equation}\label{ansatz}
\phi(t)=F^\lambda[\xi(t), (1+u(t,\omega))\omega].
\end{equation}
Inserting this ansatz into (\ref{orgprob}) must then give equations for the graph function $u$ and the curve $\xi$. We abbreviate the time-dependent metric $\tilde g^{\xi(t),\lambda}$ by $\tilde g^{\xi,\lambda}$. Recall that we denote the inner normal of $f_u$ with respect to $\tilde g$ by $\tilde\nu$ and its conormal by $\tilde \eta$ to distinguish them from the corresponding quantities for $\phi$ which do not carry the tilde. 

\paragraph{The graph function equation}\ \\ 
Inserting (\ref{ansatz}) into the evolution equation from (\ref{orgprob}) gives
$$-P_{H[\phi(t)]}^\perp\left(W[\phi(t)]\right)=\langle\dot\phi(t),\nu\rangle=\langle D_1 F^\lambda[\xi(t), f_u(t)]\dot\xi(t),\nu\rangle+\langle D_2 F^\lambda[\xi(t), f_u(t)]\dot f_u(t), \nu\rangle.$$
We use that $W$ and $H$ are invariant under diffeomorphisms to get $W[\phi(t)]=W[\lambda f_u,F[\xi(t),\cdot]^*\delta]$, $H[\phi(t)]=H[\lambda f_u,F[\xi(t),\cdot]^*\delta]$ and substitute the scaling behaviour of $W$, $H$ to find 
\begin{equation}\label{kooks1}
-\frac1{\lambda^3}P_{H[f_u,\tilde g^{\xi,\lambda}]}^\perp W[f_u,\tilde g^{\xi,\lambda}]=\langle D_1 F^\lambda[\xi(t), f_u(t)]\dot\xi(t),\nu\rangle+\langle D_2 F^\lambda[\xi(t), f_u(t)]\dot f_u(t), \nu\rangle.
\end{equation}
Here we have introduced the $L^2(f_u^*\tilde g^{\xi,\lambda})$-orthogonal projection onto $H[f_u,\tilde g^{\xi,\lambda}]^\perp$ that is defined as in Equation (\ref{projectionop}). Next we rewrite the right hand side of this equation by exploiting the formula
$$
\tilde\nu=\lambda\left(D_2 F^\lambda[\xi(t), f_u(t)]\right)^{-1}\nu.
$$
The factor of $\lambda$ is due to the scaling in the definition of $\tilde g^{\xi,\lambda}$ in Equation (\ref{metricdefinition}). This formula implies the identities
\begin{align} 
&\langle D_2 F^\lambda[\xi(t), f_u(t)]\dot f_u(t), \nu\rangle=\lambda \tilde g^{\xi,\lambda}[\dot f_u, \tilde\nu]\label{kooks2}\\
&\langle D_1 F^\lambda[\xi(t), f_u(t)]\dot\xi(t),\nu\rangle=\lambda \tilde g^{\xi,\lambda}\left[\left(D_2 F^\lambda[\xi(t), f_u(t)]\right)^{-1} D_1 F^\lambda[\xi(t), f_u(t)]\dot\xi(t),\tilde\nu\right]\label{kooks3}.
\end{align}
Inserting Equations (\ref{kooks2}) and (\ref{kooks3}) into (\ref{kooks1}) we find
\begin{equation}\label{atomheartmother}
\begin{aligned}
 \tilde g^{\xi,\lambda}[\dot f_u, \tilde\nu]=-\frac1{\lambda^4}&P_{H[f_u,\tilde g^{\xi,\lambda}]}^\perp W[f_u,\tilde g^{\xi,\lambda}]\\
-&\tilde g^{\xi,\lambda}\left[\left(D_2 F^\lambda[\xi(t), f_u(t)]\right)^{-1} D_1 F^\lambda[\xi(t), f_u(t)]\dot\xi(t),\tilde\nu\right].
\end{aligned}
\end{equation}
The identities  $A[\phi(t)]=2\pi\lambda^2$ and $C[\phi(t)]=\xi(t)$ imply $A[f_u,\tilde g^{\xi,\lambda}]=2\pi$ and $C[f_u,\tilde g^{\xi,\lambda}]=0$ (see Equation (\ref{barycenteridentity})). We put $K[u,\tilde g^{\xi,\lambda}]:=\operatorname{span}_\R(\nabla A[f_u,\tilde g^{\xi,\lambda}],\nabla C^i[f_u,\tilde g^{\xi,\lambda}])$ where $i=1,2$. In Appendix \ref{Kspace} we examine this space in more detail. In particular we show that for small enough $\|u(t)\|_{C^{4,\gamma}}$ and $\lambda$ we may consider the $L^2(f_u^*\tilde g^{\xi,\lambda})$-projection operators $P_{K[u,\tilde g^{\xi,\lambda}]}$ and $P^\perp_{K[u,\tilde g^{\xi,\lambda}]}$ onto $K[u,\tilde g^{\xi,\lambda}]$ and its $L^2(f_u^*\tilde g^{\xi,\lambda})$-orthogonal complement $K[u,\tilde g^{\xi,\lambda}]^\perp$. Differentiating $A[f_u,\tilde g^{\xi,\lambda}]=2\pi$ and $C^i[f_u,\tilde g^{\xi,\lambda}]=0$ with respect to time gives
\begin{align*}
    0=&\langle H[f_u,\tilde g^{\xi,\lambda}], \tilde g^{\xi,\lambda}(\dot f_u,\tilde\nu)\rangle_{L^2(f_u^*\tilde g^{\xi,\lambda})}+D_2 A[f_u,\tilde g^{\xi,\lambda}]\dot{\tilde g}^{\xi,\lambda},\\
    0=&\langle \nabla C^i[f_u,\tilde g^{\xi,\lambda}], \tilde g^{\xi,\lambda}(\dot f_u,\tilde\nu)\rangle_{L^2(f_u^*\tilde g^{\xi,\lambda})}+D_2 C^i[f_u,\tilde g^{\xi,\lambda}]\dot{\tilde g}^{\xi,\lambda}.
\end{align*}
Putting $C^ 0:=A$ and $\psi_\mu:=\|\nabla C^\mu[f_u,\tilde g^{\xi,\lambda}]\|_{L^2(f_u^*\tilde g^{\xi,\lambda})}^{-1}\nabla C^\mu[f_u,\tilde g^{\xi,\lambda}]$ for $\mu=0,1,2$ we may rewrite these equations as 
$$\langle \psi_\mu[f_u,\tilde g^{\xi,\lambda}], \tilde g^{\xi,\lambda}(\dot f_u,\tilde\nu)\rangle_{L^2(f_u^*\tilde g^{\xi,\lambda})}=-\frac{D_2 C^\mu[f_u,\tilde g^{\xi,\lambda}]\dot{\tilde g}^{\xi,\lambda}}{\|\nabla C^\mu[f_u,\tilde g^{\xi,\lambda}]\|_{L^2(f_u^*\tilde g^{\xi,\lambda})}}\hspace{.5cm}\text{for }\mu=0,1,2.$$
Put $A_{\mu\nu}:=\langle\psi_\mu,\psi_\nu\rangle_{L^2(f_u^*\tilde g^{\xi,\lambda})}$. In Appendix \ref{Kspace} we argue that we may invert the matrix $A_{\mu\nu}$ if $\lambda$ and $\|u(t)\|_{C^{4,\gamma}}$ are small. We denote the inverse matrix by $A^{\mu\nu}$. Multiplying the last equation with $A^{\mu\nu}\psi_\nu$, summing over $\mu,\nu$ and using Equation (\ref{projectionKspacedefinition}) we get
\begin{equation}\label{atomheartmother3}
P_{K[u,\tilde g^{\xi,\lambda}]}(\tilde g^{\xi,\lambda}(\dot f_u,\tilde\nu))=-\sum_{\mu,\nu=0}^2A^{\mu\nu}[u,\tilde g^{\xi,\lambda}]\frac{D_2 C^\mu[f_u,\tilde g^{\xi,\lambda}]\dot{\tilde g}^{\xi,\lambda}}{\|\nabla C^\mu[f_u,\tilde g^{\xi,\lambda}]\|_{L^2(f_u^*\tilde g^{\xi,\lambda})}}\psi_\nu.
\end{equation}
Equation (\ref{atomheartmother3}) gives us the component of $\tilde g^{\xi,\lambda}(\dot f_u,\tilde\nu)$ inside the space $K[u,\tilde g^{\xi,\lambda}]$. The component inside $K[u,\tilde g^{\xi,\lambda}]^\perp$ can be derived by applying the projection $P_{K[u,\tilde g^{\xi,\lambda}]}^\perp$ to Equation (\ref{atomheartmother}). Combining these two components we find  
\begin{equation}\label{atomheartmother4}
\begin{aligned}
\tilde g^{\xi,\lambda}(\dot f_u,\tilde\nu)=&\tau[u,\tilde g^{\xi,\lambda}]-P_{K[u,\tilde g^{\xi,\lambda}]}^\perp\left[\frac1{\lambda^4} W[f_u,\tilde g^{\xi,\lambda}]\right]\\
&-P_{K[u,\tilde g^{\xi,\lambda}]}^\perp\left[\tilde g^{\xi,\lambda}\left[\left(D_2 F^\lambda[\xi(t), f_u(t)]\right)^{-1} D_1 F^\lambda[\xi(t), f_u(t)]\dot\xi(t),\tilde\nu\right]\right]
\end{aligned}
\end{equation}
where we have introduced 
\begin{equation}\label{atomheartmother5}
\tau[u,\tilde g^{\xi,\lambda}]:=-\sum_{\mu,\nu=0}^2A^{\mu\nu}[u,\tilde g^{\xi,\lambda}]\frac{D_2 C^\mu[u,\tilde g^{\xi,\lambda}]\dot{\tilde g}^{\xi,\lambda}}{\|\nabla C^\mu[u,\tilde g^{\xi,\lambda}]\|_{L^2(f_u^*\tilde g^{\xi,\lambda})}}\psi_\nu.
\end{equation}

\paragraph{The barycenter equation}\ \\
Next, we derive an equation for the barycenter curve $\xi(t):=C[\phi(t)]$. For that purpose we consider small $\epsilon>0$ and express $\xi(t+\epsilon)$ in the chart $f[\xi(t),\cdot]$ centered at $\xi(t)$. That is, 
\begin{align*}
    \xi(t+\epsilon)&=f[\xi(t),(\xi^1(t+\epsilon),\xi^2(t+\epsilon))]\\
    &=\xi(t)+\sum_{i=1}^2\xi^i(t+\epsilon)b_i(\xi(t))+\varphi[\xi(t),(\xi^1(t+\epsilon),\xi^2(t+\epsilon))]N^S(\xi(t)).
\end{align*}
Clearly $\xi^i(t)=0$. Taking the scalar product with $b_i(\xi(t))$ and differentiating with respect to $\epsilon$ at $\epsilon=0$ we learn 
\begin{equation}\label{iamsotired}
\dot\xi^i(t)=\langle\dot\xi(t),b_i(\xi(t))\rangle.
\end{equation}
Recall $\phi(t)=F^\lambda[\xi(t), f_{u(t)}]$ and write $\phi(t+\epsilon)=F^\lambda[\xi(t),f_{v(t,\epsilon)}]$. Then clearly $v(t,0)=u(t)$. We wish to express the equation $\xi(t+\epsilon)=C[\phi(t+\epsilon)]$ using the local diffeomorphism $F^\lambda[\xi(t),\cdot]$ centered at $\xi(t)$. To do this note $\xi(t+\epsilon)=f[\xi(t),(\xi^1(t+\epsilon),\xi^2(t+\epsilon))]=F^\lambda[\xi(t),\lambda^{-1}(\xi^1(t+\epsilon),\xi^2(t+\epsilon)),0]$. Using Identity (\ref{barycenteridentity}) we get
\begin{equation}\label{hyppinessisawarmgun}
\lambda^{-1}\xi^i(t+\epsilon)=C^i[f_{v(t,\epsilon)},\tilde g^{\xi,\lambda}].
\end{equation}
Differentiating the left hand side at $\epsilon=0$ gives $\lambda^{-1}\dot\xi^i(t)$. We now differentiate the right hand side. To do so we investigate $\partial_\epsilon v(t,\epsilon)$ at $\epsilon=0$. Following the derivation for the graph function from the previous paragraph we can use Equation (\ref{atomheartmother}) but must drop the last term as it derived from differentiating the time-dependent chart which we do not have here. This gives
\begin{equation}\label{marthamydear}
\tilde g^{\lambda, \xi(t)}\left(\frac{\partial f_v}{\partial\epsilon}\bigg|_{\epsilon=0},\tilde\nu\right)=-\frac1{\lambda^4} P_{H[f_{u(t)},\tilde g^{\xi,\lambda}]}^\perp\left(W[f_{u(t)},\tilde g^{\xi,\lambda}]\right).
\end{equation}
Here we have used that $v(t, 0)=u(t)$. Differentiating Equation (\ref{hyppinessisawarmgun}) with respect to $\epsilon$ at $\epsilon=0$ and substituting Equations (\ref{iamsotired}) and (\ref{marthamydear}) we get the evolution equation for $\xi$:
\begin{equation}\label{barycentereq}
\dot\xi^i(t)=\langle\dot\xi(t), b_i(\xi(t))\rangle=-\frac1{\lambda^3}\langle \nabla C^i[f_{u(t)},\tilde g^{\xi,\lambda}],P_{H[f_{u(t)},\tilde g^{\xi,\lambda}]}^\perp\left(W[f_{u(t)},\tilde g^{\xi,\lambda}]\right)\rangle_{L^2(f_u^*\tilde g^{\xi,\lambda})}
\end{equation}

\paragraph{Equivalence to the Flow}\label{equivproff}\ \\
We now prove that $\phi(t):=F^\lambda[\xi(t),f_{u(t)}]$ solves the evolution equation (\ref{orgprob}) if (\ref{atomheartmother4}) and (\ref{barycentereq}) are satisfied. Using similar manipulations as in the derivation of (\ref{atomheartmother}) we find
\begin{align} 
\langle\dot\phi(t),\nu\rangle&=\lambda\tilde g^{\xi,\lambda}[\dot f_u, \tilde\nu]+\lambda\underbrace{\tilde g^{\xi,\lambda}\left[\left(D_2 F^\lambda[\xi(t), f_u(t)]\right)^{-1} D_1 F^\lambda[\xi(t), f_u(t)]\dot\xi(t),\tilde\nu\right]}_{=:X}\nonumber\\
&\overset{(\ref{atomheartmother4})}=\lambda \tau[u,\tilde g^{\xi,\lambda}]-\frac1{\lambda^3}P_{K[u,\tilde g^{\xi,\lambda}]}^ \perp(W[f_u,\tilde g^{\xi,\lambda}])+\lambda P_{K[u,\tilde g^{\xi,\lambda}]}(X).\label{equivequation}
\end{align}
Next we rewrite the definition of $\tau$ in Equation (\ref{atomheartmother5}) by computing $D_2 C^\mu[u,\tilde g^{\xi,\lambda}]\dot{\tilde g}^{\xi,\lambda}.$ This computation is moved to Appendix \ref{metricderivative} as it is quite lengthy. There we show
\begin{equation}\label{tauequation}
\tau[u,\tilde g^{\xi,\lambda}]=-P_{K[u,\tilde g^{\xi,\lambda}]}(X)+\sum_{\nu=0}^2\sum_{i=1}^2\frac{A^{i\nu}[u,\tilde g^{\xi,\lambda}]\dot\xi^i}{\lambda\|\nabla C^i[f_u,\tilde g^{\xi,\lambda}]\|_{L^2(f_u^*\tilde g^{\xi,\lambda})}}\frac{\nabla C^\nu[f_u,\tilde g^{\xi,\lambda}]}{\|\nabla C^\nu[f_u,\tilde g^{\xi,\lambda}]\|_{L^2(f_u^*\tilde g^{\xi,\lambda})}}.
\end{equation}
Inserting this formula into (\ref{equivequation}) and substituting Equation (\ref{barycentereq}) for $\dot\xi^i(t)$  we find 
\begin{align*} 
\langle\dot\phi(t),\nu\rangle=&-\frac1{\lambda^3}P_{K[u,\tilde g^{\xi,\lambda}]}^\perp(W[f_u,\tilde g^{\xi,\lambda}])\\
&-\sum_{\nu=0}^2\sum_{i=1}^2\frac{A^{i\nu}[u,\tilde g^{\xi,\lambda}]\langle \nabla C^i[f_{u(t)},\tilde g^{\xi,\lambda}],P_{H[f_{u(t)}, \tilde g^{\xi,\lambda}]}^\perp W[f_{u(t)},\tilde g^{\xi,\lambda}]\rangle_{L^2(f_u^*\tilde g^{\xi,\lambda})} }{\lambda^3\|\nabla C^i[f_u,\tilde g^{\xi,\lambda}]\|_{L^2(f_u^*\tilde g^{\xi,\lambda})}}\psi_\nu[u,\tilde g^{\xi,\lambda}].
\end{align*}
We may as well take the summation index $i$ in the second sum to also run over $0$ as $C^0=A$ implies $\langle\nabla C^0[f_u,\tilde g^{\xi,\lambda}], P_{H[u,\tilde g^{\xi,\lambda}]}^\perp(\hdots)\rangle=0$. So
\begin{align*} 
\langle\dot\phi(t),\nu\rangle&=-\frac1{\lambda^3}P_{K[u,\tilde g^{\xi,\lambda}]} ^\perp(W[f_u,\tilde g^{\xi,\lambda}])-\frac1{\lambda^3}P_{K[u,\tilde g^{\xi,\lambda}]}(P_{H[u,\tilde g^{\xi,\lambda}]}^\perp W[f_u,\tilde g^{\xi,\lambda}])\\
&=-\frac1{\lambda^3}P_{H[f_u,\tilde g^{\xi,\lambda}]}^\perp(W[f_u,\tilde g^{\xi,\lambda}]).
\end{align*}
Here we have used $H[f_u,\tilde g^{\xi,\lambda}]=\nabla C^0[f_u,\tilde g^{\xi,\lambda}]\in K[u,\tilde g]$ and hence $P_K^\perp (H[f_u,\tilde g^{\xi,\lambda}])=0$ to argue $P_K^\perp=P_K^\perp\circ P_H^\perp$. 
Finally, we use the diffeomorphism invariance and scaling of $W$ and $H$ to deduce
$$\langle\dot\phi(t),\nu\rangle=-P_{H[\phi(t)]}^\perp(W[\phi(t)]).$$

\paragraph{Rewriting the system}\ \\
In principal equations (\ref{atomheartmother4}) and (\ref{barycentereq}) constitute the system we must examine. We may however rewrite it as follows: Consider a pair $(u,\xi)$ that solves (\ref{atomheartmother4}) and (\ref{barycentereq}). First we write $\dot\xi(t)=\sum_{i=1}^2\langle\dot\xi(t), b_i(\xi(t))\rangle b_i(\xi(t))$. Now insert the evolution equation for $\langle\dot\xi(t), b_i(\xi(t))\rangle$ and use this formula to eliminate $\dot\xi(t)$ from the last term in Equation (\ref{atomheartmother4}). Putting 
\begin{align*} 
\tilde J_i[u,\xi,\lambda]&:=\tilde g^{\xi,\lambda}\left[\left(D_2 F^\lambda[\xi(t), f_u(t)]\right)^{-1} D_1 F^\lambda[\xi(t), f_u(t)]b_i(\xi(t)),\tilde\nu\right]\\
I^ i[u,\tilde g^{\xi,\lambda}]&:=\langle W[u,\tilde g], P_{H[u,\tilde g^{\xi,\lambda}]}^\perp\nabla C^ i[u,\tilde g^{\xi,\lambda}]\rangle_{L^2(f_u^*\tilde g^{\xi,\lambda})}
\end{align*}
for $i=1,2$, we derive the following system (note the summation convention):
\begin{align*}
    \tilde g^{\xi,\lambda}(\dot f_u,\tilde\nu)=&\tau[u,\tilde g^{\xi,\lambda}]-P_{K[u,\tilde g^{\xi,\lambda}]}^\perp\left[\frac1{\lambda^4} W[u,\tilde g^{\xi,\lambda}]\right]+\frac1{\lambda^3}I^i[u,\tilde g^{\xi,\lambda}]P_{K[u,\tilde g^{\xi,\lambda}]}^\perp[\tilde J_i[u,\xi,\lambda]]\\
    \langle\dot\xi(t),b_i(\xi(t))\rangle=&-\frac1{\lambda^3}I^i[u,\tilde g^{\xi,\lambda}]
\end{align*}
Clearly any solution to this pair of equations also solves (\ref{atomheartmother4}) and (\ref{barycentereq}). Finally, we rescale the evolution equations by introducing $u'(t,\omega):=u(\lambda^4 t,\omega)$ and $\xi'(t):=\xi(\lambda^4 t)$. We do not write $u'$ and $\xi'$ however. From now on $u$ and $\xi$ will represent the rescaled quantities. They satisfy
\begin{equation}\label{evolutionlaws}
\left\{\begin{aligned}
    \tilde g^{\xi,\lambda}(\dot f_u,\tilde\nu)=&\tau[u,\tilde g^{\xi,\lambda}]-P_{K[u,\tilde g^{\xi,\lambda}]}^\perp\left[ W[u,\tilde g^{\xi,\lambda}]\right]+\lambda I^i[u,\tilde g^{\xi,\lambda}]P_{K[u,\tilde g^{\xi,\lambda}]}^\perp[\tilde J_i[u,\xi,\lambda]]\\
    \langle\dot\xi(t),b_i(\xi(t))\rangle=&-\lambda I^i[u,\tilde g^{\xi,\lambda}].
\end{aligned}\right.\end{equation}

\paragraph{Boundary conditions}\ \\
Next we rewrite the boundary conditions in terms of $u$ and $\xi$. As $\tilde g^{\xi,\lambda}=\lambda^{-2}F^\lambda[\xi(t),\cdot]^*\delta$ we have
\begin{equation}\label{boudaryconditionsequi}
\def\arraystretch{2.6}
\begin{array}{lcl}
\displaystyle\frac{\partial\phi}{\partial\eta}=N^S\circ \phi & \displaystyle\Leftrightarrow &  \displaystyle\frac{\partial f_u}{\partial\tilde\eta}=\tilde\nu_{\R^2}\circ f_u.\\
\displaystyle\frac{\partial H}{\partial \eta}+Hh^S(\nu,\nu)=0 & \displaystyle\Leftrightarrow &  \displaystyle\frac{\partial H}{\partial\tilde \eta}+H\tilde h^{\R^2}(\tilde\nu,\tilde\nu)=0.
\end{array}
\end{equation}
Here $\tilde\nu_{\R^2}$ denotes the normal of $\R^2\times \set0$ with respect to $\tilde g^{\xi,\lambda}$ that coincides with $e_3$ in the euclidean case. We now reformulate the first condition. Let $\tilde \tau$ denote (one of) the unit tangential vector fields along $\partial \Sp^2_+$. Then of course $Df_u\tilde \tau$, $Df_u\tilde \eta$ and $\tilde \nu$ are defined at each point on the boundary of $f_u$ and constitute a $\tilde g^{\xi,\lambda}$-orthogonal basis for $\R^3$. We now have 
$$Df_u\tilde \eta=\tilde\nu_{\R^2}\Leftrightarrow Df_u\tilde \eta\perp_{\tilde g^{\xi,\lambda}}\R^2\overset{(!)}\Leftrightarrow\tilde\nu\in\R^2\Leftrightarrow\tilde g^{\xi,\lambda}(\tilde\nu,\tilde\nu_{\R^2})=0.$$
In the step marked by $(!)$ we used $(Df\tilde \eta)^\perp=\operatorname{span}_\R\set{Df_u\tilde \tau,\tilde\nu}$.
We summarize these boundary conditions as $B[u,\tilde g]=0$ where $\tilde g=\tilde g^{\xi,\lambda}$ and 
$$B[u,\tilde g]:=\left(\tilde g(\tilde\nu,\tilde\nu_{\R^2}), \frac{\partial H}{\partial\tilde\eta}+H\tilde h^ {\R^2}(\tilde\nu,\tilde \nu)\right).$$

\paragraph{Initial values}\ \\
Finally, we discuss possible initial values $u_0$ and $p:=\xi(0)$ for the evolution equations (\ref{evolutionlaws}). We say that a pair $(u_0, p)\in C^{4,\gamma}(\Sp^2_+)\times S$ is \emph{admissible} if 
\begin{enumerate}[(\textrm{A}1)]
    \item it is compatible with the boundary conditions. That is, $B[u_0,\tilde g^{p,\lambda}]=0$,
    \item it satisfies $A[u_0,\tilde g^{p,\lambda}]=2\pi$ and $C[u_0,\tilde g^{p,\lambda}]=0$. 
\end{enumerate}
The first condition must clearly be true by the formulation of problem (\ref{orgprob}). The second condition is derived from the fact that we formulate (\ref{orgprob}) inside the set $\mathcal S(\lambda,\theta)$ which requires $A[\phi(t)]=2\pi\lambda^2$ and hence $A[u_0,\tilde g^{p,\lambda}]=2\pi$ as well as  $\xi(t)=C[\phi(t)]$ which is true only if $C[u(t),\tilde g^{\xi(t),\lambda}]=0$ (see Equation (\ref{barycenteridentity}) and Appendix \ref{Barycenter}).\\

If the point $p$ is clear from context we refer to a function $u_0\in C^{4,\gamma}(\Sp^2_+)$ as an admissible initial value if the pair $(u_0, p)$ is an admissible initial value.

\paragraph{Conclusion}\ \\
Let $(u_0,p)$ be an admissible initial value. Abbreviating $\tilde g:=\tilde g^{\xi,\lambda}$ we are looking for a solution to the following system of equations:
\begin{equation}\label{system}
\left\{\begin{aligned}
    &\tilde g(\dot f_u,\tilde\nu)=\tau[u,\tilde g]-P_{K[u,\tilde g]}^\perp\left[ W[u,\tilde g]\right]+\lambda I^i[u,\tilde g]P_{K[u,\tilde g]}^\perp[\tilde J_i[u,\xi,\lambda]],\\
    &\langle\dot\xi(t),b_i(\xi(t))\rangle=-\lambda I^i[u,\tilde g],\\
    &B[u,\tilde g]=0\\
    &u(0)=u_0\\
    &\xi(0)=p
\end{aligned}\right.
\end{equation}

\section{An abstract perturbation problem}
\setcounter{equation}0
Let $T>1$ and denote the inner conormal of $\Sp^2_+\subset(\R^3,\delta)$ by $\eta$. We define the spaces 
\begin{align*}
    X_T&:=\left\{w\in C^{4,1,\gamma}([0,T]\times \Sp^2_+)\ \bigg|\ \frac{\partial w}{\partial\eta}=\frac{\partial\Delta w}{\partial\eta}=0\right\},\\
    Y_T&:=\left\{\tilde w\in C^{4,1,\gamma}([0,T]\times \Sp^2_+)\ \bigg |\ 
    \begin{array}{c} 
     \forall t\geq 0:\int_{\Sp^2_+}\tilde w(t)d\mu_{\Sp^2}=0,\ \exists c_1\in C^{0,\frac\gamma4}([0,T],\R),\ c_0\in\R:\\
     \dot{\tilde w}+\Delta^2\tilde w=c_1,\ \Delta^2 \tilde w(0,\cdot)=c_0
     \end{array}
     \right\},
\end{align*}
as well as the following elliptic analogues:
\begin{align*}
    X_0&:=\left\{w_0\in C^{4,\gamma}( \Sp^2_+)\ \bigg|\ \frac{\partial w_0}{\partial\eta}=\frac{\partial\Delta w_0}{\partial\eta}=0\right\}\\
    Y_0&:=\left\{\tilde w_0\in C^{4,\gamma}(\Sp^2_+)\ \bigg|\ \int_{\Sp^2_+}\tilde w_0d\mu_{\Sp^2}=0,\ \exists c_0\in\R: \Delta^2 \tilde w_0=c_0\right\}
\end{align*}
\begin{lemma}[A direct decomposition of \text{$C^{4,1,\gamma}([0,T]\times \Sp^2_+)$  and $C^{4,\gamma}(\Sp^2_+)$}]\label{decomposition}\hfill
\begin{enumerate}
    \item $C^{4,\gamma}(\Sp^2_+)=X_0\oplus Y_0$. The linear projections $\pi_{X_0}:C^{4,\gamma}(\Sp^2_+)\rightarrow X_0$ and $\pi_{Y_0}:C^{4,\gamma}(\Sp^2_+)\rightarrow Y_0$ are continuous. 
    \item $C^{4,1,\gamma}([0,T]\times \Sp^2_+)=X_T\oplus Y_T$. The linear projections $\pi_{X_T}:C^{4,1,\gamma}([0,T]\times \Sp^2_+)\rightarrow X_T$ and $\pi_{Y_T}:C^{4,1,\gamma}([0,T]\times \Sp^2_+)\rightarrow Y_T$ are continuous. 
\end{enumerate}
\end{lemma}
\begin{proof}
We first prove the time-independent case. It is easy to check that $X_0\cap Y_0=\set 0$. Now let $u_0\in C^{4,\gamma}(\Sp^2_+)$. We consider the following problem:
\begin{equation}\label{tildew0prob}
\left\{\begin{array}{rcll}
\Delta^2\tilde w_0&=&-\avint_{\partial \Sp^2_+}\frac{\partial\Delta u_0}{\partial\eta}dS&\text{in $\Sp^2_+$},\\
\frac\partial{\partial\eta}\tilde w_0&=&\frac\partial{\partial\eta}u_0&\text{on $\partial \Sp^2_+$},\\
\frac\partial{\partial\eta}\Delta \tilde w_0&=&\frac\partial{\partial\eta}\Delta u_0&\text{on $\partial \Sp^2_+$},\\
\int_{\Sp^2_+}\tilde w_0d\mu_{\Sp^2}&=&0.
\end{array}\right.
\end{equation}
This problem has a unique solution $\tilde w_0\in C^{4,\gamma}(\Sp^2_+)$ (see e.g. \cite{adn1}). We set $w_0:=u_0-\tilde w_0$ and claim that this gives the required decomposition. Considering (\ref{tildew0prob}) it is clear that $\tilde w_0\in Y_0$ and it is readily shown that $w_0\in X_0$. All that is left to do is establishing the continuity of the projections. Schauder theory (e.g. \cite{adn1}) implies
\begin{equation}\label{preoj1}
    \|\tilde w_0\|_{C^{4,\gamma}}\leq C\|u_0\|_{C^{4,\gamma}}.
\end{equation}
Using the definition of $w_0$ and (\ref{preoj1}) the estimate $\|w_0\|_{C^{4,\gamma}}\leq C(\Sp^2_+,\gamma)\|u_0\|_{C^{4,\gamma}}$ follows.

Next, we show the time-dependent case. First suppose that $u\in X_T\cap Y_T$. It is easy to see that such $u$ solves
$$
\left\{\begin{array}{rcll}
\dot u+\Delta^2u&=&0&\text{in $[0,T]\times \Sp^2_+$},\\
\frac\partial{\partial\eta}u&=&0&\text{on $ [0,T]\times\partial  \Sp^2_+$},\\
\frac\partial{\partial\eta}\Delta u&=&0&\text{on $ [0,T]\times\partial  \Sp^2_+$},\\
u(0,\cdot)&=&0&\text{on }\set0\times \Sp^2_+
\end{array}\right.
$$
and is therefore $0$ by uniqueness of this problem. For $u\in C^{4,1,\gamma}([0,T]\times \Sp^2_+)$ set $u_0:=u(0,\cdot)\in C^{4,\gamma}(\Sp^2_+)$. By the time-independent case there exist unique $w_0\in X_0$ and $\tilde w_0\in Y_0$ such that $u_0=w_0+\tilde w_0$. Now consider the problem

$$
\left\{\begin{array}{rcll}
\dot{\tilde  w}+\Delta^2\tilde w&=&-\avint_{\partial \Sp^2_+}\frac{\partial\Delta u}{\partial\eta}dS&\text{in $[0,T]\times \Sp^2_+$},\\
\frac\partial{\partial\eta}\tilde w&=&\frac\partial{\partial\eta}u&\text{on $ [0,T]\times\partial  \Sp^2_+$},\\
\frac\partial{\partial\eta}\Delta \tilde w&=&\frac\partial{\partial\eta}\Delta u&\text{on $ [0,T]\times\partial  \Sp^2_+$},\\
\tilde w(0,\cdot)&=&\tilde w_0&\text{on }\set0\times \Sp^2_+.
\end{array}\right.
$$
This problem has a unique solution $\tilde w\in C^{4,1,\gamma}([0,T]\times \Sp^2_+)$ (see e.g. \cite{simon, eidelman}) as the necessary compatibility conditions are satisfied. Let $w:=u-\tilde w$. We claim that this provides the claimed decomposition. Fist we use parabolic Schauder theory to get
\begin{equation}\label{proj3}
\|\tilde w\|_{C^{4,1,\gamma}_T}\leq C(T)(\|u\|_{C^{4,1,\gamma}_T}+\|\tilde w_0\|_{C^{4,\gamma}})\overset{(\ref{preoj1})}\leq C(T)\|u\|_{C^{4,1,\gamma}_T}.
\end{equation}
Note that $\tilde w$ also satisfies
$$\int_{\Sp^2_+}\tilde w(0)d\mu_{\Sp^2}=0\hspace{.5cm}\text{and}\hspace{.5cm}\frac{d}{dt}\int_{\Sp^2_+} \tilde w(t)d\mu_{\Sp^2}=-2\pi \avint_{\partial \Sp^2_+}\frac{\partial\Delta u}{\partial\eta}dS-\int_{\Sp^2_+}\Delta^2 ud\mu_{\Sp^2}=0.$$
So $\tilde w\in Y_T$. It is easy to see that $w\in X_T$. Finally, we combine the definition of $w$ and Equation (\ref{proj3}) to get $\|w\|_{C^{4,1,\gamma}_T}\leq C(T)\|u\|_{C^{4,1,\gamma}_T}$. 
\end{proof}
For $l\geq 0$ and $T>1$ we introduce
\begin{align*} 
G_0^l&:= C^l(Z_2, M_3^ {\operatorname{sym}}(\R))\hspace{.5cm}&\text{and}&\hspace{.5cm}\|\cdot\|_{G_0^l}:=\|\cdot\|_{ C^l(Z_2, M_3^ {\operatorname{sym}}(\R))}\\
G_T^l&:= C^{1,\frac\gamma4}([0,T], C^l(Z_2, M_3^ {\operatorname{sym}}(\R)))\hspace{.5cm}&\text{and}&\hspace{.5cm}\|\cdot\|_{G_T^l}:=\|\cdot\|_{C^{1,\frac\gamma4}([0,T], C^l(Z_2, M_3^ {\operatorname{sym}}(\R)))}.
\end{align*}

Let $l\geq 7$. We consider the operator 
\begin{align*}
&B_0:C^{4,\gamma}(\Sp^2_+)\times G_0^l\rightarrow C^{3,\gamma}(\partial \Sp^2_+)\times C^{1,\gamma}(\partial \Sp^2_+)\\
&(u_0,\tilde g_0)\rightarrow \bigg(\tilde g_0(\tilde\nu,\tilde\nu_{\R^2}), \frac{\partial H}{\partial\tilde\eta}+H\tilde h^{\R^2}(\tilde\nu,\tilde\nu)\bigg).    
\end{align*}
Here $\tilde\nu$ and $\tilde\eta$ are defined with respect to the metric $\tilde g_0$. $B_0$ is well defined on the set $\|u_0\|_{C^{4,\gamma}}+\|\tilde g_0-\delta\|_{G_0^l}<\epsilon$ and is of class $C^{l-7}$. 
Next we define the analogue time-dependent operator. For $T>1$ let
\begin{align*}
    &B_T:C^{4,1,\gamma}([0,T]\times \Sp^2_+)\times G_T^l\rightarrow C^{3,0,\gamma}([0,T]\times \partial \Sp^2_+)\times C^{1,0,\gamma}([0,T]\times \partial \Sp^2_+)\\
&B_T[u,\tilde g]:=\bigg(\tilde g(\tilde\nu,\tilde\nu_{\R^2}), \frac{\partial H}{\partial\tilde\eta}+H\tilde h^{\R^2}(\tilde\nu,\tilde\nu)\bigg).
\end{align*}
Now  $\tilde\nu$ and $\tilde\eta$ are defined with respect to the metric $\tilde g$. $B_T$ is well defined on the set $\|u\|_{C^{4,1,\gamma}_T}+\|\tilde g-\delta\|_{G_T^l}<\epsilon$ and is of class $C^{l-7}$. 

\begin{lemma}[Boundary value Lemma]\label{boundaryvalueslem}
Let $l\geq 8$. 
\begin{enumerate}
    \item There exist $\sigma_0,\theta_0,\tilde\theta_0>0$ and a $C^{l-7}$-map $\psi_0:X_0(\theta_0)\times G_0^l(\delta;\sigma_0)\rightarrow Y_0(\tilde\theta_0)$ such that within the respective neighbourhoods
    $$B_0[w_0+\tilde w_0,\tilde g_0]=0\hspace{.2cm}\Leftrightarrow\hspace{.2cm}\tilde w_0=\psi_0[w_0,\tilde g_0].$$
    The map $\psi_0$ satisfies $D_1\psi_0[0,\delta]=0$.
    \item There exist $\sigma(T),\theta(T),\tilde\theta(T)>0$ and a $C^{l-7}$-map $\psi_T:X_T(\theta(T))\times G_T^l(\delta;\sigma(T))\rightarrow Y_T(\tilde\theta(T))$ such that within the respective neighbourhoods
    $$B_T[w+\tilde w,\tilde g]=0\hspace{.2cm}\Leftrightarrow\hspace{.2cm}\tilde w=\psi_T[w,\tilde g].$$
    The map $\psi_T$ satisfies $D_1\psi_T[0,\delta]=0$.
\end{enumerate}
\end{lemma}
\begin{proof}
We first prove the time-independent case. $B_0$ is of class $C^{l-7}$ and in Lemma \ref{boundarylin} we show 
$$D_1B_0[0,\delta]\varphi=\left(\frac{\partial\varphi}{\partial\eta},-\frac{\partial}{\partial\eta}(\Delta+2)\varphi\right).$$
Note that $X_0=\operatorname{ker}(D_1B_0[0,\delta])$. On a neighbourhood of $(0,0,\delta)\in X_0\times Y_0\times G_0^ l$ we define $\bar B_0:X_0\times Y_0\times G^l_0\rightarrow  C^{3,\gamma}(\partial \Sp^2_+)\times C^{1,\gamma}(\partial \Sp^2_+)$ by mapping $\bar B_0[w_0,\tilde w_0,\tilde g_0]:=B_0[w_0+\tilde w_0,\tilde g_0]$. By construction $D_2\bar B_0[0,0,\delta]$ is injective. It is also onto as for any $(\alpha_0,\beta_0)\in  C^{3,\gamma}(\partial \Sp^2_+)\times C^{1,\gamma}(\partial \Sp^2_+)$ we can obtain $\tilde w_0\in Y_0$ satisfying $D_2\bar B_0[0,0,\delta]=(\alpha_0,\beta_0)$ by solving
$$\left\{\begin{array}{rcll}
\Delta^2 \varphi_0&=&\avint_{\partial \Sp^2_+}\beta_0+2\alpha_0&\text{in }\Sp^2_+,\\
\frac{\partial \varphi_0}{\partial\eta}&=&\alpha_0 &\text{on }\partial \Sp^2_+,\\
\frac{\partial \Delta \varphi_0}{\partial\eta}&=&-\beta_0-2\alpha_0 &\text{on }\partial \Sp^2_+,\\
\int \varphi_0d\mu_{\Sp^2}&=&0.& 
\end{array}\right.$$
This problem has a unique solution $\varphi_0\in C^{4,\gamma}(\Sp^2_+)$ (see e.g. \cite{adn1}) which clearly lies in $Y_0$. The Schauder estimate 
\begin{equation}\label{boundary1}
\|\varphi_0\|_{C^{4,\gamma}}\leq C(\|\alpha_0\|_{C^{3,\gamma}}+\|\beta_0\|_{C^{1,\gamma}})
\end{equation}
guarantees $D_2 \bar B_0[0,0,\delta]$ to have a bounded inverse. The implicit function theorem yields the neighbourhoods as well as the $C^{l-7}$ diffeomorphism. Finally, we prove the formula. For that let $w_0\in X_0$ and compute 
$$0=\frac d{dt} B_0[tw_0+ \psi_0[tw_0,\delta],\delta]=D_1 B_0[0,\delta]w_0+D_1B_0[0,\delta]D_1\psi_0[0,\delta]w_0.$$
Note $w_0\in X_0=\operatorname{ker}D_1 B_0[0,\delta]$. Hence the first term vanishes and $D_1\psi_0[0,\delta]w_0 \in \operatorname{ker}D_1 B_0[0,\delta]=X_0$. So $D_1\psi_0[0,\delta]w_0\in X_0\cap Y_0=\set0$ for all $w_0\in X_0$ and therefore $D_1\psi_0[0,\delta]=0$.

\ \\
For the time-dependent case we essentially repeat the same argument. First, on a neighbourhood of $(0,0,\delta)\in X_T\times Y_T\times G_T^ l$, we define 
$\bar B_T:X_T\times Y_T\times G_T^l\rightarrow C^{3,0,\gamma}([0,T]\times \Sp^2_+)\times C^{1,0,\gamma}([0,T]\times \Sp^2_+)$
by mapping $\bar B_T[w,\tilde w,\tilde g]:=B_T[w+\tilde w,\tilde g]$ and compute
$$D_2\bar  B_T[0,0,\delta]\varphi=\left(\frac{\partial\varphi}{\partial\eta},-\frac{\partial(\Delta+2)\varphi}{\partial\eta}\right).$$
$D_2\bar  B_0[0,0,\delta]$ is injective as any element $\varphi\in\ker D_2\bar  B_T[0,0,\delta] $ is in $X_T\cap Y_T$ and thus $0$. Next $D_2\bar  B_T[0,0,\delta]$ is onto. For $(\alpha,\beta)\in C^{3,0,\gamma}([0,T]\times \partial \Sp^2_+)\times C^{1,0,\gamma}([0,T]\times \partial \Sp^2_+)$ let $\alpha_0:=\alpha(0,\cdot)$ and $\beta_0:=\beta(0,\cdot)$. Then we have seen in the time-independent case that there exists a unique $\varphi_0\in Y_0$ such that $D_2\bar B_0[0,0,\delta]\varphi_0=(\alpha_0,\beta_0)$. Now consider the problem
$$\left\{\begin{array}{rcll}
\dot \varphi+\Delta^2 \varphi&=&\avint_{\partial \Sp^2_+}\beta+2\alpha&\text{in }[0,T]\times \Sp^2_+,\\
\frac{\partial \varphi}{\partial\eta}&=&\alpha &\text{on }[0,T]\times\partial \Sp^2_+,\\
\frac{\partial \Delta \varphi}{\partial\eta}&=&-\beta-2\alpha &\text{on }[0,T]\times\partial \Sp^2_+,\\\
\varphi(\cdot,0)&=&\varphi_0&\text{on }\set0\times \Sp^2_+.
\end{array}\right.$$
By construction of $\varphi_0$ the necessary compatibility conditions are satisfied and we deduce that there exists a unique function $\varphi\in C^{4,1,\gamma}([0,T]\times \Sp^2_+)$ that is a solution of this problem (see e.g. \cite{simon,eidelman}). Proving that this solution lies in $Y_T$ is done as in the proof of Lemma \ref{decomposition}. Due to Schauder estimates and Estimate (\ref{boundary1}) we get
$$\|\varphi\|_{C^{4,1,\gamma}_T}\leq C(T)\left(\|\varphi_0\|_{C^{4,\gamma}}+\|\alpha\|_{C^{3,0,\gamma}_T}+\|\beta\|_{C^{1,0,\gamma}_T}\right)\overset{(\ref{boundary1}) }\leq C(T)\left(\|\alpha\|_{C^{3,0,\gamma}_T}+\|\beta\|_{C^{1,0,\gamma}_T}\right)$$
and hence learn that the inverse map is continuous. 
The implicit function theorem guarantees the existence of the neighbourhoods and of the function $\psi_T$. The formula follows as in the time-independent case.
\end{proof}
\noindent
\textbf{Remarks}\hfill
\begin{enumerate}
    \item It is easy to check that $\psi_T[w,\tilde g](0)=\psi_0[w(0),\tilde g(0)]$.

\item $\psi_0,\psi_T\in C^2$ for $l\geq 9$. Using $D_1\psi_0[0,\delta]=0$ and $D_1\psi_T[0,\delta]=0$ it is then readily shown that after potentially shrinking the neighbourhoods in Lemma \ref{boundaryvalueslem} the following hold: Given $w_0\in X_0(\theta_0)$ and $\tilde g_0\in G_0^ l(\sigma_0)$ or $w\in X_T(\theta(T))$ and $\tilde g\in G_T^ l(\sigma(T))$ we have
\begin{align*}
    \|\psi_0[w_0,\tilde g_0]\|_{C^{4,\gamma}}&\leq C(\|\tilde g_0-\delta\|_{G_0^ l}+\|w_0\|_{C^{4,\gamma}}^2)\\
    \|\psi_T[w,\tilde g]\|_{C^{4,1,\gamma}_T}&\leq C(T)\left(\|\tilde g-\delta\|_{G_T^ l}+\|w\|_{C^{4,1,\gamma}_T}^2\right)
\end{align*}
\end{enumerate}

We now wish to study a prototype of the flow equations (\ref{system}). For $\tilde g\in G_T^ l$ and $u\in C^{4,1,\gamma}([0,T]\times \Sp^2_+)$ we put 
\begin{equation}\label{taudef}
\tau[u,\tilde g ]:=-\sum_{\mu,\nu=0}^2A^ {\mu\nu}[u,\tilde g ]\frac{D_2 C^\mu[u,\tilde g ]\dot{\tilde g}}{\|\nabla C^\mu[u,\tilde g ]\|_{L^2(f_u^*\tilde g )}}\frac{\nabla C^\nu[u,\tilde g ]}{\|\nabla C^\nu[u,\tilde g ]\|_{L^2(f_u^*\tilde g )}}
\end{equation}
and define the operator
\begin{equation}\label{Qoperatordefinition}
\begin{aligned} 
&Q_T:C^{4,1,\gamma}([0,T]\times \Sp^2_+)\times C^{4,\gamma}(\Sp^2_+) \times G_T^l\times C^{0,0,\gamma}([0,T]\times \Sp^2_+)\rightarrow C^{0,0,\gamma}([0,T]\times \Sp^2_+)\\
&Q_T[u,u_0,\tilde g,f]:=\tilde g(\dot f_u,\tilde \nu[u,\tilde g])-\tau[u,\tilde g]+P_{K[u,\tilde g]}^\perp(W[u,\tilde g]-f).
\end{aligned}
\end{equation}
For $l\geq 8$ the operator $Q_T$ is well defined on a small neighbourhood of $(0,0,\delta,0)$ and is of class $C^{l-8}$. We also introduce the operator $\bar Q_T:X_T\times X_0\times G_T^ l\times C^{0,0,\gamma}([0,T]\times \Sp^2_+)\rightarrow C^{0,0,\gamma}([0,T]\times \Sp^2_+)\times X_0$ that maps 
$$\bar Q_T[w,w_0,\tilde g, f]:=\left(Q_T[w+\psi_T[w,\tilde g], w_0+\psi_0[w_0,\tilde g(0)], \tilde g,f], w(0,\cdot)-w_0\right).$$
By the standard rules for composition of operators we see that $\bar Q_T$ is also of class $C^{l-8}$ for $l\geq 8$.

\begin{theorem}\label{firstestimate}
Let $T>1$ and suppose $l\geq 9$. There exists $\theta(T)>0$, $\theta'(T)>0$, $\sigma(T)>0$ and $\delta(T)>0$ and a $C^{l-8}$-map $w:X_0(\theta(T))\times G_T^l(\delta;\sigma(T))\times C^{0,0,\gamma}_T(\delta(T))\rightarrow X_T(\theta'(T))$ such that within the respective neighbourhoods
$$\bar Q_T[w,w_0,\tilde g, f]=(0,0)\hspace{.2cm}\Leftrightarrow\hspace{.2cm} w=w[w_0,\tilde g, f].$$
If additionally $l\geq 10$, $A[w_0+\psi_0[w_0,\tilde g(0)],\tilde g(0)]=2\pi$ and $C[w_0+\psi_0[w_0,\tilde g(0)],\tilde g(0)]=0$ the following estimates hold:
\begin{align*} 
&\|w\|_{C^{4,1,\gamma}_T}\leq C(T)\bigg[\|w_0\|_{C^{4,\gamma}}+\|\tilde g-\delta\|_{G_T^l}+\|P_{K[0,\delta]}^\perp f\|_{C^{0,0,\gamma}_T}+\|f\|_{C^{0,0,\gamma}_T}^2\bigg],\\
&\|w(T)\|_{C^{4,\gamma}}\leq Ce^{-6 T}\|w_0\|_{C^{4,\gamma}}+C(T)\bigg[\|\tilde g-\delta\|_{G_T^l}+\|P_{K[0,\delta]}^\perp f\|_{C^{0,0,\gamma}_T}+\|f\|_{C^{0,0,\gamma}_T}^2\bigg].
\end{align*}
\end{theorem}
\begin{proof}
For $l\geq 9$ both $Q_T$ and $\bar Q_T$ are of class $C^{l-8}\subset C^1$. Using $D_1\psi_T[0,\delta]=0$, $D_1\psi_0[0,\delta]=0$ and
\begin{equation}\label{Qderivative}
D_1 Q_T[0,0,\delta,0]\varphi=-\left(\dot\varphi+\frac12\Delta(\Delta+2)\varphi\right)
\end{equation}
we learn that $D_1\bar Q_T[0,0,\delta,0]$ is an isomorphism with bounded inverse. Indeed  $\bar Q_T[0,0,\delta,0]\varphi=(\psi,\varphi_0)$ for $\varphi\in X_T$ if and only if 
$$\left\{\begin{aligned}
\dot\varphi+\frac12\Delta(\Delta+2)\varphi&=-\psi\hspace{.2cm}\text{on $[0,T]\times \Sp^2_+$}\\
\frac{\partial\varphi}{\partial\eta}&=0\hspace{.2cm}\text{on $[0,T]\times \partial \Sp^2_+$}\\
\frac{\partial\Delta\varphi}{\partial\eta}&=0\hspace{.2cm}\text{on $[0,T]\times \partial \Sp^2_+$}\\
\varphi(0,\cdot)&=\varphi_0\hspace{.2cm}\text{on $\set 0\times \Sp^2_+$}.
\end{aligned}\right.$$
Schauder Theory (see e.g. \cite{simon,eidelman}) implies the existence and uniqueness of a solution $\varphi\in C^{4,1,\gamma}([0,T]\times \Sp^2_+)$ that must then be an element in $X_T$. Additionally the Schauder estimate 
$$\|\varphi\|_{C^{4,1,\gamma}_T}\leq C(T)\left(\|\varphi_0\|_{C^{4,\gamma}}+\|\psi\|_{C^{0,0,\gamma}_T}\right)$$
implies that $D_1\bar Q_T[0,0,\delta,0]$ is an isomorphism with bounded inverse. Hence the first part of the theorem follows. To prove the estimates we first note that after potentially shrinking all neighbourhoods in the Theorem the implicit function theorem implies the estimate 
\begin{equation}\label{wnaiveestimate}
    \|w\|_{C^{4,1,\gamma}_T}\leq C(T,\gamma)(\|\tilde g-\delta\|_{G_T^l}+\|f\|_{C^{0,0,\gamma}_T}+\|w_0\|_{C^{4,\gamma}}).
\end{equation}
Expanding $\bar Q_T[w,w_0,\tilde g, f]$ around $(0,0,\delta,0)$ yields the equation 
\begin{equation}\label{wproblem}
\left\{
\begin{aligned} 
&\dot w+\frac 12\Delta(\Delta+2) w=\sum_{\mu=0}^2\frac{D_2 C^\mu[0,\delta]\dot{\tilde g}}{\|\nabla C^\mu[0,\delta]\|_{L^2(\Sp^2_+)}^2}\nabla C^\mu[0,\delta]\\
&\hspace{4cm}+P_{K[0,\delta]}^\perp \left[D_2 W[0,\delta](\tilde g-\delta)-f\right]+R,\\
&\frac{\partial w}{\partial\eta}=\frac{\partial\Delta w}{\partial\eta}=0,\\
&w(0)=w_0.
\end{aligned}
\right.
\end{equation}
where the remainder $R$ satisfies the estimate $\|R\|_{C^{0,0,\gamma}_T}\leq C(T)\left(\|w\|_{C^{4,1,\gamma}_T}^2+\|\tilde g-\delta\|_{G_T^l}^2+\|f\|_{C^{0,0,\gamma}_T}^2\right)$ as due to $l\geq 10$ we know $\bar Q_T\in C^2$. Inserting Estimate (\ref{wnaiveestimate}) for $\|w\|_{C^{4,1,\gamma}_T}$ yields $\|R\|_{C^{0,0,\gamma}_T}\leq C(T)\left(\|w_0\|_{C^{4,1,\gamma}_T}^2+\|\tilde g-\delta\|_{G_T^l}^2+\|f\|_{C^{0,0,\gamma}_T}^2\right)$. Applying the Schauder Theory from e.g. \cite{simon,eidelman} to problem (\ref{wproblem}) and inserting this estimate for $R$ we immediately find the estimate 
\begin{equation}\label{wbetterstimate}
    \|w\|_{C^{4,1,\gamma}_T}\leq C(T,\gamma)(\|\tilde g-\delta\|_{G_T^l}+\|P_{K[0,\delta]}^\perp  f\|_{C^{0,0,\gamma}_T}+\|w_0\|_{C^{4,\gamma}}+\|f\|_{C^{0,0,\gamma}_T}^2).
\end{equation}
This is the first estimate claimed in the Theorem. To deduce the decay estimate we decompose $w=w^1+w^2$ where $w^1(t)\in\operatorname{span}_\R\set{1,\omega^1,\omega^2}$ for all $t\in[0,T]$ and $w^2(t)\perp_{L^2(\Sp^2_+)}\operatorname{span}_\R\set{1,\omega^1,\omega^2}$. Splitting the remainder $R$ and the initial value $w_0$ similarly into $R^1+R^2$ and $w^1_0+w^2_0$ we now examine separate problems for $w^1$ and $w^2$. $w^1$ satisfies
$$\left\{\begin{aligned}
&\dot w^1+\frac12\Delta(\Delta+2)w^1=\sum_{\mu=0}^2\frac{D_2 C^\mu[0,\delta]\dot{\tilde g}}{\|\nabla C^\mu[0,\delta]\|_{L^2(\Sp^2_+)}^2}\nabla C^\mu[0,\delta]+R^1,\\
&\frac{\partial w^1}{\partial\eta}=\frac{\partial \Delta w^1}{\partial\eta}=0,\\
&w^1(0)=w^1_0.
\end{aligned}\right.$$
We apply the Schauder Theory from e.g. \cite{simon,eidelman} to derive an estimate for $w^1$. Inserting the estimate for the remainder from above we find  
\begin{equation}\label{w1betterstimate}
    \|w^1\|_{C^{4,1,\gamma}_T}\leq C(T,\gamma)(\|\tilde g-\delta\|_{G_T^l}+\|w_0^1\|_{C^{4,\gamma}}+\|w_0\|_{C^{4,\gamma}}^2+\|f\|_{C^{0,0,\gamma}_T}^2).
\end{equation}
Next we establish an estimate for $w_0^1$ by expanding the equation $A[w_0+\psi_0[w_0,\tilde g_0],\tilde g_0]=2\pi=A[0,\delta]$. Using $A\in C^{l-6}\subset C^2$, $\psi_0\in C^{l-7}\subset C^2$, $\nabla A[0,\delta]=-2$ and $D_1\psi_0[0,\delta]=0$ we deduce 
\begin{equation}\label{woestimate}
2\left|\int_{\Sp^2_+}w_0d\mu_{\Sp^2}\right|=\left|\langle \nabla A[0,\delta],w_0\rangle_{L^2(\Sp^2_+)}\right|\leq C(\|w_0\|_{C^{4,\gamma}}^2+\|\tilde g-\delta\|_{G_T^ l}).
\end{equation}
By applying the same reasoning to the barycenter $C^ i$ we may derive a similar bound for $\int_{\Sp^2_+}w_0 \omega^ id\mu_{\Sp^2}$ as $\nabla C^i[0,\delta]=-\frac3{2\pi}\omega^i$. Hence $\|w^1_0\|_{C^{4,\gamma}}\leq C(\|w_0\|_{C^{4,\gamma}}^2+\|\tilde g-\delta\|_{G_T^ l})$. Combining this estimate with (\ref{w1betterstimate}) gives 
\begin{equation}\label{w1betterstimate2}
    \|w^1\|_{C^{4,1,\gamma}_T}\leq C(T)(\|\tilde g-\delta\|_{G_T^l}+\|P_{K[0,\delta]}f\|_{C^{0,0,\gamma}_T}+\|w_0\|_{C^{4,\gamma}}^2+\|f\|_{C^{0,0,\gamma}_T}^2).
\end{equation}
Next, we examine the problem that is solved by $w^2$.
$$\left\{\begin{aligned}
&\dot w^2+\frac12\Delta(\Delta+2)w^2=P_{K[0,\delta]}^\perp \left[D_2 W[0,\delta](\tilde g-\delta)-f\right]+R^2,\\
    &\frac{\partial w^2}{\partial\eta}=\frac{\partial \Delta w^2}{\partial\eta}=0,\\
&w^2(0)=w^2_0.
\end{aligned}\right.$$
As $w^2(t)\perp_{L^2(\Sp^2_+)}\operatorname{span}_\R\set{1,\omega^1,\omega^2}$ for all $t\geq 0$ we may use Appendix \ref{schauder} to derive the improved Schauder estimate.
\begin{align} 
\hspace{-.5cm}\|w^2(T)\|_{C^{4,\gamma}}&\leq Ce^{-6 T}\|w_0^2\|_{C^{4,\gamma}}+C(T)\bigg[\|\tilde g-\delta\|_{G_T^ l}+\|P_{K[0,\delta]}^\perp f\|_{C^{0,0,\gamma}_T}+\|R^2\|_{C^{0,0,\gamma}_T}\bigg]\nonumber\\
&\hspace{-1cm}\leq Ce^{-6 T}\|w_0^2\|_{C^{4,\gamma}}+C(T)\bigg[\|\tilde g-\delta\|_{G_T^    l}+\|P_{K[0,\delta]}^\perp f\|_{C^{0,0,\gamma}_T}+\|w_0\|_{C^{4,\gamma}}^2+\|f\|_{C^{0,0,\gamma}_T}^2\bigg]\label{w2betterstimate2}.
\end{align}
In the second step, we inserted the estimate for $R$ from above. 
The decay estimate follows by combining (\ref{w1betterstimate2}), (\ref{w2betterstimate2}), the estimate for $\|w^1_0\|_{C^{4,\gamma}}$ from above and shrinking the neighbourhood from which $w_0$ is taken so that $C(T)\|w_0\|_{C^{4,\gamma}}\leq Ce^{-6 T}$.  
\end{proof}

\section{Willmore flow with prescribed barycenter curve}\label{section4}
\setcounter{equation}0
We now apply the result from the previous section to derive the existence of a solution to the Willmore flow with \emph{prescribed barycenter curve}. That is, given a curve $\xi:[0,T]\rightarrow S$ with initial condition $\xi(0)=p$ for some $p\in S$ we investigate the following system that only contains the evolution equation for the graph function form (\ref{system}) with $\tilde J_i$ written out explicitly.
\begin{equation}\label{prescxisystem}
\left\{\begin{aligned}
    \tilde g(\dot f_u,\tilde \nu)=&\tau[u,\tilde g]-P_{K[u,\tilde g]}^\perp[W[u,\tilde g]]\\
    &+P_{K[u,\tilde g]}^\perp\bigg[I^i[u,\tilde g]\tilde g\bigg(\lambda(D_2 F^\lambda[\xi( t),f_{u}(t)])^{-1}D_1 F^\lambda[\xi( t),f_{u}(t)]b_i(\xi(t)),\tilde\nu\bigg)\bigg],\\
    B_T[u,\tilde g]&=0,\\
    u(0)&=u_0.
\end{aligned}\right.
\end{equation}
Ultimately we want to put $\tilde g:=\tilde g^{\xi,\lambda}$ but will keep $\tilde g$ abstract for now. To parameterize the curves $\xi$ near $p$ we use the chosen orthonormal frame $(b_1,b_2)$ in a neighbourhood of $p$. Given a vector field $\vec\beta:[0, T]\rightarrow D_{r_1}$ with $r_1$ as in Subsection \ref{terminilogy} we can consider the curve
$$\xi_{p,\vec\beta}:[0,T]\rightarrow S,\ \xi_{p,\vec\beta}(t):=f[p,\vec\beta(t)]=p+\beta^ i(t)b_i(p)+\varphi[p,\vec\beta(t)] N^S(p).$$
It is easy to verify that this accounts for all curves close to $p$ for $t\in[0,T]$. For $\Omega\in C^n$ and $n\geq 4$ the map $ C^{1,\frac\gamma4}([0, T], D_{r_1})\ni \vec\beta\rightarrow \xi_{p,\vec\beta}\in C^{1,\frac\gamma4}([0, T], S)$ is of class $C^{n-4}$. 
On a small neighbourhood of $(0,0,0)$, we may define the map
\begin{align*} 
J_i:&(-1,1)\times C^{4,1,\gamma}([0,T]\times \Sp^2_+)\times C^{1,\frac\gamma4}([0,T],\R^2)\rightarrow C^{0,0,\gamma}([0,T]\times \Sp^2_+),\\
&(\lambda, u, \vec\beta)\mapsto \lambda\bigg(D_2 F^\lambda[\xi_{p,\vec\beta},f_u]\bigg)^{-1}D_1 F^\lambda[\xi_{p,\vec\beta},f_u]b_i(\xi_{p,\vec\beta}).
\end{align*}
Let $\Gamma^{k}_{\ ij}:=\langle b_k, \nabla _{b_i}b_j\rangle$. Then we may follow the analysis from \cite{AK} (Proof of Theorem 2) where the following explicit coordinate expression for $J_i$ is derived:
\begin{equation}\label{formula}
\lambda \bigg(D_2 F^\lambda[a,x,z]\bigg)^{-1}D_1 F^\lambda[a, x,z]b_i(a)=e_i+
\begin{bmatrix}
(\lambda x^ j\Gamma^ k_{\ ij}(a)-h_{ik}^S(a)(\varphi[a, \lambda x]+\lambda z) )e_k\\
\lambda zh_{ik}^S(a)\partial_k\varphi[a,\lambda x]
\end{bmatrix}
\end{equation}
Using this formula it is readily checked that $J_i$ is of class $C^{n-4}$. Formula (\ref{formula}) also implies that for $\|\vec\beta\|_{C^1}\leq 1$ and $\|u\|_{C^{4,1,\gamma}_T}\leq 1$ we have the estimate
\begin{equation}\label{Jeclose}
\|J_i[\lambda, u,\vec\beta]-e_i\|_{C^{0,0,\gamma}_T}\leq C(\Omega,\gamma, T)|\lambda|.
\end{equation}

Comparing the evolution equation (\ref{prescxisystem}) to the abstract perturbation problem we studied in Theorem \ref{firstestimate} we see 
\begin{equation}\label{thisisthechoiceforf}
I^ i[u,\tilde g]P^\perp_{K[u,\tilde g]}\left[\tilde g\left(\lambda \bigg(D_2 F^\lambda[\xi,f_u]\bigg)^{-1}D_1 F^\lambda[\xi,f_u]b_i(\xi),\tilde\nu\right)\right]\hspace{.1cm}\leftrightarrow\hspace{.1cm}f.
\end{equation}

Suppose that $l\geq 9$ and $n\geq l$ are integers. With these choices $n-4\geq l-8\geq 1$. On a small neighbourhood of $(0,0,\delta,0,0)$ we introduce the operator
\begin{align*} 
&\hat Q_T:X_T\times X_0\times G_T^l\times (-1,1)\times C^{1,\frac\gamma4}([0,T],\R^2)\rightarrow C^{0,0,\gamma}([0,T]\times \Sp^2_+)\\
&\hat Q_T[w,w_0,\tilde g, \lambda,\vec\beta]:=\bar Q_T[w,w_0, \tilde g,0]+\left(P_{K[u,\tilde g]}^\perp\bigg[I^i[u,\tilde g] \tilde g(J_i[\lambda, u,\vec\beta],\tilde\nu[u,\tilde g])\bigg]\bigg|_{u=w+\psi_T[w,\tilde g]},0\right).
\end{align*}
Recall that we use the reference point $p\in S$ to define the curve $\xi_{p,\vec\beta}$. We already know that $\bar Q_T$ is of class $C^{l-8}$. It is also not difficult to check that $I^i\in C^{l-8}$, $\tilde\nu\in C^{l-8}$ and $J^i\in C^{n-4}\subset C^{l-8}$ as maps into $C^{0,0,\gamma}([0,T]\times\Sp^2_+)$. Hence $\hat Q_T\in C^{l-8}\subset C^1$. As with the abstract perturbation problem we now prove that there exist a unique solution $w=w[w_0,\tilde g, \lambda,\vec\beta]$ for small enough data.  

\begin{theorem}[Short time existence with prescribed barycenter curve]\label{prescirbedcurve1}\hfill\\
Let $T>1$, $n\geq l$ and $l\geq 9$. There exist
$\theta(T)>0$, $\theta'(T)>0$, $\sigma(T)>0$, $\rho(T)>0$, $\lambda_0(T)>0$
and a $C^{l-8}$-map $w:X_0(\theta(T))\times G_T^l(\delta;\sigma(T))\times[-\lambda_0(T),\lambda_0(T)]\times C^{1,\frac\gamma4}_T(\rho(T))\rightarrow X_T(\theta'(T))$ such that within the respective neighbourhoods
$$ \hat Q_T[w,w_0,\tilde g, \lambda,\vec\beta]=(0,0)\hspace{.2cm}\Leftrightarrow\hspace{.2cm} w=w[w_0,\tilde g,\lambda,\vec\beta].$$
For $l\geq 10$ the solution $w$ satisfies
\begin{align*} 
&\|w\|_{C^{4,1,\gamma}_T}\leq C(T)\left[\|w_0\|_{C^{4,\gamma}}+\|\tilde g-\delta\|_{G_T^l}+\lambda^2\right]\\
&\|w(T)\|_{C^{4,\gamma}}\leq Ce^{-6T}\|w_0\|_{C^{4,\gamma}}+C(T)\left[\|\tilde g-\delta\|_{G_T^l}+\lambda^2\right].
\end{align*}
\end{theorem}
\begin{proof}
First we prove existence. As deduced above $\hat Q_T\in C^1$. Next note that $\bar Q_T[0,0,\delta,0]=(0,0)$ and $I^i[u,\delta]=0$ and thus $\hat Q_T[0,0,\delta,0,0]=(0,0)$. 
Next we compute the differential with respect to the first argument. To do this we note that $I^i[0,\delta]=0$ and that $J_i[0,0,0]=e_i$. Hence we learn that the map 
\begin{equation}\label{concretefdef}
\begin{aligned}
&\mathcal I:X_T\times X_0\times G_T^l\times (-1,1)\times C^{1,\frac\gamma4}([0,T],\R^2)\rightarrow C^{0,0,\gamma}([0,T]\times \Sp^2_+)\\
&\mathcal I[w,w_0,\tilde g,\lambda,\vec\beta]:=P_{K[u,\tilde g]}^\perp\bigg[I^i[u,\tilde g] \tilde g(J_i[\lambda, u,\vec\beta],\tilde\nu[u,\tilde g])\bigg]\bigg|_{u=w+\psi_T[w,\tilde g]}
\end{aligned}
\end{equation}
satisfies
$$D_1 \mathcal I[0,0, \delta,0,0]=P_{K[0,\delta]}^\perp(-\omega^i)D_1 I^i[0,\delta]=0.$$
Since $D_1\psi_T[0,\delta]=0$ we get $D_1\hat Q_T[0,0,\delta,0,0]=D_1\bar Q_T[0,0,\delta,0]$ which is an isomorphism as we have already discussed in Theorem \ref{firstestimate}. The implicit function theorem gives the the existence and uniqueness as well as the estimate
\begin{equation}\label{naivew1}
\|w\|_{C^{4,1,\gamma}_T}\leq C(T)\left[\|w_0\|_{C^{4,\gamma}}+|\lambda|+\|\vec\beta\|_{C^{1,\frac\gamma4}_T}+\|\tilde g-\delta\|_{G_T^l}\right].
\end{equation}
We put $u:=w+\psi_T[w,\tilde g]$. After potentially shrinking the neighbourhoods from where $\lambda,\vec\beta,w_0$ and $\tilde g$ are taken we can use the second remark after Lemma \ref{boundaryvalueslem} to get
\begin{align}
\|u\|_{C^{4,1,\gamma}_T}&\leq 2\|w\|_{C^{4,1,\gamma}_T}+C(T)\|\tilde g-\delta\|_{G_T^l}\label{naiveu1}\\
&\overset{(\ref{naivew1})}\leq C(T)(\|w_0\|_{C^{4,\gamma}}+\|\tilde g-\delta\|_{G_T^l}+|\lambda|+\|\vec\beta\|_{C^{1,\frac\gamma4}_T}) \label{naiveu2}.
\end{align}

We now prove the estimates claimed in the Theorem. $w$ satisfies the equation $\bar Q_T[w,w_0,\tilde g,f]=(0,0)$ with $f$ as in (\ref{thisisthechoiceforf}).
Thus Theorem \ref{firstestimate} can be used to obtain better estimates. To apply this theorem there are two terms that we need to estimate:
$$\left\|P_{K[0,\delta]}^\perp\bigg[I^i[u,\tilde g] \tilde g(J_i[\lambda, u,\vec\beta],\tilde\nu[u,\tilde g])\bigg]\right\|_{C^{0,0,\gamma}_T}
\hspace{.5cm}\text{and}\hspace{.5cm}
\left\|I^i[u,\tilde g] \tilde g(J_i[\lambda, u,\vec\beta],\tilde\nu[u,\tilde g])\right\|_{C^{0,0,\gamma}_T}^2.$$

\ \\
\noindent
\textbf{First term}\ \\
First, we may take the term $I^i[u,\tilde g]$ out of the projection operator, as it only depends on time. Note that $I^i$ is of class $C^{l-8}\subset C^1$ for $l\geq 9$. Using Estimate (\ref{naiveu2}) for $u$ to see that $u$ is small we get 
\begin{equation}\label{eqeq1}
\|I^i[u,\tilde g]\|_{C^{0,\frac\gamma4}_T}\leq C(T)\left(\|u\|_{C^{4,1,\gamma}_T}+\|\tilde g-\delta\|_{G_T^l}\right).
\end{equation}
Next we use the Equation (\ref{formula}) and can write $J_i=e_i+R_i$ where we know from Equation (\ref{Jeclose}) that for small enough $\vec\beta$ we get $\|R_i\|_{C^{0,0,\gamma}_T}\leq C(T)|\lambda|.$ Clearly
\begin{equation}\label{eqeq2}
\tilde g(J_i[\lambda,u,\vec\beta],\tilde\nu[u,\tilde g])=\tilde g(e_i, \tilde\nu[u,\tilde g])+\tilde g(R_i[\lambda,\xi_{p,\vec\beta},u], \tilde\nu[u,\tilde g]).
\end{equation}

We deal with both terms individually:
\begin{enumerate}
        \item For $u=0$ and $\tilde g=\delta$ we have $\tilde g(e_i,\tilde\nu[u,\tilde g])=-\omega^ i$. As the map $\tilde\nu$ is of class $C^{l-6}\subset C^1$ we get 
    \begin{equation}\label{eqeq3}
    \|\tilde g(e_i, \tilde\nu[u,\tilde g])+\omega^ i\|_{C^ {0,0,\gamma}_T}\leq C(T)(\|u\|_{C^{4,1,\gamma}_T}+\|\tilde g-\delta\|_{G_T^l}).
    \end{equation}
    
    \item For the second term we use the estimate we have for $R_i$ and learn that 
        \begin{equation}\label{eqeq5}
        \|\tilde g(R_i[\lambda,\xi_{p,\vec\beta},u], \tilde\nu[u,\tilde g])\|_{C^{0,0,\gamma}_T}\leq C(T)|\lambda|\|\tilde\nu\|_{C^{0,0,\gamma}_T}\|\tilde g\|_{G_T^l}.
        \end{equation}
        The last two factors are not small but bounded.
\end{enumerate}
 Combining Estimates (\ref{eqeq1}), (\ref{eqeq2}), (\ref{eqeq3}) and  (\ref{eqeq5}) and using $P_{K[0,\delta]}^\perp(\omega^ i)=0$ we learn 
\begin{align}
    \left\|P_{K[0,\delta]}^\perp\bigg[I^i[u,\tilde g] \tilde g(J_i[\lambda, u,\vec\beta],\tilde\nu[u,\tilde g])\bigg]\right\|_{C^{0,0,\gamma}_T}&\leq C(T)\left(\|u\|_{C^{4,1,\gamma}_T}^2+\|\tilde g-\delta\|_{G_T^l}^2+\lambda^2\right)\nonumber\\
    &\overset{(\ref{naiveu1}) }\leq  C(T)\left(\|w\|_{C^{4,1,\gamma}_T}^2+\|\tilde g-\delta\|_{G_T^l}^2+\lambda^2\right)\label{diegl1}.
\end{align}

\pagebreak
\ \\
\noindent
\textbf{Second term}\ \\
To estimate $I^i$ we can reuse Equation (\ref{eqeq1}). Next we note that $\|\tilde\nu[u,\tilde g]\|_{C^{0,0,\gamma}_T}$ and $\|\tilde g\|_{G_T^ l}$ are both bounded. Finally, we can bound $J^i$ by using the formula (\ref{formula}) and Estimate (\ref{Jeclose}). So
\begin{align} 
\left\|\bigg[I^i[u,\tilde g] \tilde g(J_i[\lambda, u,\vec\beta],\tilde\nu[u,\tilde g])\bigg]\right\|_{C^{0,0,\gamma}_T}^2&\leq C(T)(\|u\|_{C^{4,1,\gamma}_T}^2+\|\tilde g-\delta\|_{G_T^l}^2+\lambda^2)\nonumber\\
&\overset{(\ref{naiveu1})}\leq C(T)(\|w\|_{C^{4,1,\gamma}_T}^2+\|\tilde g-\delta\|_{G_T^l}^2+\lambda^2)\label{diegl2}.
\end{align} 

\ \\
\noindent
\textbf{Combining both estimates}\ \\
We apply Theorem \ref{firstestimate} and insert the estimates (\ref{diegl1}) and (\ref{diegl2}) to get
$$\|w\|_{C^{4,1,\gamma}_T}\leq C(T)\left[\|w_0\|_{C^{4,\gamma}}+\|\tilde g-\delta\|_{G_T^l}+\lambda^2+\|w\|_{C^{4,1,\gamma}_T}^2\right].$$
Using Estimate (\ref{naivew1}) we see that after potentially shrinking the neighbourhoods we may assume $C(T)\|w\|_{C^{4,1,\gamma}_T}\leq\frac12$ and can thus absorb $\|w\|_{C^{4,1,\gamma}_T}^2$ to the left hand side. This gives 
\begin{equation}\label{dfghjdfghjfghj}
\|w\|_{C^{4,1,\gamma}_T}\leq C(T)\left[\|w_0\|_{C^{4,\gamma}}+\|\tilde g-\delta\|_{G_T^l}+\lambda^2\right],
\end{equation}
which is the first estimate claimed in the Theorem. For the second part we use the Estimate for $\|w(T)\|_{C^{4,\gamma}}$ from Theorem \ref{firstestimate} and insert Estimates  (\ref{diegl1}) and (\ref{diegl2}) to obtain 
\begin{align*}
    \|w(T)\|_{C^{4,\gamma}}\leq & Ce^{-6T}\|w_0\|_{C^{4,\gamma}}+C(T)\left[\|\tilde g-\delta\|_{G_T^l}+\lambda^2+\|w\|_{C^{4,1,\gamma}_T}^2\right]\\
\overset{(\ref{dfghjdfghjfghj})}\leq &Ce^{-6T}\|w_0\|_{C^{4,\gamma}}+C(T)\left[\|\tilde g-\delta\|_{G_T^l}+\lambda^2\right]+C(T)\|w_0\|_{C^{4,\gamma}}^2.
\end{align*}
After potentially shrinking the neighbourhood to which $w_0$ belongs we may assume that $\|w_0\|_{C^{4,\gamma}}C(T)\leq e^{-6T}$ and have therefore established the decay estimate for $w(T)$.
\end{proof}

In Theorem \ref{prescirbedcurve1} we deal with an arbitrary metric. In the problem we are investigating we are, however, interested into the particular choice of metric that is $\tilde g^ {\xi_{p,\vec\beta},\lambda}$ with $\xi_{p,\vec\beta}$ being defined as above. As we will now investigate $\tilde g$ as a function depending on $\lambda$ and $\vec\beta$ we will write $\tilde g_p[\lambda,\vec\beta]:=\tilde g^ {\xi_{p,\vec\beta},\lambda}$.\\

\begin{korollar}[Flow with prescribed barycenter curve]\label{prescirbedcurve3}\hfill\\
Let $p\in S$, $T>1, r\geq 2$ and suppose that $\Omega$ is of class $C^{13+2r}$. There exists $\theta(T)>0$, $\theta'(T)>0$, $\rho(T)>0$, $\lambda_0(T)>0$ and a $C^r$-map 
$u_p:X_0(\theta(T))\times[-\lambda_0(T),
\lambda_0(T)]\times C^{1,\frac\gamma4}_T(\rho(T))\rightarrow C^{4,1,\gamma}_T(\theta'(T))$
such that $u_p[w_0,\lambda,\vec\beta]$ is the unique solution to (\ref{prescxisystem}) with prescribed barycenter curve $\xi=\xi_{p,\vec\beta}$, metric $\tilde g=\tilde g_p[\lambda,\vec\beta]$ and initial value $u_0:=w_0+\psi_0[w_0,\tilde g^{p,\lambda}]$. $u$ satisfies the estimates
\begin{align*}
&\|u_p\|_{C^{4,1,\gamma}_T}\leq C(T)\left[\|u_0\|_{C^{4,\gamma}}+|\lambda|\right]\hspace{.3cm}\text{and}\hspace{.3cm}\|u_p(T)\|_{C^{4,\gamma}}\leq Ce^{-6T}\|u_0\|_{C^{4,\gamma}}+C(T)|\lambda|.
\end{align*}
\end{korollar}
\begin{proof}
The explicit formula for $\tilde g^{\xi,\lambda}$ in Equation (\ref{metricformula}) as well as the regularity of $\varphi$ combined with the regularity of the map $\vec\beta\mapsto \xi_{p,\vec\beta}$ discussed above readily imply that 
$$\tilde g_p:(-1,1)\times C^{1,\frac\gamma4}([0,T],\R^2)\rightarrow G_T^{8+r},\ (\lambda,\vec\beta)\mapsto \tilde g_p[\lambda,\vec\beta]$$
is well defined on a small neighbourhood of $(0,0)$ and of class $C^r$. As long as $\|\vec\beta\|_{C^{1,\frac\gamma4}_T}\leq 1$ it is easy to derive the estimate  
\begin{equation}\label{metricestimate}
\|\tilde g_p[\lambda,\vec\beta]\|_{G_T^{8+r}}\leq C(S,T)|\lambda|.
\end{equation}
Put $l=8+r$. Due to (\ref{metricestimate}) we may choose $\lambda_0(T)$ and $\rho(T)$ small and then use $\tilde g_p[\lambda,\vec\beta]$ in Theorem \ref{prescirbedcurve1}. This allows us to define
$$\bar w_p[w_0,\lambda,\vec\beta]:=w[w_0,\tilde g_p[\lambda,\beta],\lambda, \vec\beta]$$
with $w$ as in Theorem \ref{prescirbedcurve1}. $\bar w$ is well defined on a neighbourhood of $(0,0,0)$ and of class $C^{r}$ as $\tilde g_p[\cdot,\cdot]\in C^r$ and $w$ is of class $C^{l-8}=C^r$. On a small neighbourhood of $(0,0,0)$ we now define a map $u_p:X_0\times (-1,1)\times C^{1,\frac\gamma4}([0,T], \R^2)\rightarrow C^{4,1,\gamma}([0,T]\times \Sp^2_+)$ by
$$u_p[w_0,\lambda,\vec\beta]:=\bar w_p[w_0,\lambda,\vec\beta]+\psi_T[\bar w_p[w_0,\lambda,\vec\beta],\tilde g_p[\lambda,\vec\beta]].$$
$u_p\in C^r$ as $\psi_T\in C^{l-8}=C^r$ by Lemma \ref{boundaryvalueslem}. Using Lemma \ref{boundaryvalueslem} and Theorem \ref{prescirbedcurve1} we see that $u_p$ has the claimed properties. 
\end{proof}

\section{Proof of Theorem 1}\label{proofontheorem1}
\setcounter{equation}0
We consider the space 
$$C^{1,\frac\gamma4}_0([0,T],\R^2):=\set{\vec\beta\in C^{1,\frac\gamma4}([0,T],\R^2)\ |\ \beta(0)=0}.$$
Given $p\in S$ and small $\lambda>0$, $w_0\in X_0$ and $\vec\beta\in C^{1,\frac\gamma4}_0([0, T], \R^2)$ we get a function $u_p[w_0,\lambda,\vec\beta]$. The pair $(u_p[w_0,\lambda,\vec\beta],\xi_{p,\vec\beta})$ is a solution to (\ref{system}) with initial values $\xi(0)=p$ and $u_p(0)=w_0+\psi_0[w_0,\tilde g^{p,\lambda}]$ if and only if 
\begin{equation}\label{xiconditionode}
\langle b_i(\xi_{p,\vec\beta}(t)),\dot\xi_{p,\vec\beta}(t)\rangle=-\lambda I^i[u_p[w_0,\lambda,\vec\beta],\tilde g_p[\lambda,\vec\beta]].
\end{equation}
To find such $\vec\beta$ we again employ the implicit function theorem. For that we must first define a suitable operator. On a small neighbourhood for $(0,0,0)$ we put
\begin{align*} 
&\mathcal T_p:C^{1,\frac\gamma4}_0([0,T],\R^2)\times X_0\times (-1,1)\rightarrow (C^{0,\frac\gamma4}([0,T],\R))^2\\
&(\vec\beta,w_0,\lambda)\mapsto \left(\langle b_i\circ\xi_{p,\vec\beta}, \dot\xi_{p,\vec\beta}\rangle+\lambda I^i[u_p[w_0,\lambda,\vec\beta],\tilde g_p[\lambda,\vec\beta]]\right)_{i=1,2}.
\end{align*}
Clearly (\ref{xiconditionode}) is equivalent to finding zeros of $\mathcal T_p$. 

\begin{lemma}[Short time existence of the barycenter curve]\label{shorttime}\ \\
Let $p\in S$, $r\geq 2$ and suppose $\Omega\in C^{13+2r}$. For all $T>1$ there exist $\theta(T)>0$, $\rho(T)>0$, $\lambda(T)>0$ and a $C^r$-map $\vec\beta_p:[-\lambda_0(T),\lambda_0(T)]\times X_0(\theta(T))\rightarrow C^{1,\frac\gamma4}_T(\rho(T))$ such that within the respective neighbourhoods
$$
\mathcal T_p[\vec\beta, w_0,\lambda]\hspace{.2cm}\Leftrightarrow\hspace{.2cm}\vec\beta=\vec\beta_p[\lambda, w_0].
$$
$\vec \beta_p$ satisfies the estimate $\|\vec\beta_p\|_{C^{1,\frac\gamma4}_T}\leq C(T)(|\lambda|+\|w_0\|_{C^{4,\gamma}})$.
\end{lemma}
\begin{proof}
Putting $l=8+r$ and $n=13+2r$ the operator $\mathcal T_p$ is of class $C^r$. To see this we note that $I^i$ and $u_p$ are of class $C^{r}$. We have already pointed out that $\vec\beta\mapsto\xi_{p,\vec\beta}$ is of class $C^{n-4}\subset C^{l-8}$. Finally, we have the following two remarks:
\begin{enumerate}
    \item $C^{1,\frac\gamma4}([0,T],\R^3)\ni\xi\mapsto\dot\xi\in C^{0,\frac\gamma4}([0,T],\R^3)$ is $C^\infty$ as it is a bounded linear map.
    \item As the vector field $b_i$ are $C^{n-1}$ we have that $C^{1,\frac\gamma4}([0,T],\R^3)\ni\xi\mapsto b_i\circ\xi\in C^{0,\frac\gamma4}([0,T],\R^3)$ is of class $C^{n-3}\subset C^r$.
\end{enumerate}\

Next we observe $\mathcal T_p[\vec\beta,0,0]=\left(\langle b_i\circ\xi_{p,\vec\beta}, \dot\xi_{p,\vec\beta}\rangle\right)_{i=1,2}$. In particular $\mathcal T_p[0,0,0]=0$. Thus $(0,0,0)$ is a zero of $\mathcal T_p$. It remains to study the Frechet differential 
$$D_1\mathcal T_p[0,0,0]\vec\beta=\frac d{d\epsilon}\bigg|_{\epsilon=0}\mathcal T_p[\epsilon\vec\beta,0,0]=\frac d{d\epsilon}\bigg|_{\epsilon=0}\left(\langle b_i\circ\xi_{p,\epsilon\vec\beta}, \dot\xi_{\epsilon\vec\beta}\rangle\right)_{i=1,2}$$
Clearly $\xi_{p,\vec 0}\equiv p$. Hence $\dot\xi_{p,\vec 0}\equiv 0$ and
\begin{align*} 
D_1 \mathcal T_p[0,0,0]\vec\beta&=\left(\langle b_i(p),\frac{\partial}{\partial\epsilon}\bigg|_{\epsilon=0}\dot\xi_{\epsilon\vec\beta}\rangle\right)_{i=1,2}\\
&=\left(\langle b_i(p),\frac{\partial}{\partial\epsilon}\bigg|_{\epsilon=0}\left[\epsilon\dot\beta^j(t)b_j(p)+\epsilon D_2\varphi[p,\epsilon\vec\beta(t)]\dot{\vec\beta}(t)N^S(p)\right]
\rangle\right)_{i=1,2}\\
&=\dot{\vec\beta}.
\end{align*}
The last step used the fact that $(b_1(p), b_2(p))$ is an orthonormal basis of $T_pS$. By the definition of the space $C^{1,\frac\gamma4}_0([0,T], \R^2)$ we see that $D_1\mathcal T_p[0,0,0]$ is an isomorphism into $C^{0,\frac\gamma4}([0,T],\R^2)$ with bounded inverse 
$$\left(D_1\mathcal T_p[0,0,0]\right)^{-1}\vec v=\left(t\mapsto \int_0^ t\vec v(s)ds\right).$$
The lemma follows from the implicit function theorem.
\end{proof}
\begin{korollar}[Shorttime existence of the flow]\label{shorttime2}\hfill \\
Let $T>1$, $p\in S$, $r\geq 2$ and suppose that $\Omega\in C^{13+2r}$. There exist, $\theta(T)$, $\theta'(T)$, $\rho(T)>0$, and $\lambda_0(T)>0$ such that for each admissible initial value $(u_0,p)\in  C^{4,\gamma}(\Sp^2_+)\times S$ satisfying $\|u_0\|_{C^{4,\gamma}}\leq \theta(T)$ and $|\lambda|\leq\lambda_0(T)$ there exists exactly one pair $(u,\vec\beta)\in C^{4,1,\gamma}([0,T]\times \Sp^2_+)\times  C^{1,\frac\gamma4}_0([0,T],\R^2) $ that satisfies system (\ref{system}) on $[0,T]$ as well as the estimates $\|u\|_{C^{4,1,\gamma}_T}\leq \theta'(T)$ and $\|\vec\beta\|_{C^{1,\frac\gamma4}_T([0,T],\R^2)}\leq \rho(T)$.\\

The graph function $u$ satisfies the estimates
$$\|u\|_{C^{4,1\gamma}_T}\leq C(T)\left[\|u_0\|_{C^{4,\gamma}}+|\lambda| \right]\hspace{.3cm}\text{and}\hspace{.3cm}\|u(T)\|_{C^{4,\gamma}}\leq C_0e^{-6T}\|u_0\|_{C^{4,\gamma}}+C_1(T)|\lambda|.$$
The vector field $\vec\beta$ satisfies the estimate
$\|\vec\beta\|_{C^{1,\frac\gamma4}_T}\leq C(T)(|\lambda|+\|u_0\|_{C^{4,\gamma}})$.
\end{korollar}
\begin{proof}
Let $(u_0,p)$ be admissible. We may decompose $u_0=w_0+\tilde w_0$. Lemma \ref{decomposition} implies $\|w_0\|_{C^{4,\gamma}}+\|\tilde w_0\|_{C^{4,\gamma}}\leq C\|u_0\|_{C^{4,\gamma}}$. Hence, for small enough $\|u_0\|_{C^{4,\gamma}}$ we can ensure $\|w_0\|_{C^{4,\gamma}}$ and $\|\tilde w_0\|_{C^{4,\gamma}}$ to be small enough to apply Lemma \ref{boundaryvalueslem} which implies $\tilde w_0=\psi_0[w_0,\tilde g^{p,\lambda}]$.\\
Choosing $\theta(T)$ small enough we can ensure $\|w_0\|_{C^{4,\gamma}}$ to be small enough to apply Lemma \ref{shorttime} and derive the existence of $\vec\beta$ satisfying the estimate claimed in the Corollary.\\
For small enough $\lambda$ and $\|u_0\|_{C^{4,\gamma}}$ we may apply Corollary \ref{prescirbedcurve3} to derive the existence of $u$ as well as the claimed estimates.
\end{proof}

Using the Estimates in Corollary \ref{shorttime2}, it is now easy to prove the long-time existence of the flow. 

\begin{korollar}[Long-time existence of the flow]\label{longimeexistence}\hfill \\
There exist $\theta_0>0$ and $\lambda_0>0$ such that for all $\lambda\in [0,\lambda_0]$, $p\in S$ and admissible initial values $\|u_0\|_{C^{4,\gamma}}\leq \theta_0$ there exist:
\begin{enumerate}
    \item A curve $\xi\in C^{1,\frac\gamma4}([0,\infty), S)$ satisfying $\xi(0)=p$,
    \item two $C^{1,\frac\gamma4}$ curves $b_1, b_2:[0,\infty)\rightarrow TS$ such that $(b_1(t), b_2(t))$ is an orthonormal basis of $T_{\xi(t)} S$ for all $t\geq 0$,
    \item a graph function $u\in C^{4,1,\gamma}([0,\infty)\times \Sp^2_+)$ satisfying $u(0)=u_0$
\end{enumerate} 
such that system (\ref{system}) is satisfied by $(u,\xi)$ on $[0,\infty)$. The function $u$ satisfies the decay estimate 
$$\|u\|_{C^{4,1,\gamma}([\tau,\infty)\times \Sp^2_+)}\leq C(\Sp^2_+,\gamma)\bigg(e^{-\alpha\tau}\|u_0\|_{C^{4,\gamma}}+\lambda\bigg)$$
for constants $C>0$ and $\alpha>0$ independent of all choices involved and the barycenter curve $\xi$ satisfies $\|\dot\xi\|_{C^{0,\frac\gamma4}([0,\infty),\R^3)}\leq C(\lambda+\|u_0\|_{C^{4,\gamma}}).$
\end{korollar}
\begin{proof}
Keeping the notation from Corollary \ref{shorttime2}, choose a time $T>1$ so large that  $C_0e^{-6T}<\frac12$, put $\theta_0:=\theta(T)$, fix $\lambda_0'(T)>0$ such that $\lambda_0'(T)C_1(T)<\frac12\theta_0$ and put $\lambda_0:=\min(\lambda_0(T),\lambda_0'(T))$.\\

Let $p_1:=p$ and choose a frame $(b^{(1)}_1, b^{(1)}_2)$ around $p_1$. Corollary \ref{shorttime2} implies the existence of a curve $\xi_1\in C^{1,\frac\gamma4}([0,T], S)$ with $\xi_1(0)=p_1$ and a function $u_1\in C^{4,1,\gamma}([0,T]\times \Sp^2_+)$ solving system (\ref{system}) on the time interval $[0,T]$ and satisfying the estimates
$$\|u_1(T)\|_{C^{4,\gamma}}\leq \frac12\|u_0\|_{C^{4,\gamma}}+C_1(T)\lambda<\theta_0\hspace{.5cm}\text{and}\hspace{.5cm}\|\dot\xi\|_{C^{0,\frac\gamma4}_T}\leq C(T)(\lambda+\|u_0\|_{C^{4,\gamma}}).$$

Let $p_2:=\xi_1(T)$ and choose a frame $(b^{(2)}_1, b^{(2)}_2)$  around $p_2$ that agrees with $(b^{(1)}_1, b^{(1)}_2)$ on a small neighbourhood of $p_2$. We apply \ref{shorttime2} with the same choice for $\lambda$, the point $p_2$ and the initial value $u_1(T)$ to obtain a a curve $\xi_2\in C^{1,\frac\gamma4}([0,T], S)$ and a function $u_2\in C^{4,1,\gamma}([0,T]\times \Sp^2_+)$ solving system (\ref{system}) on the time interval $[0,T]$ and satisfying the estimates
$$\|u_2(T)\|_{C^{4,\gamma}}\leq \frac12\|u_1(T)\|_{C^{4,\gamma}}+C_1(T)\lambda\hspace{.5cm}\text{and}\hspace{.5cm}\|\dot\xi_2\|_{C^{0,\frac\gamma4}_T}\leq C(T)(\lambda+\|u_1(T)\|_{C^{4,\gamma}}).$$ 
Inserting the estimate for $\|u_1(T)\|_{C^{4,\gamma}}$ and using the definition of $\lambda_0$ yields
\begin{align*}
    &\|u_2(T)\|_{C^{4,\gamma}}\leq \frac14\|u_0\|_{C^{4,\gamma}}+\left(1+\frac12\right)C_1(T)\lambda <\frac14\theta_0+\frac32C_1(T)\lambda<\theta_0,\\
    &\|\dot\xi_2\|_{C^{0,\frac\gamma4}_T}\leq C(T)\left(\lambda+\frac12\|u_0\|_{C^{4,\gamma}}+C_1(T)\lambda\right)=\frac{C(T)}2\|u_0\|_{C^{4,\gamma}}+C(T)(1+C(T))\lambda.
\end{align*}

Inductively we now get the existence of points $p_n\in S$ curves $\xi_n\in C^{1,\frac\gamma4}([0,T], S)$ and functions $u_n\in C^{4,1,\gamma}([0,T]\times \Sp^2_+)$ solving system (\ref{system}) on the time interval $[0,T]$ with initial values $\xi_n(0)=p_n=:\xi_{n-1}(T)$ and $u_n(0)=u_{n-1}(T)$ that satisfy the estimate
\begin{equation}\label{u0estimate}
    \|u_n(T)\|_{C^{4,\gamma}}    \leq \frac1{2^n}\|u_0\|_{C^{4,\gamma}}+C_1(T)\lambda\sum_{k=0}^{n-1}\frac1{2^k}<\theta_0.
\end{equation}
In the last step we have used the choice $C_1(T)\lambda\leq\frac12\theta_0$. Applying Corollary \ref{shorttime2} as well as Estimate (\ref{u0estimate}) gives
\begin{align}
    \|u_n\|_{C^{4,1,\gamma}_T}&\leq C(T)\left[\|u_n(0)\|_{C^{4,\gamma}}+\lambda\right]\nonumber\\
    &=C(T)\left[\|u_{n-1}(T)\|_{C^{4,\gamma}}+\lambda\right]\nonumber\\
    &\overset{(\ref{u0estimate})}\leq C(T)\left[\frac1{2^{n-1}}\|u_0\|_{C^{4,\gamma}}+C_1(T)\lambda\sum_{j=0}^{n-2}\frac1{2^k}\right]\label{decay}\\
    &\leq  C(T)\left[\|u_0\|_{C^{4,\gamma}}+\lambda\right]\nonumber.
\end{align}
Finally, we get the following estimate for the barycenter curves:
\begin{align*} 
\|\dot\xi_n\|_{C^{0,\frac\gamma4}_T}\leq & C(T)(\lambda+\|u_{n-1}(T)\|_{C^{4,\gamma}}) \\
\leq &C(T)\left(\lambda+\frac1{2^ {n-1}}\|u_0\|_{C^{4,\gamma}}+C_1(T)\lambda\sum_{j=0}^ {n-2}\frac1{2^ j}\right)\\
\leq &  C(T)(\lambda+\|u_0\|_{C^{4,\gamma}})
\end{align*}

We can now \emph{`glue'} the solution on $[0,\infty)$ together by defining functions $\xi:[0,\infty)\rightarrow S$ and $u:[0,\infty)\times \Sp^2_+\rightarrow\R$ as follows:
\begin{equation}\label{uefinition}
\xi(t):=\left\{\begin{aligned}
\xi_1(t) &\hspace{.2cm} \text{if $t\in[0,T)$},\\
\xi_2(t) &\hspace{.2cm} \text{if $t\in[T,2T)$},\\
\xi_3(t) &\hspace{.2cm} \text{if $t\in[2T,3T)$},\\
\vdots & \hspace{.4cm}\vdots
\end{aligned}\right.
\hspace{.75cm}\text{and}\hspace{.75cm}
u(t,\omega):=\left\{\begin{aligned}
u_1(t,\omega) &\hspace{.2cm} \text{if $t\in[0,T)$},\\
u_2(t,\omega) &\hspace{.2cm} \text{if $t\in[T,2T)$},\\
u_3(t,\omega) &\hspace{.2cm} \text{if $t\in[2T,3T)$},\\
\vdots & \hspace{.4cm}\vdots
\end{aligned}\right.
\end{equation}
The fact that $u\in C^{4,1}([0,\infty)\times \Sp^2_+)$ and $\xi\in C^1([0,\infty),S)$ is a direct consequence from us choosing the $(n+1)$-th frame to agree with the $n$-th frame in a small neighbourhood of $p_{n+1}$. The fact that $u\in C^{4,1,\gamma}([0,\infty)\times \Sp^2_+)$ and $\xi\in C^{1,\frac\gamma4}([0,\infty),S)$ is easily derived by distinguishes the cases $|t_1-t_2|\geq 1$ and $|t_1-t_2|\leq 1$. This gives
\begin{align*} 
\|\dot\xi\|_{C^{0,\frac\gamma4}([0,\infty), \R^3)}&\leq C(T)(\|u_0\|_{C^{4,\gamma}}+\lambda)
\hspace{.5cm}\text{and}\hspace{.5cm}\|u\|_{C^{4,1,\gamma}([0,\infty)\times \Sp^2_+)}\leq C(T)\left(\|u_0\|_{C^{4,\gamma}}+\lambda\right).
\end{align*}
Finally, we prove the decay estimate claimed in the theorem. Let $m\geq n\in\N$ and use Estimate (\ref{decay}) to get 
$$\|u_m\|_{C^{4,1,\gamma}_T}\leq C(T)\left(\frac 2{2^m}\|u_0\|_{C^{4,\gamma}}+\lambda\right)\leq C(T)\left(\frac 2{2^n}\|u_0\|_{C^{4,\gamma}}+\lambda\right).$$
As this is true for all $m\geq n$ the definition of $u$ in Equation (\ref{uefinition}) implies
$$\|u\|_{C^{4,1,\gamma}([nT,\infty)\times \Sp^2_+)}\leq C2^{-n}\|u_0\|_{C^{4,\gamma}}+C\lambda.$$
Now let $\tau\geq 0$ and choose the unique $n\in\N_0$ such that $nT\leq \tau<nT+T$. Then $2^{-n}\leq 2^{1-T^{-1}\tau}$. Let $\alpha:=\frac{\ln2}T\in (0,1)$. Recall that we have chosen a fixed $T>1$ dependent only of $\Sp^2_+$ and the Hölder exponent $\gamma\in(0,1)$. Thus $\alpha=\alpha(\Sp^2_+,\gamma)$ and we derive
$$\|u\|_{C^{4,1,\gamma}([\tau,\infty)\times \Sp^2_+)}\leq\|u\|_{C^{4,1,\gamma}([nT,\infty)\times \Sp^2_+)}\leq Ce^{-\alpha\tau}\|u_0\|_{C^{4,\gamma}}+C\lambda.$$
\end{proof}
\begin{korollar}[Subconvergence]\label{subconvergence}\ \\
Let $\Omega\in C^{17}$, $0\leq\beta<\gamma<1$, $\lambda_0$, $\theta_0$, $\theta_1$ be as in Theorem \ref{theorem1}, $\lambda\leq\lambda_0$, $\phi_0\in S^{4,\gamma}(\lambda,\theta_0)$ and denote the solution to the area preserving Willmore flow with initial value $\phi_0$ by $\phi(t)$. Then any sequence $t_n\uparrow\infty$ contains a subsequence $t_{n_k}$ such that $\phi(t_{n_k})\rightarrow \phi_\infty$ in $C^{4,\beta}(\Sp^2_+)$ where $\phi_\infty\in S^{4,\beta}(\lambda, \theta_1)$ is a critical point of the elliptic problem (\ref{elliptic}). 
\end{korollar}
\begin{proof}
By Corollary \ref{longimeexistence} $\phi(t)$ is $C^{4,\gamma}$-bounded uniformly over time. Hence any sequence $t_n\uparrow\infty$ must contain a subsequence $t_{n_k}$ so that $\phi(t_{n_k})\rightarrow\phi_\infty$ in $C^{4,\beta}$ for some $\phi_\infty\in C^{4,\beta}(\Sp^2_+)$.\\

Now let $t_n\uparrow\infty$ such that $\phi(t_n)\rightarrow \phi_\infty$ in $C^{4,\beta}$. We prove $P_H^\perp(W(\phi_\infty))=0$ by contradiction. Else we may assume that for $\|\phi-\phi_\infty\|_{C^{4,\beta}}\leq 2\delta_0$ we have $\rho_0\leq \|P_H^\perp (W(\phi))\|_{C^{0,\beta}}\leq 2\rho_0$ for some $\delta_0>0$ and $\rho_0>0$. Clearly $\phi(t)$ cannot be $2\delta_0$-close to $\phi_\infty$ for all large times as this would imply 
$$\mathcal W(\phi(t))=\mathcal W(\phi(t_0))+\int_{t_0}^t-|P_H^\perp (W(\phi(s))|^2ds\leq \mathcal W(\phi(t_0))-\rho_0^2 (t-t_0)\rightarrow-\infty$$
for some fixed $t_0$ and $t\rightarrow\infty$. As $\phi(t_n)\rightarrow\phi_\infty$ we may however assume that $\|\phi(t_n)-\phi_\infty\|_{C^{4,\beta}}\leq\delta_0$ and choose sequences $\sigma_n$ and $\tau_n$ satisfying $\sigma_n<\tau_n<\sigma_{n+1}$ such that $\|\phi(\sigma_n)-\phi_\infty\|_{C^{4,\beta}}=\delta_0$, $\|\phi(\tau_n)-\phi_\infty\|_{C^{4,\beta}}=2\delta_0$ and $\|\phi(t)-\phi_\infty\|_{C^{4,\beta}}\leq 2\delta_0$ for all $t\in[\sigma_n,\tau_n]$. Then 
$$\|\phi(\tau_n)-\phi(\sigma_n)\|_{C^{0,\beta}}\leq \int_{\sigma_n}^{\tau_n}\|\dot\phi(s)\|_{C^{0,\beta}}ds\leq \int_{\sigma_n}^{\tau_n}\| P_H^\perp W(\phi(s))\|_{C^{0,\beta}}ds\leq 2\rho_0(\tau_n-\sigma_n).$$
Note that by Theorem \ref{theorem1} we may bound $\|\phi(t)\|_{C^{4,\gamma}}\leq K$ independent of $t$. For any $\mu>0$ we may use Ehrlings Lemma to get
\begin{align*} 
\delta_0\leq & \|\phi(\tau_n)-\phi(\sigma_n)\|_{C^{4,\beta}}\leq \mu\|\phi(\tau_n)-\phi(\sigma_n)\|_{C^{4,\gamma}}+C(\mu)\|\phi(\tau_n)-\phi(\sigma_n)\|_{C^{0,\beta}}\\
\leq & 2K\mu+C(\mu)2\rho_0(\tau_n-\sigma_n).
\end{align*}
Choosing $\mu=\frac{\delta_0}{2(1+K)}$ gives $\tau_n-\sigma_n\geq \kappa_0(\delta_0, K)>0$. 
Since $\mathcal W(\phi(t))$ decreases we deduce
$$\mathcal W(\phi(\sigma_{n+1}))\leq \mathcal W(\phi(\sigma_n))-\int_{\sigma_n}^{\tau_n}|P_H^\perp W(\phi(s))|^2\leq \mathcal W(\phi(t_n))-\rho_0^2\kappa_0.$$
Iterating gives $\mathcal W(\phi(\tau_n))\rightarrow-\infty$ for $n\rightarrow\infty$ which is a contradiction. 
\end{proof}

\section{Proof of Theorem 2}
\setcounter{equation}0
\subsection{Proof of Theorem 2i)}
The key-observation to prove Theorem \ref{theorem2}i) is the following Lemma:

\begin{lemma}\label{parity2}
Let $p_0\in S$, put $q_0:=\frac{\partial \tilde g ^{p_0,\lambda}}{\partial\lambda}\big|_{\lambda=0}$ and let $u_0\in C^{4,\gamma}(\Sp^2_+)$ be even. 
\begin{enumerate}
    \item For small $\epsilon,\mu>0$ the following functions are even: $B_0[\epsilon u_0,\delta+\mu q_0]$, $W[\epsilon u_0,\delta+\mu q_0] $, $H[\epsilon u_0,\delta+\mu q_0]$.
    \item For small $\epsilon,\mu>0$ and $i=1,2$ the function $\nabla C^i[\epsilon u_0,\delta+\mu q_0] $ is odd.
    \item $D_2 C[0,\delta]q_0=0$ and $D_2\psi_0[0,\delta]q_0$ (from Lemma \ref{boundaryvalueslem}) is even.
\end{enumerate}
If $T>1$ and $\alpha\in C^{1,\frac\gamma4}([0,T], S)$ put $q:=\frac{\partial \tilde g ^{\alpha,\lambda}}{\partial\lambda}\big|_{\lambda=0}$. Then $D_2\psi_T[0,\delta]q$ is even. 
\end{lemma}
\begin{proof}
Let $h_{ij}$ denote the second fundamental form of $S$ in the chart $f[p_0,\cdot]$. Recalling Equation (\ref{metricformula}) we get
$$q_0=\frac{\partial\tilde g^{p_0,\lambda}}{\partial\lambda}\bigg|_{\lambda=0}=\begin{bmatrix}
   0 & 0&  h_{1a}x^ a\\
   0 & 0&  h_{2a}x^ a\\
   h_{1a}x^ a & h_{2a}x^ a & 0
\end{bmatrix}\hspace{.5cm}\text{and put }\hspace{.5cm}T:=\begin{bmatrix}
    -1 & & \\
       & -1 & \\
       & & 1 \\
\end{bmatrix}.$$
It is easy to check that $(T^*q_0)(x)=q_0(x)$. For any $u_0\in C^{4,\gamma}(\Sp^2_+)$ and $\tilde g_0\in G_0^l$ we have $\mathcal O[u_0, T^*\tilde g_0]= \mathcal O[u_0\circ T, \tilde g_0]\circ T$ for $\mathcal O\in\set{B_0, W}$, $A[u_0, T^*\tilde g_0]= A[u_0\circ T, \tilde g_0]$ and $C[u_0, T^*\tilde g_0]= -C[u_0\circ T, \tilde g_0]$. See e.g. Lemma 10 in \cite{AK}. These relations readily imply the first two statements.  Exploiting the parity of the barycenter we get
$$D_2 C[0,\delta]q_0=\frac d{d\mu}\bigg|_{\mu=0}C[0,\delta+\mu q_0]=-\frac d{d\mu}\bigg|_{\mu=0}C[0\circ T,\delta+\mu q_0]=-D_2 C[0,\delta]q_0.$$
Hence the first part of the third statement is established. For the second part let $\tilde w_0:=D_2 \psi_0[0,\delta]q_0$. Then by construction of $\psi_0$ in Lemma \ref{boundaryvalueslem} and the first part of this Lemma
$$\left(\frac{\partial\tilde w_0}{\partial\eta},-\frac{\partial}{\partial\eta}(\Delta+2)\tilde w_0\right)=-D_2 B_0[0,\delta]\frac{\partial \tilde g^{p_0,\lambda}}{\partial\lambda}\bigg|_{\lambda=0}=\frac d{d\mu}\bigg|_{\mu=0}B_0[0,\delta+\mu q_0]=\textrm{even}.$$
By definition of the space $Y_0$ we deduce that $\tilde w_0$ is the unique solution to a problem of the form 
$$\Delta^2 \tilde w_0\equiv const.,\hspace{.5cm} \partial_\eta\tilde w_0,\ \partial_\eta\Delta\tilde w_0=\textrm{even}\hspace{.5cm}\text{and}\hspace{.5cm}\int_{\Sp^2_+}\tilde w_0d\mu_{\Sp^2}=0$$
which implies that $\tilde w_0$ is also even. A similar argument proves the claimed parity statement of $D_2\psi_T[0,\delta]q.$ 
\end{proof}

Let $p\in S$, $u_0\in C^{4,\gamma}(\Sp^2_+)$ and $\lambda>0$ be as in Corollary \ref{longimeexistence}. Let $u\in C^{4,1,\gamma}([0,\infty)\times \Sp^2_+)$ and $\xi\in C^{1,\frac\gamma4}([0,\infty),S)$ denote the unique solutions discussed in Theorem \ref{longimeexistence}. The time-dependent metric $\tilde g^{\xi,\lambda}$ is now defined for all $t\in[0,\infty)$. By the decay estimate provided in Theorem \ref{longimeexistence} we deduce that there exists a time $T_0(\lambda)>0$ such that $\|u\|_{C^{4,1,\gamma}([T_0,\infty)\times \Sp^2_+)}\leq M_0\lambda$. Thus, without loss of generality, we may assume
\begin{equation}\label{WLOGtheo2i}
    \|u\|_{C^{4,1,\gamma}([0,\infty)\times \Sp^2_+)}\leq M_0\lambda.
\end{equation}
for the proof of Theorem \ref{theorem2}i).  Given an arbitrary $T>1$ the graph function $u$ satisfies the Equation 
\begin{equation}\label{nonlinearueq}
\left\{\begin{aligned} 
Q_T[u, u_0,\tilde g^{\xi,\lambda},f]&=0,\\
B_T[u,\tilde g^{\xi,\lambda}]&=0,\\
u(0)&=u_0,
\end{aligned}\right.
\end{equation}
where $Q_T$ was defined in Equation (\ref{Qoperatordefinition}) and $f$ is chosen as
$$f=I^ i[u,\tilde g]P^\perp_{K[u,\tilde g]}\left[\tilde g\left(\lambda \bigg(D_2 F^\lambda[\xi,f_u]\bigg)^{-1}D_1 F^\lambda[\xi,f_u]b_i(\xi),\tilde\nu\right)\right]$$
as we described in Equation (\ref{thisisthechoiceforf}). We recall Estimates (\ref{diegl1}) and (\ref{diegl2}):
\begin{equation}\label{recallf}
\|P_{K[0,\delta]}^\perp f\|_{C^{0,0,\gamma}_T}+\|f\|_{C^{0,0,\gamma}_T}^2\leq C(T)\left(\|u\|_{C^{4,1,\gamma}_T}^2+\|\tilde g-\delta\|_{G_T^ l}^2+\lambda^2\right)\leq C(T,M_0)\lambda^2
\end{equation}
In the last step we have used (\ref{WLOGtheo2i}). The idea of the proceeding analysis is the following: We expand the nonlinear equation (\ref{nonlinearueq}) to first order in $u$ and $\lambda$. Making use of Lemma \ref{parity2} and Estimate (\ref{WLOGtheo2i}) we find
$$\left\{\begin{aligned}
    \dot u+\frac12\Delta(\Delta+2) u&=\textrm{even}\cdot\lambda+\mathcal O(\lambda^2), \\
    \frac{\partial u}{\partial\eta},-\frac{\partial}{\partial\eta}(\Delta+2) u&=\textrm{even}\cdot\lambda+\mathcal O(\lambda^2),\\
    u(0)&=u_0,
\end{aligned}\right.$$
where `$\textrm{even'}=\mathcal O(\lambda^0)$. This allows the derivation of the improved decay estimate 
\begin{equation}\label{odddecayeq}
\|u^-(T)\|_{C^{4,\gamma}}\leq C e^{-6T}\|u_0^-\|_{C^{4,\gamma}}+C(T)\lambda^2.
\end{equation}
Once this is established we may follow the \emph{gluing together} strategy demonstrated in the proof of Corollary \ref{longimeexistence} to deduce
$$\|u^-\|_{C^{4,1,\gamma}([\tau,\infty)\times\Sp^2_+)}\leq C\left(e^{-\alpha \tau}\|u^-(0)\|_{C^{4,\gamma}}+\lambda^2\right)$$
which implies Theorem \ref{theorem2}i).\\

Finally, recall the direct decompositions discussed in Lemma \ref{decomposition}. Defining the spaces
$$X_0^\parallel:=\operatorname{span}_\R\set{1,\omega^1,\omega^2}\hspace{.5cm}\textrm{and}\hspace{.5cm}X_0^\perp:=X_0\cap \left(X_0^\parallel\right)^{\perp_{L^2(\Sp^2_+)}}.$$
we may refine the decompositions from Lemma \ref{decomposition} by writing 
\begin{equation}\label{directdec3}
C^{4,\gamma}(\Sp^2_+)=X_0^\perp\oplus X_0^\parallel\oplus Y_0\hspace{.5cm}\textrm{and}\hspace{.5cm}C^{4,1,\gamma}([0,T]\times\Sp^2_+)=X_T^\perp\oplus X_T^\parallel\oplus Y_T
\end{equation}
where for $*=\parallel,\perp $ we introduced $X_T^*:=\set{u\in X_T\ |\ u(t)\in X_0^*\ \forall t\in[0,T]}$. We denote the projections from $C^{4,\gamma}(\Sp^2_+)$/ $C^{4,1,\gamma}([0,T]\times \Sp^2_+)$ onto the direct summands by $\pi_{X_0^\perp}$/ $\pi_{X_T^\perp}$, $\pi_{X_0^\parallel}$/ $\pi_{X_T^\parallel}$ and $\pi_{Y_0}$/ $\pi_{Y_T}$ respectively. It is readily checked that all these projections preserve parity and hence e.g. $\pi_{X_T^\perp}(u^-)=(\pi_{X_T^\perp}(u))^-$.

\begin{lemma}\label{odddecay}
Let $\Omega\in C^{13+2r}$ with $r\geq 2$ and $T>1$. There exists $\lambda_0(T)>0$ such that for $\lambda\in[0,\lambda_0]$ and a solution $(u,\xi)$ from \ref{longimeexistence} satisfying  (\ref{WLOGtheo2i}) we have
\begin{align*} 
&\|u^-\|_{C^{4,1,\gamma}_T}\leq C(T, M_0)\left[\|u_0^-\|_{C^{4,\gamma}}+\lambda^2\right]\\
&\|u^-(T)\|_{C^{4,\gamma}}\leq C(M_0)e^{-6T}\|u_0^-\|_{C^{4,\gamma}}+C(T,M_0)\lambda^2.
\end{align*}
\end{lemma}
\begin{proof}
Put $l:=8+r$. The operator $Q_T$ from (\ref{Qoperatordefinition}) is of class $C^{l-8}=C^r$ and the map $\lambda\rightarrow \tilde g^{\xi,\lambda}\in G_T^{8+r}$ is of class $C^r$. Putting $q:=\frac{\partial g^{\xi,\lambda}}{\partial\lambda}\big|_{\lambda=0}$, $q_0:=q(0)$ and expanding Equation (\ref{nonlinearueq}) around $(u, u_0, \lambda,f)=(0,0,0,0)$ then gives 
\begin{equation}\label{theo2iueqproof}
    \left\{\begin{aligned}
    \dot u+\frac12\Delta(\Delta+2)u&=\sum_{\mu=0}^2\frac{D_2 C^\mu[0,\delta]q}{\|\nabla C^\mu[0,\delta]\|_{L^2(\Sp^2_+)}^2}\nabla C^\mu[0,\delta]\\
    &\hspace{1cm}+P_{K[0,\delta]}^\perp \bigg[D_2 W[0,\delta]q-f\bigg]+R_1,\\
    \left(\frac{\partial u}{\partial\eta},-\frac{\partial}{\partial\eta}(\Delta+2)u\right)&=-D_2 B_T[0,\delta]q+R_2,\\
    u(0)&=u_0.
\end{aligned}\right.
\end{equation}
Using $Q_T\in C^2$ as well as Estimates (\ref{WLOGtheo2i}) and (\ref{recallf}) we get
$$\|R\|_{C^{0,0,\gamma}_T}\leq C(T)\left[\|u\|_{C^{4,1,\gamma}_T}^2+\lambda^2+\|f\|_{C^{0,0,\gamma}_T}^2\right]\leq C(T)\lambda^2.$$
We apply the projections $\pi_{X_T^\perp}$ and $\pi_{X_T^\parallel}$ to Equation (\ref{theo2iueqproof}). Using Lemma \ref{parity2} we may infer  that $v^-:=\pi_{X_T^\perp}(u^-)$ and $\tilde v^-:=\pi_{X_T^\parallel}(u^-)$ satisfy the following equations:
$$\left\{\begin{aligned}
   & \dot v^-+\frac12\Delta(\Delta+2)v^-=-P_{K[0,\delta]}^\perp [f]^-+\pi_{X_T^\perp}(R_1^-),\\
    &\frac{\partial v^-}{\partial\eta}=\frac{\partial\Delta v^-}{\partial\eta}=0,\\
   & v^-(0)=\pi_{X_0^\perp}(u_0^-).
\end{aligned}\right.
\hspace{.5cm}\textrm{and}\hspace{.5cm}
\left\{\begin{aligned}
    &\dot {\tilde v}^-+\frac12\Delta(\Delta+2)\tilde v^-=\pi_{X_T^\parallel}(R^-),\\
    &\frac{\partial \tilde v^-}{\partial\eta}=\frac{\partial\Delta \tilde v^-}{\partial\eta}=0,\\
    &\tilde v^-(0)=\pi_{X_0^\parallel}(u_0^-).
\end{aligned}\right.
$$
Applying the improved Schauder estimates from Appendix  \ref{improvedschauder} to the first equation and regular Schauder estimates (see e.g. \cite{simon, eidelman}) to the second one immediately gives 
\begin{align} 
\|v^-\|_{C^{4,1,\gamma}_T}\leq  C(T)&\left[\|u_0^-\|_{C^{4,\gamma}}+\lambda^2\right],\hspace{.5cm}\|v^-(T)\|_{C^{4,\gamma}}\leq Ce^{-6T}\|u_0^-\|_{C^{4,\gamma}}+C(T)\lambda^2\label{line1fgr}\\
\textrm{as well as}\hspace{.5cm}&\|\tilde v^-\|_{C^{4,1,\gamma}_T}\leq C(T)[\|\pi_{X_0^\parallel} u^-_0\|_{C^{4,\gamma}}+\lambda^2].\label{line2fgr}
\end{align}
By definition $\tilde v^-_0\in\operatorname{span}_\R\set{\omega^1,\omega^2}$. Using $C^i\in C^2$, $C[u_0,\tilde g^{p,\lambda}]=0$, $\nabla C^i[0,\delta]=-\frac3{2\pi}\omega^i$ and Lemma \ref{parity2} we get
\begin{align}
    \|\tilde v^-_0\|_{C^{4,\gamma}}&\leq C\max_{i=1,2}\left|\int_{\Sp^2_+}\omega^i u_0\right|\nonumber\\
    &\leq C\left|C[u_0,\tilde g^{p,\lambda}]-D_2C[0,\delta]q_0\right|+C(\|u_0\|_{C^{4,\gamma}}^2+\lambda^2)\nonumber\\
    &\leq C\lambda^2\label{line3fgr}
\end{align}
Let $w^-:=v^-+\tilde v^-=\pi_{X_T}(u^-)$ and $\tilde w^-:=\pi_{Y_T}(u^-)$. Combining Estimates (\ref{line1fgr}), (\ref{line2fgr}) and (\ref{line3fgr}) we derive
\begin{equation}\label{fgrwestimate}
\|w^-\|_{C^{4,1,\gamma}_T}\leq C(T)\left[\|u_0^-\|_{C^{4,\gamma}}+\lambda^2\right]\hspace{.5cm}\text{and}\hspace{.5cm}\|w^-(T)\|_{C^{4,\gamma}}\leq Ce^{-6T}\|u_0^-\|_{C^{4,\gamma}}+C(T)\lambda^2.
\end{equation}
Finally, we use $\tilde w^-=\psi_T[w,\tilde g^{\xi,\lambda}]^-$, $D_1\psi_T[0,\delta]=0$, the parity of $D_2\psi_T[0,\delta]q$ derived in Lemma \ref{parity2} and Estimates (\ref{WLOGtheo2i}) and (\ref{fgrwestimate}) to get
\begin{equation}\label{fgrtildewest}
\|\tilde w^-\|_{C^{4,1,\gamma}_T}\leq C(T)(\|w\|_{C^{4,1,\gamma}_T}^2+\lambda^2)\leq C(T)\lambda^2
\end{equation}
which, when combined with Estimate (\ref{fgrwestimate}), implies the Lemma as $u^-=w^-+\tilde w^-$.
\end{proof}

As noted above, Lemma \ref{odddecay} implies Theorem \ref{theorem2}i).

\subsection{Proof of Theorem 2ii)}
Recall that for a map $v:\Sp^2_+\rightarrow\R$ we denote the even part by $v^+$ and the odd part by $v^-$. We follow an idea from Mattuschka \cite{Mattuschka}. For $p\in S$, $\lambda>0$ and $M>0$ put 
$$\mathcal A_{p,\lambda}^M:=\left\{u_0\in C^{4,\gamma}(\Sp^2_+)\ \bigg|\ \substack{
    \text{$(u_0,p)$ is admissible, }\\
    \|u_0\|_{C^{4,\gamma}}\leq M\lambda\text{ and }\|u_0^-\|_{C^{4,\gamma}}\leq M\lambda^2
    }\right\}.$$
Note that for $u\in\mathcal A_{p,\lambda}^M$ the immersion $\phi_{p,u}^\lambda$ (recall Subsection \ref{terminilogy}) satisfies $\phi^\lambda_{p,u}\in \mathcal S^-(\lambda, M)$.\\

Let $(u,\xi)$ denote a solution as it is described in \ref{longimeexistence} and denote the basis vectors along $\xi$ by $b_i:[0,\infty)\rightarrow TS$. By assumption we have $u(t)\in\mathcal A_{\xi(t),\lambda}^M$ for all $t\geq 0$. The strategy to prove Theorem \ref{theorem2}ii) is to define a family of maps
$$\hat u_t:[0,\lambda]\rightarrow C^{4,\gamma}(\Sp^2_+)\hspace{.5cm}\textrm{such that }\hat u_t[\lambda]=u(t),\ \hat u_t[0]=0\hspace{.2cm}\textrm{and}\hspace{.2cm}\hat u_t[s]\in \mathcal A_{\xi(t),s}^{C(M)}.$$
Once this is established we may use the evolution equation (\ref{fgrwestimate}) to get the Equation 
$$\langle \dot\xi(t),b_i(t)\rangle=-\lambda I^i[u(t),\tilde g^{\xi(t),\lambda}]=-\lambda I^i[\hat u_t[\lambda],\tilde g^{\xi(t),\lambda}],$$
expand the right hand side in $\lambda$ and obtain the Theorem. The first step towards the construction of a suitable map is the following lemma.

\begin{lemma}\label{constrainslemma}
Let $l\geq 8$. There exist $\theta_0,\theta_1,\theta_2, \sigma_0>0$ and $C^{l-7}$ maps $w_0^\parallel:X_0^\perp(\theta_0)\times G_0^l(\delta;\sigma_0)\rightarrow X_0^\parallel(\theta_1)$ and $\tilde w_0:X_0^\perp(\theta_0)\times G_0^l(\delta;\sigma_0)\rightarrow Y_0(\theta_2)$  such that for all $(w_0^\perp,w_0^\parallel,\tilde w_0,\tilde g_0)\in X_0^\perp(\theta_0)\times X_0^\parallel(\theta_1)\times Y_0(\theta_2)\times G_0^l(\delta;\sigma_0)$  we have
    $$\begin{aligned}
    &B_0[w^\perp_0+ w^\parallel_0+\tilde w_0,\,\tilde g_0]=0\\
    &C[w^\perp_0+ w^\parallel_0+\tilde w_0,\,\tilde g_0]=0\\
    &A[w^\perp_0+ w^\parallel_0+\tilde w_0,\,\tilde g_0]=2\pi
    \end{aligned}\hspace{.3cm}\Leftrightarrow\hspace{.3cm} w^\parallel_0=w^\parallel_0[w^\perp_0,\tilde g_0]\text{ and }\tilde w_0=\tilde w_0[w^\perp,\tilde g_0].$$

The maps satisfy $D_1 w_0^\parallel[0,\delta]=0$ and $D_1 \tilde w_0[0,\delta]=0$.
\end{lemma}
\begin{proof}
From Lemma \ref{boundaryvalueslem} we immediately find that $\tilde w_0=w^\perp_0+w^\parallel_0+\psi_0[w^\perp_0+w^\parallel_0,\tilde g_0]$ must be true to satisfy $B_0[w^\perp_0+ w^\parallel_0+\tilde w_0,\tilde g_0]=0$. On a neighbourhood of $(0,0,\delta)$ we now define a map $F:X_0^\perp\times X_0^\parallel\times G_0^ l\rightarrow \R^3$ by
$$(w^\perp_0,w^\parallel_0,\tilde g_0)\mapsto \left(A\left[w^\perp_0+w^\parallel_0+\psi_0[w^\perp_0+w^\parallel_0,\tilde g_0],\tilde g_0\right]-2\pi, C\left[w^\perp_0+w^\parallel_0+\psi_0[w^\perp_0+w^\parallel_0,\tilde g_0],\tilde g_0\right]\right).$$
As $\psi_0, A$ and $C$ are of class $C^{l-7}$ we have $F\in C^{l-7}$. Clearly $F[0,0,\delta]=0$. Additionally
$$D_2F[0,0,\delta]\varphi=\left(2\int_{\Sp^2_+}\varphi d\mu_{\Sp^2},-\frac3{2\pi}\int_{\Sp^2_+}\varphi\omega^ 1 d\mu_{\Sp^2},-\frac3{2\pi}\int_{\Sp^2_+}\varphi\omega^ 2 d\mu_{\Sp^2}\right).$$
The implicit function theorem implies the existence of the neighbourhoods and the map $w^\parallel_0$ as well as its regularity. To establish $D_1 w_0^\parallel[0,\delta]=0$ we use $D_1\psi_0[0,\delta]=0$ and argue as in the proof of Lemma \ref{boundaryvalueslem}. $D_1\tilde w_0[0,\delta]=0$ then follows from
\begin{equation}\label{tildew0functional}
\tilde w_0[w^\perp,\delta]=\psi_0[w^\perp_0+w^\parallel_0[w^\perp_0,\delta],\delta].
\end{equation}
\end{proof}
We need the following properties of the maps $w^\parallel_0$ and $\tilde w_0$ that we just constructed.

\begin{lemma}\label{evenfunctions}
Let $\Omega\in C^{13+2r}$ so that $\lambda\mapsto\tilde g ^{p,\lambda}\in G_0^{8+r}$ is a map of class $C^r$. Let $M>0$: There exists $\lambda_0(M)>0$ such that for $\lambda\in[0,\lambda_0(M)]$:
\begin{enumerate}
\item The following are even functions:
\begin{equation*}
    D_2 w_0^\parallel[0,\delta]\frac{\partial\tilde g^{p,\lambda}}{\partial\lambda}\bigg|_{\lambda=0}\hspace{.5cm}\text{and}\hspace{.5cm}
D_2 \tilde w_0[0,\delta]\frac{\partial\tilde g^{p,\lambda}}{\partial\lambda}\bigg|_{\lambda=0}
\end{equation*}
 \item For $(w_0^\perp,p)\in X_0\times S$ satisfying $\|w_0^\perp\|_{C^{4,\gamma}}\leq M\lambda$ put $u_0:=w_0^\perp+w^\parallel[w_0^\perp,\tilde g^{p,\lambda}]+\tilde w[w_0^\perp,\tilde g^{p,\lambda}]$. Then $\|u_0^--(w_0^\perp)^-\|\leq C(M)\lambda^2$.
    \item  For admissible $(u_0,p)\in C^{4,\gamma}(\Sp^2_+)\times S$ satisfying $\|u_0\|_{C^{4,\gamma}}\leq M\lambda$ and $w_0^\perp:=\pi_{X_0^\perp}(u_0)$ we have $\|u_0^--(w_0^\perp)^-\|\leq C(M)\lambda^2$.
\end{enumerate}
\end{lemma}
\begin{proof}
We exploit $D_2 C^ i[0,\delta]\frac{\partial \tilde g^{p,\lambda}}{\partial\lambda}\big|_{\lambda=0}=0$ from Lemma \ref{parity2}. 
Using $D_2 \tilde w_0[0,\delta]=D_2\psi_0[0,\delta]$ from Equation (\ref{tildew0functional}) and the parity of $D_2 \psi_0[0,\delta ]\frac{\partial\tilde g^{p,\lambda}}{\partial\lambda}\big|_{\lambda=0}$ from Lemma \ref{parity2} we deduce the second claimed parity and can compute 
\begin{align}
    0&=\frac d{d\lambda}\bigg|_{\lambda=0}C^ i[w^\parallel_0[0,\tilde g^{p,\lambda}]+\tilde w_0[0,\tilde g^{p,\lambda}],\tilde g^{p,\lambda}]\nonumber\\
    &=-\frac3{2\pi}\int_{\Sp^2_+}\omega^ iD_2w_0^ \parallel[0,\delta]\frac{\partial \tilde g^{p,\lambda}}{\partial\lambda}\bigg|_{\lambda=0}d\mu_{\Sp^2}-\frac3{2\pi}\int_{\Sp^2_+}\omega^ iD_2 \tilde w_0[0,\delta]\frac{\partial \tilde g^{p,\lambda}}{\partial\lambda}\bigg|_{\lambda=0}d\mu_{\Sp^2}\nonumber\\
    &=-\frac3{2\pi}\int_{\Sp^2_+}\omega^ iD_2w_0^ \parallel[0,\delta]\frac{\partial \tilde g^{p,\lambda}}{\partial\lambda}\bigg|_{\lambda=0}d\mu_{\Sp^2}\label{andhenceeven}
\end{align}
As $D_2w_0^ \parallel[0,\delta]\frac{\partial \tilde g^{p,\lambda}}{\partial\lambda}\big|_{\lambda=0}\in X_0^\parallel=\operatorname{span}_\R\set{1,\omega^1,\omega^2}$ Equation (\ref{andhenceeven}) implies the claimed parity. For the second part we use Lemma \ref{constrainslemma} and part 1 of this Lemma to estimate 
\begin{align*} 
\|u_0^--(w_0^\perp)^-\|_{C^{4,\gamma}}\leq & \|w_0^\parallel[w^\perp_0, \tilde g^{p,\lambda}]^-\|_{C^{4,\gamma}}+\|\tilde w_0[w^\perp_0, \tilde g^{p,\lambda}]^-\|_{C^{4,\gamma}}\\
\leq & C\left(\|w^\perp_0\|_{C^{4,\gamma}}^2+\lambda^2\right)\\
\leq & C(M)\lambda^2.
\end{align*}
The final statement follows from the second by noting that admissibility implies $u_0=w_0^\perp+w^\parallel_0[w_0^\perp,\tilde g^{p,\lambda}]+\tilde w_0[w_0^\perp,\tilde g^{p,\lambda}]$ and that the continuity of the projection $u_0\mapsto w_0^\perp$ gives the necessary estimate. 
\end{proof}

\begin{lemma}\label{map}
Let $\Omega\in C^{13+2r}$ with $r\geq 2$ and $M>0$. There exists $\lambda_0(M)>0$ such that for all $p\in S$, $\lambda\in[0,\lambda_0(M)]$ and $u_0\in\mathcal A^M_{p,\lambda}$ there exists a map $\hat u_0:[0,\lambda]\rightarrow C^{4,\gamma}(\Sp^2_+)$ of class $C^r$ that has the following properties:
\begin{enumerate}
    \item $\hat u_0(0)=0$ and $\hat u(\lambda)=u$. Additionally $\hat u_0(s)\in \mathcal A^{C(M)}_{p,s}$ for all $s\in[0,\lambda]$. 
    \item $\hat u_0'(0)$ is an even function.
    \item $\|\hat u_0\|_{C^r([0,\lambda], C^{4,\gamma}(\Sp^2_+))}\leq C(M,\Omega,r).$
\end{enumerate}
\end{lemma}
\begin{proof}
Using the direct decomposition (\ref{directdec3}) we may write $u_0=w_0^\perp+w_0^\parallel+\tilde w_0$ for unique $w_0^\perp\in X_0^\perp,\ w_0^\parallel\in X_0^\parallel$ and $\tilde w_0\in Y_0$. As $\|u_0\|_{C^{4,\gamma}}\leq M\lambda$ we may conclude $\|w^\perp_0\|_{C^{4,\gamma}}+\|w^\parallel_0\|_{C^{4,\gamma}}+\|\tilde w_0\|_{C^{4,\gamma}}\leq C(M)\lambda$ by using continuity of the projections. Since $u_0$ is admissible we may conclude that for small enough $\lambda_0(M)$ and $\lambda\leq \lambda_0(M)$ we have $w^\parallel_0=w^\parallel_0[w^\perp_0,\tilde g^{p,\lambda}]$ and $\tilde w_0=\tilde w_0[w^\perp_0,\tilde g^{p,\lambda}]$. Additionally, Lemma \ref{evenfunctions} implies $\|\left(w^\perp_0\right)^-\|_{C^{4,\gamma}}\leq C(M)\lambda^2$.
Combining these estimates we derive that the $C^\infty$-\,curve
$$\hat w^\perp_0:[0,\lambda]\rightarrow C^{4,\gamma}(\Sp^2_+),\ \hat w_0(s):=\frac s\lambda \left(w^\perp_0\right)^++\frac{s^2}{\lambda^2}\left(w^\perp_0\right)^-$$
has all derivatives with respect to $s$ bounded by a constant $C(M)$ independent of $\lambda$. Also note that $\|\hat w^\perp_0(s)\|_{C^{4,\gamma}}\leq C\|w^\perp_0\|_{C^{4,\gamma}}\leq C(M)\lambda$. Hence, for $\lambda$ small enough, we may define the curve 
$$\hat u_0:[0,\lambda]\rightarrow C^{4,\gamma}(\Sp^2_+),\ \hat u_0(s):=\hat w^\perp_0(s)+w_0^\parallel[\hat w^\perp_0(s),\tilde g^{p,s}]+\tilde w_0[\hat w^\perp_0(s),\tilde g^{p,s}].$$
Note that for $\Omega\in C^{13+2r}$ the map $s\mapsto \tilde g^{p,s}\in G_0^{8+r}$ is of class $C^r$ and that $w^\parallel_0$ and $\tilde w_0$ are both maps of class $C^r$. Hence $\hat u_0\in C^r([0,\lambda], C^{4,\gamma}(\Sp^2_+))$ and
\begin{equation}
    \|\hat u_0\|_{C^r([0,\lambda], C^{4,\gamma}(\Sp^2_+))}\leq C(M,\Omega,r)
\end{equation}
independent of $\lambda$ as this is the case for the curve $s\mapsto \hat w^\perp_0(s)$. The facts $\hat u_0(0)=0$ and $\hat u_0(\lambda)=u_0$ are trivial. Also note that by construction  $\hat u_0(s)\in \mathcal A^{C(M)}_{s, p}$. Indeed this follows from  the last  statement of Lemma \ref{evenfunctions}. Finally, $\hat u_0'(0)$ is even due to Lemma \ref{evenfunctions} as
    $$\hat u'_0(0)=\frac1\lambda\left(w_0^\perp\right)^++D_2w_0^\parallel[0,\delta]\frac{\partial\tilde g^{p,s}}{\partial s}\bigg|_{s=0}+D_2\tilde w_0[0,\delta]\frac{\partial\tilde g^{p,s}}{\partial s}\bigg|_{s=0}.$$

\end{proof}

Using the map $\hat u_0$ constructed in Lemma \ref{map} we may now prove the second part of Theorem \ref{theorem2}. 
\begin{theorem}
Let $\Omega\in C^{21}$ and $M>0$. There exists $\lambda_0(M)>0$ such that any solution $(u,\xi)$ to (\ref{system}) with $u(t)\in\mathcal A_{\xi(t),\lambda}^M$ for all times $t\geq 0$ satisfies
$$\sup_{t\geq 0}\left|\dot\xi(t)-\frac32\lambda^3\nabla H^S(\xi(t))\right|\leq C(M)\lambda^4.$$
\end{theorem}
\begin{proof}
By assumption $\Omega\in C^{13+2r}$ with $r:=4$. For $t\geq 0$ let $u_t(\omega):=u(t,\omega)$. Then, for any $t\geq 0$, we may define the map 
$\hat u_t:[0,\lambda]\rightarrow C^{4,\gamma}(\Sp^2_+)$ as in Lemma \ref{map} and get $\hat u_t(s)\in\mathcal A_{\xi(t),s}^{C(M)}$ for all $t\geq 0$. Now pick $t\geq 0$ and use Equation (\ref{system}) to write
$$\langle\dot\xi(t), b_i(t)\rangle=-\lambda I^i[u(t),\tilde g^{\xi(t),\lambda}]=-sI^i[\hat u_t(s),\tilde g^{\xi(t),s}]\bigg|_{s=0}^{s=\lambda}.$$
In the last step we have used $\hat u_t(\lambda)=u_t=u(t)$, $\hat u_t(0)=0$ and $I^i[0,\delta]=0$. Note:
\begin{enumerate}
    \item $I^i:C^{4,\gamma}(\Sp^2_+)\times G_0^{8+r}\rightarrow\R$ is of class $C^r$ on a small neighbourhood of $(0,\delta)$ as stated in Section \ref{section4}.
    \item $\hat u_t:[0,\lambda]\rightarrow C^{4,\gamma}(\Sp^2_+)$ is of class $C^r$ by Lemma \ref{map}.
    \item $s\mapsto \tilde g^{p,s}\in G_0^{8+r}$ is of class $C^r$ as we discussed in the proof of Corollary \ref{prescirbedcurve3}.
\end{enumerate}
Hence $y_{t,i}:[0,\lambda]\rightarrow \R,\ y_{t,i}(s):=-s I^i[\hat u_t(s),\tilde g^{\xi(t),s}]$ is of class $C^r\subset C^4$ and thus
$$y_{t,i}(\lambda)=y_{t,i}(0)+ y_{t,i}'(0)\lambda+\frac12y_{t,i}''(0)\lambda^2+\frac16 y_{t,i}'''(0)\lambda^3+\frac{\lambda^4}6\int_0^1y^{(4)}_{t,i}(\lambda\rho)(1-\rho)^3d\rho.$$
The last term can be bounded by $C(M)\lambda^4$ using Lemma \ref{map}. A direct computation in Appendix \ref{expansion} shows $y_{t,i}(0)=y_{t,i}'(0)=y_{t,i}''(0)=0$ and $y_{t,i}'''(0)=9\partial_i H^S(\xi(t))$ with $\partial_i$ taken in the chart $f[\xi(t),\cdot]$. Thus for all $t\geq 0$ 
$$\left|\langle\dot\xi(t),b_i(\xi(t))\rangle-\frac32\lambda^3\partial_i H^S(\xi(t))\right|\leq C(M)\lambda^4.$$
\end{proof}
This is Theorem \ref{theorem2}ii). Just recall that we rescaled time by a factor of $\lambda^4$ right before Equation (\ref{evolutionlaws}).

\section{Prove of Theorem 3}
We need the following uniqueness result of Alessandroni and Kuwert from \cite{AK}, Theorem 3:
\begin{theorem}\label{kuwertunique}
Let $\Omega\in C^{12}$ and $q\in S$ be a nondegenerate critical point of $H^S$. There exists $\lambda_0>0$, a neighbourhood $U$ of $q$ and a $C^1$ curve $\gamma:[0,\lambda_0]\rightarrow U$ satisfying $\gamma(0)=q$ such that for each $\lambda\in[0,\lambda_0)$ there exists a unique solution $\phi_\lambda\in \mathcal M^{4,\gamma}_\lambda(S)$ of (\ref{elliptic}) with area $2\pi\lambda^2$ and barycenter $C[\phi_\lambda]\in U$. The barycenter satisfies $C[\phi_\lambda]=\gamma(\lambda).$ 
\end{theorem}

To deduce the convergence of the flow, we first prove that under the assumptions of Theorem \ref{theorem3} any barycenter curve must stabilise in a small neighbourhood of a critical point of $H^S$. For that purpose we introduce $U(a,r):=B_r(a)\cap S$ for $a\in S$ and $r>0$.

\begin{lemma}\label{stabilizelemma}
Let $\Omega\in C^{23}$ and assume that $H^S$ only has finitely many critical points $p_1,\hdots,p_m$ all of which are nondegenerate. Then for all $\epsilon_0>0$ small enough, there exists $\lambda_0(\epsilon_0)>0$ such that the following is true:\\

Let $\lambda\leq\lambda_0$, $\phi_0\in S(\lambda,\theta_0)$, denote the solution to the area preserving Willmore flow with initial value $\phi_0$ by $\phi(t)$ and its barycenter curve by $\xi(t)$. There exists a time $T=T(\phi_0,\lambda,\epsilon_0)$ and $1\leq i\leq m$ such that $\xi(t):=C[\phi(t)]\in U(p_i,\epsilon_0\lambda^{\frac14})$ for all $t\geq T$.
\end{lemma}
\begin{proof}
For $h>0$ let $\mathcal N(\epsilon_0,h):=\bigcup_{i=1}^mU(p_i,\epsilon_0h)$ and $\mathcal N^c(\epsilon_0,h):=S\backslash \mathcal N(\epsilon_0,h)$. As $H^S$ only has finitely many critical points that are all nondegenerate we get $\inf_{\mathcal N^c(\epsilon_0,h)}|\nabla H^S|\geq \kappa_0\epsilon_0h$ for some $\kappa_0>0$, $h\leq 1$ and small enough values of $\epsilon_0$. Putting $a(t):=\xi(\lambda t)$ we may use Theorem \ref{theorem2}ii) for large enough times to estimate
\begin{equation}\label{hdecay}
\frac d{dt}H^S(a(t))=\nabla H^S(a(t))\left(\frac 32\nabla H^S(a(t))+R(t)\right)\geq \frac 32 |\nabla H^S(a(t))|^2-C\lambda.
\end{equation}

Let $t_n\uparrow\infty$. Then there exists a subsequence $(t_{n_k})$ so that $a(t_{n_k})\in \mathcal N(\epsilon_0,\lambda^{\frac13})$. Indeed if this was not true there would be a time $T_0>0$ with the property that $a(t)\in\mathcal N^c(\epsilon_0,\lambda^{\frac13})$ for all times $t\geq T_0$. Estimate (\ref{hdecay}) and the lower bound for $\nabla H^S$ on $\mathcal N^c(\epsilon_0,\lambda^{\frac13})$ give for small $\lambda(\epsilon_0)$ and $t\rightarrow\infty$
$$H(a(t))-H(a(T_0))=\int_{T_0}^t\frac d{ds} H^S(a(s))ds\geq \int_{T_0}^t \frac 32\kappa_0^2\epsilon_0^2\lambda^{\frac23} -C\lambda ds\rightarrow\infty$$
which is a contradiction. \\

Now let $p$ and $q$ be (not necessarily distinct) critical points of $H^S$ and let us examine what happens when $a$ travels from $a(t_1)\in U(p,\epsilon_0\lambda^{\frac13})$ to $a(t_2)\in U(q,\epsilon_0\lambda^{\frac13})$ when it exists $\mathcal N(\epsilon_0,\lambda^{\frac14})$ during its journey. We may choose times $\tilde t_1<\tilde t_2$ such that $a(\tilde t_1)\in \partial U(p,\epsilon\lambda^{\frac14})$ to $a(\tilde t_2)\in \partial U(q,\epsilon\lambda^{\frac14})$ as well as $a(s)\in \mathcal N^c(\epsilon_0,\lambda^{\frac14})$ for all times $s\in[\tilde t_1,\tilde t_2]$.
As there are only finitely many critical points we may bound $d(a(\tilde t_1),a(\tilde t_2))\geq C(S)>0$ for $\epsilon_0$ and $\lambda(\epsilon_0)$ small. Using Theorem \ref{theorem2} to bound $\dot a$ we conclude
\begin{equation}\label{timebound}
    0<C(S)\leq d(a(\tilde t_1),a(\tilde t_2))\leq \int_{\tilde t_1}^{\tilde t_2}|\dot a(s)|ds\leq\tilde  C(S,\phi_0)|\tilde t_2-\tilde t_1|
\end{equation}
It is easy to see that $|H(a(t_1))-H(p)|\leq C(S)\epsilon_0^2\lambda^{\frac23}$ and that for small enough $\lambda$ 
\begin{equation}\label{integralnegative} 
\int_{t_1}^{\tilde t_1}\nabla H(a(s))\dot a(s)ds\geq \int_{t_1}^{\tilde t_1}\frac 32|\nabla H(a(s))|^2-C\lambda ds\geq  \int_{t_1}^{\tilde t_1}\frac32\kappa_0^2\epsilon_0^2\lambda^{\frac23}-C\lambda ds\geq 0
\end{equation}
by Theorem \ref{theorem2}ii). Similar statements are true for $t_2$ and $\tilde t_2$. Using the Expansion for $\dot a$ from Theorem \ref{theorem2}ii) we compute
\begin{align*}
    H(q)-H(p)&\geq  H(a(t_2))-H(a(t_1))-C(S)\epsilon_0^2\lambda^{\frac23}\\
           & \overset{(\ref{integralnegative})}\geq  \int_{\tilde t_1}^{\tilde t_2}\nabla H(a(s))\dot a(s)ds-C(S)\epsilon_0^2\lambda^{\frac23}\\
          &  \geq \int_{\tilde t_1}^{\tilde t_2} \frac32\kappa_0^2\epsilon_0^2\lambda^{\frac12}-C\lambda ds-C(S)\epsilon_0^2\lambda^{\frac23}\\
            &\geq C(\epsilon_0,S)\lambda^{\frac12}|\tilde t_2-\tilde t_1|-C(S)\epsilon_0^2\lambda^{\frac23}
\end{align*}
In the second to last step we used the gradient estimate for $\nabla H^S$ on $\mathcal N^c(\epsilon_0, \lambda^{\frac14})$. Using (\ref{timebound}) we get $H(q)>H(p)$ for small enough $\lambda$ and as there are only finitely many critical points $a(t)$ must stabilise in $U(p,\epsilon_0\lambda^{\frac13})$ for some critical point $p$. 
\end{proof}

We now prove Theorem \ref{theorem3}
\begin{proof}
Let $\phi(t)$ be a flow line. Lemma \ref{stabilizelemma} implies that $\xi(t):=C[\phi(t)]$ must stabilize in a neighbourhood $U(p,\epsilon_0\lambda^{\frac14})$ of some critical point $p$ of $H^S$. Using Lemma \ref{subconvergence} choose $t_n\rightarrow \infty$ such that $\phi(t_n)$ converges to a critical point $\phi_\infty$ of the elliptic problem (\ref{elliptic}). Now let $\tau_n\rightarrow\infty$ be any other sequence and assume that $\phi(\tau_n)\not\rightarrow \phi_\infty$. Lemma \ref{subconvergence} implies that there exists a subsequence $\tau_{n_k}$ and a solution $\phi_\infty'\neq \phi_\infty$ of the elliptic problem \ref{elliptic} that lies in $C^{4,\beta}(S^2_+)$ such that $\phi(\tau_{n_k})\rightarrow\phi'_\infty$ in $C^{4,\beta}$. As $\phi'_\infty$ and $\phi_\infty$ both have their barycenter in $U(p,\epsilon_0\lambda^{\frac13})$ this contradicts Theorem \ref{kuwertunique} for small enough $\lambda$. The fact, that the limit $\phi_\infty$ is one of the solutions constructed in \cite{AK} also follows from the uniqueness result in Theorem \ref{kuwertunique}.
\end{proof}

\appendix
\renewcommand{\theequation}{\thesection.\arabic{equation}}
\section{Expansion}\label{expansion}
\setcounter{equation}0
For the duration of this section we will write $u(s):=\hat u_t(s)$, $p:=\xi(t)$, $\tilde g(s):=\tilde g^{p,s}$ and $d\mu[u,\tilde g]:=d\mu_{f_u^*\tilde g}$. We consider the map 
$$z(s):=I^i[u(s),\tilde g(s)]=\int_{\Sp^2_+}W[u(s),\tilde g(s)]P_{H[u(s),\tilde g(s)]}^\perp\nabla C^i[u(s),\tilde g(s)]d\mu[u(s),\tilde g(s)]$$
and prove that $z(0)=z'(0)=0$ and $z''(0)=-3\partial_i H^S(p)$. This implies the identities for $y_{t,i}(s)=-sz(s)$.\\

First we note that as $W[0,\delta]=0$, $H[0,\delta]=2$ and $\nabla C^i[0,\delta]=-\frac3{2\pi}\omega^i$ we may conclude $z(0)=0$ and
$$z'(0)=-\frac 3{2\pi}\int_{\Sp^2_+}\left(D_1 W[0,\delta]\frac{\partial u}{\partial s}\bigg|_{s=0}+D_2 W[0,\delta]\frac{\partial \tilde g}{\partial s}\bigg|_{s=0}\right)\omega^i d\mu_{\Sp^2}.$$
As $u'(0)$ is even $D_1 W[0,\delta]u'(0)=-\Delta(\Delta+2)u'(0)$ is also even and the first integral vanishes. For the second integral we recall Lemma \ref{parity2} to learn that  $D_2 W[0,\delta]\tilde g'(0)$ is an even function and deduce that the second integral also vanishes. It remains to compute $z''(0)$. Again using $W[0,\delta]=0$ we get
\begin{align}
    z''(0)=&\frac{d^2}{ds^2}\bigg|_{s=0}\int_{\Sp^2_+}W[s u'(0), \delta+s\tilde g'(0)]\left(P_{H}^\perp \nabla C^i\right)[su'(0),\delta+s\tilde g'(0)]d\mu[su'(0),\delta+s\tilde g'(0)]\nonumber\\
    &+\int_{\Sp^2_+}\left(D_1 W[0,\delta]u''(0)+D_2 W[0,\delta]\tilde g''(0)\right)\left(P_H^\perp\nabla C^i\right)[0,\delta]d\mu_{\Sp^2}.\label{lastline}
\end{align}
We claim that the first line vanishes.  First note that by Lemma \ref{parity2} the integrand in the first line is an odd function. We now need to check that the derivative of the measure at $s=0$ is even\footnote{As $W[0,\delta]=0$ one of the $s$-derivatives must act on the Willmore operator. Hence it suffices to check the first derivative of the measure.}. For this we use the following two formulas from \cite{AK} (see Lemma 1\footnote{Note that our variation is along the outer normal.} and the proof of Lemma 7)
\begin{align*} 
&D_1 d\mu[0,\delta]\varphi=2\varphi d\mu_{\Sp^2},\\
&D_2 d\mu[0,\delta]q=\frac12\tr_{\Sp^2}q.
\end{align*}
Inserting $\tilde g'(0)$ from the proof of Lemma \ref{parity2} we see that 
\begin{align*}
   &D_1 d\mu[0,\delta]u'(0)=2u'(0)\\
   &D_2 d\mu[0,\delta]\tilde g'(0)=\tr_{\R^3} \tilde g'(0)-g'(0)(\omega,\omega)=-2h_{ab}(p)\omega^ a\omega^ b\omega^3
\end{align*}
are both even. Thus the first line in Equation (\ref{lastline}) vanishes and we must only compute the last line. By definition
$$W[u,\tilde g]=\frac12\left(\Delta_g H+|h^ 0|^2H+\operatorname{Ric}^{\tilde g}(\tilde\nu,\tilde\nu)H\right).$$
As $\tilde g(s)$ is the pullback of a flat metric we may drop the last term. Next we note that $h^ 0[0,\delta]=0$ and hence $D_1 |h^ 0|^2[0,\delta]=0$ and similarly $D_2 |h^ 0|^2[0,\delta]=0$. Using $H[0,\delta]=2$ we get
$$\left(D_1\Delta[0,\delta]\varphi\right)H[0,\delta]=\left(D_2\Delta[0,\delta]q\right)H[0,\delta]=0$$
for arbitrary $\varphi$, $q$. Combining these considerations we deduce
\begin{equation}\label{fdoubleprimie0}
z''(0)=\frac12\int_{\Sp^2_+}\Delta\left(D_1 H[0,\delta]u''(0)+ D_2H[0,\delta]\tilde g''(0)]\right)\frac{-3}{2\pi}\omega^ id\mu_{\Sp^2}.
\end{equation}
Using \cite{AK} Lemma 7 and the standard result for the variation of $H$, we get
\begin{align} 
D_2 H[0,\delta]\varphi&=-(\Delta+|h^0|^2)\varphi=-(\Delta+2)\varphi\label{Hlinearization},\\
D_2 H[0,\delta]q&=-\frac12\tr_{\Sp^2}\nabla_{\tilde \nu} q+\tr_{\Sp^2}\nabla_\cdot q(\tilde \nu,\cdot)+q(\tilde \nu,\tilde \nu)-\tr_{\Sp^2}q.\label{Hderivative}
\end{align}
We apply (\ref{Hderivative}) to $q=\tilde g''(0)$. Denoting the second fundamental form of $S$ in the chart $f[p,\cdot]$ by $h_{ij}$ we can use Formula (\ref{metricformula}) to get ($*$-\,entries are to be inferred from symmetry)
$$\tilde g''(0)=\begin{bmatrix}
2h_{1a}h_{1b}x^ax^b & 2h_{1a}h_{2b}x^ax^b & \partial_1 h_{ab}x^ax^b\\
* & 2h_{2a}h_{2b}x^ax^b & \partial_2 h_{ab}x^ax^b\\
*& * &0
\end{bmatrix}.$$
Inserting into (\ref{Hderivative}) and repeatedly using $\tr_{\Sp^2}A=\tr_{\R^3}A-A(\omega,\omega)$ gives 
\begin{align}
\int_{\Sp^2_+}\Delta D_2 H[0,\delta]\tilde g''(0)\omega^id\mu_{\Sp^2}&=\int_{\Sp^2_+}\Delta\left(6\partial_a h_{kj}\omega^k\omega^j\omega^a\omega^3-2\partial_jh_{ja}\omega^a\omega^3\right)\omega^id\mu_{\Sp^2}\nonumber\\
&=-3\pi\partial_i H(p).\label{schauhier}   
\end{align}
Next we must consider 
\begin{equation}\label{intbyparts}
\int_{\Sp^2_+} \Delta D_1 H[0,\delta]u''(0) d\mu_{\Sp^2}=-\int_{\Sp^2_+}\Delta(\Delta+2)u''(0)\omega^id\mu_{\Sp^2}=\int_{\partial \Sp^2_+}\frac{\partial \Delta u''(0)}{\partial\eta}\omega^i d\mu_{\Sp^2}.
\end{equation}
To evaluate this integral we must exploit the boundary condition $B_0[u(s),\tilde g(s)]=0$ and differentiate twice.

\begin{lemma}\label{boundarylin}
The following identities hold:
\begin{align*} 
D_1 B[0,\delta]\varphi&=\left(\frac{\partial\varphi}{\partial\eta},-\frac{\partial}{\partial\eta}(\Delta+2)\varphi\right)\\
D_2 B[0,\delta]\tilde g''(0)&=\left(\partial_a h_{bc}\omega^ a\omega^ b\omega^ c, 10\partial_a h_{bc}\omega^a\omega^ b\omega^ c-2\partial_a h_{ab}\omega^b\right)
\end{align*}
\end{lemma}

Denote arbitrary even functions from $\Sp^2_+$ to $\R$ by $E$. Once Lemma \ref{boundarylin} is proven we may combine it with Lemma \ref{parity2} to get
\begin{align} 
0=&D_1 B_0[0,\delta]u''(0)+D_2 B_0[0,\delta]g''(0)+\frac {d^2}{ds^2}B_0[su'(0),\delta+ s\tilde g'(0)]\nonumber\\
=&\left(\frac{ u''(0)}{\partial\eta}+\partial_a h_{bc}\omega^ a\omega^ b\omega^ c+E, -\frac\partial{\partial\eta}(\Delta+2) u''(0)+10\partial_a h_{bc}\omega^a\omega^ b\omega^ c-2\partial_a h_{ab}\omega^b+E\right).\label{boundaryequs}
\end{align}
Combining both components of Equation (\ref{boundaryequs}) gives 
$$\frac{\partial\Delta u''(0)}{\partial\eta}=12\partial_a h_{bc}\omega^a\omega^b\omega^c-2\partial_a h_{ab}\omega^b+E.
$$
Inserting into Equation (\ref{intbyparts}) and combining with Equation (\ref{schauhier}) yields
$$\frac12\int_{\Sp^2_+}\Delta\left(D_1 H[0,\delta]u''(0)+ D_2H[0,\delta]\tilde g''(0)]\right)d\mu_{\Sp^2}=2\pi\partial_i H^S(p).$$
Multiplying with $-\frac 3{2\pi}$ and recalling Equation (\ref{fdoubleprimie0}) shows $z''(0)=-3\partial_i H^S(p)$. It remains to prove Lemma \ref{boundarylin}.

\begin{proof}
Let $\omega_0\in\partial \Sp^2_+$ and $\phi:U\subset\R^2\rightarrow\phi(U)\subset\Sp^2_+$ be a parameterization near $\omega_0$ so that $g_{ij}:=\langle\partial_i\phi,\partial_j\phi\rangle$ satisfies $g_{ij}(\omega_0)=\delta_{ij}$ and $\partial_k g_{ij}(\omega_0)=0$. Put $f_\epsilon(\omega):=(1+\epsilon \varphi(\omega))\omega$, $\mathfrak g(\mu):=\delta+\mu q$ with $q:=\tilde g''(0)$ and let $\tilde \nu(\epsilon,\mu)$ denote the inner normal of $f_\epsilon$ with respect to $\mathfrak g(\mu)$ (so $\tilde\nu(0,0)=-\omega$). Then at $\omega_0$
\begin{equation}\label{erivativesofnormals}
\frac d{d\epsilon}\bigg|_{\epsilon,\mu=0}\tilde\nu=\partial_i \varphi\partial_i\phi\hspace{.5cm}\text{and}\hspace{.5cm}\frac d{d\mu}\bigg|_{\epsilon,\mu=0}\tilde\nu=q(\phi,\partial_i\phi)\partial_i\phi+\frac12 q(\phi,\phi)\phi.
\end{equation}
The second formula is proven in \cite{AK} (see Lemma 7). To get the first one first note that $\langle\tilde\nu(\epsilon,0),\tilde\nu(\epsilon,0)\rangle=1$ implies that $\partial_\epsilon\tilde\nu(0,0)$ must be tangential. Hence
$$\frac{\partial\tilde\nu}{\partial\epsilon}\bigg|_{\epsilon,\mu=0}=\langle\partial_i\phi,\frac{\partial\tilde\nu}{\partial\epsilon}\bigg|_{\epsilon,\mu=0}\rangle\partial_i\phi=-\langle\tilde\nu(0,0),\frac{\partial}{\partial\epsilon}\bigg|_{\epsilon=\mu=0}\partial_i f_\epsilon\rangle\partial_i\phi=\partial_i\varphi\partial_i\phi,$$
where we used that $\tilde\nu(0,0)=-\phi$. For $\tilde g=\delta$ we have $\tilde\nu_{\R^2}=e_3$ and $\tilde h^{\R^2}=0$. Using Equations (\ref{Hlinearization}) and (\ref{erivativesofnormals}) we get
$$D_1 B_0[0,\delta]\varphi=\left(\langle \frac{\partial\tilde\nu}{\partial\epsilon}\bigg|_{\epsilon,\mu=0}, e_3\rangle, \frac\partial{\partial\eta}D_1 H[0,\delta]\varphi\right)=\left(\langle \nabla\omega^3,\nabla\varphi\rangle, -\frac{\partial(\Delta+2)\varphi}{\partial\eta}\right)$$
 at $\omega_0$. As $\nabla\omega^3=e_3=\eta$ the first identity follows.\\

 To establish the second formula we use Lemma 3 from \cite{AK}, $H[0,\delta]=2$, $\tilde\nu(0,0)=-\phi$ and note $\tilde h^{\R^2}_{ij}\equiv 0$ for $\mu=0$ to get 
\begin{equation}\label{appcomputation0} D_2B[0,\delta]\mathfrak g'(0)=\left(\frac d{d\mu}\bigg|_{\mu=0}\frac{\langle \tilde \nu(0,\mu), e_3\rangle}{\sqrt{\mathfrak g(\mu)^{33}}},\frac{\partial}{\partial\eta}D_2 H[0,\delta]\mathfrak g'(0)+2\frac{d\tilde h^{\R^2}}{\partial\mu}\bigg|_{\mu=0}(\phi,\phi) \right).
\end{equation}
 $\tilde \nu(0,0)=-\omega\perp e_3$ and $\mathfrak g'(0)=q$. Using Equation (\ref{erivativesofnormals}) then gives
\begin{equation}\label{appcomputation1}
     \frac d{d\mu}\bigg|_{\mu=0}\frac{\langle \tilde\nu(0,\mu), e_3\rangle}{\sqrt{\mathfrak g(\mu)^{33}}}= q(\phi,\partial_i\phi)\partial_i\phi^3= q_{\mu\nu}\omega^\mu\langle\nabla\omega^\nu,\nabla\omega^3\rangle=q_{\mu 3}\omega^\mu
 \end{equation}
 at $\omega_0$. As $q(t\omega)=t^2 q(\omega)$ we learn $\nabla_{\tilde\nu} q=-2q$. Using Equation (\ref{Hderivative}) then gives $D_2 H[0,\delta] q=3q(\tilde\nu,\tilde\nu)+\tr_{\R^3}\nabla_\cdot q(\tilde\nu,\cdot)$. Inserting $\tilde\nu=-\omega$ we may write $\tr_{\R^3}\nabla_\cdot q(x,\cdot)=\partial_\mu q_{\mu\nu} x^\nu$ and $q(\tilde\nu,\tilde\nu)=q_{\mu\nu}x^\mu x^\nu$. As we take $\partial_\eta$ we only need terms that contain precisely one $x^3$. Hence at $\omega_0$
\begin{equation}\label{appcomputation2}D_2\left(\frac {\partial H}{\partial\tilde\eta}\right)[\delta] q=\frac{\partial}{\partial\eta}\left(6\partial_ a h_{ij}\omega^ i\omega^ j\omega^ a\omega^ 3-2\partial_jh_{ja} \omega^a\omega^3\right)=6\partial_ a h_{ij}\omega^ i\omega^ j\omega^ a-2\partial_jh_{ja} \omega^a.
\end{equation}
Finally, we must linearize $\tilde h^{\R^2}_{ij}=\mathfrak g_{\alpha\beta}\Gamma^\alpha_{\ ij}\tilde\nu_{\R^2}^ \beta$. As $\Gamma^ \alpha_{\ ij}$ vanishes for $\mu=0$ we must only take the derivatives of the Christoffel symbols into account. Using $\partial_3 q=0$ an easy computation then shows
\begin{equation}\label{appcomputation3}
D_2 \left(H\tilde h^{\R^2}(\tilde\nu,\tilde\nu)\right)q=2\omega^ i\omega^ jD \Gamma^ 3_{\ ij}[\delta]q=(2\partial_i q_{3j}-\partial_3 q_{ij})\omega^ i\omega^ j=4\partial_a h_{ij}\omega^ i\omega^ j\omega^ a.
\end{equation}
Equations (\ref{appcomputation0}), (\ref{appcomputation1}), (\ref{appcomputation2}) and (\ref{appcomputation3}) imply the second formula in the Lemma.
\end{proof}

\section{Solution of the elliptic problem}\label{AKcritpoint}
Alessandroni and Kuwert \cite{AK} study the elliptic problem (\ref{elliptic}) for $\phi\in\mathcal S(\lambda,\theta)$ by making the ansatz 
$$\phi^{p,\lambda}=F^\lambda[p, f_u].$$
For given $p\in S$ they first derive a solution to the elliptic problem with prescribed barycenter $C[\phi]=p$. This is achieved by applying the diffeomorphism $F^\lambda[p,\cdot]$ and studying the Equation in $\R^3$ with the background metric $\tilde g=\tilde g^{p,\lambda}$:
$$\left\{\begin{aligned}
    &P_{K[u,\tilde g]}^\perp W[u,\tilde g]=0,\\
    &B_0[u,\tilde g]=0,\\
    &C[u,\tilde g]=0,\\
    &A[u,\tilde g]=2\pi.
\end{aligned}\right.$$
It is shown\footnote{See Lemma 6 and Proposition 1 in their paper. Also note the varying definitions of $B$.} that the unique solution $u(\lambda)\in C^{4,\gamma}(\Sp^2_+)$ depends regularly on $\lambda$. In fact, it is shown that for $\Omega\in C^m$ and $l=m-1\geq 6$ we have $\tilde g^{p,\lambda}\in G_0^l(\delta;\sigma_0)$ for $\lambda\leq \lambda_0(\Omega,\sigma_0)$ and that $u\in C^{l-4}([0,\lambda_1), C^{4,\gamma}(\Sp^2_+))$ for some small $\lambda_1(\Omega)>0$  . Putting $\varphi:=u'(0)$ and using Lemma \ref{parity2} we see that 
$$\left\{\begin{aligned}
    &\Delta(\Delta+2)\varphi=\textrm{even},\\
    &\frac{\partial\varphi}{\partial\eta},\frac{\partial\Delta\varphi}{\partial\eta}=\textrm{even},\\
    &-\frac3{2\pi}\int_{\Sp^2_+}\varphi \omega^i=0,\\
    &2\int_{\Sp^2_+}\varphi=-\frac d{d\lambda}\bigg|_{\lambda=0}A[0, \tilde g^{p,\lambda}].
\end{aligned}\right.$$
The uniqueness of this linear problem implies that $\varphi$ is even and hence $\|u^-\|_{C^{4,\gamma}}\leq C\lambda^2$.  Thus $\phi^{p,\lambda}\in \mathcal S^-(\lambda, K)$ for suitable $K$. Finally, in Section 3 of their paper a suitable choice for $p$ is derived that makes $\phi^{p,\lambda}$ a solution of (\ref{elliptic}).

\section{Parabolic Schauder Theory and regularity of meta maps}\label{schauder}
\setcounter{equation}0
The following definitions are taken from \cite{eidelman}. Let $\Omega\subset\R^n$ be a bounded domain, $T>0$ and $\Omega_T:=[0,T]\times \Omega$. For $(l,\alpha)\in\N_0\times\N_0^n$ we abbreviate $D^{l,\alpha} u(t,x):=\partial_t^l D_x^\alpha u(t,x)$. For $m\in\N_0$ let 
\begin{align*} 
&C^m(\bar \Omega_T):=\set{u:\bar\Omega_T\rightarrow\R\ |\ \textrm{$D^{l,\alpha }u$ is defined in $\operatorname{int}\Omega_T$ for $4l+|\alpha|\leq m$  and continuous on $\bar\Omega_T$}}.\\
&\|u\|_{C^{m}(\bar\Omega_T)}:=\sum_{4l+|\alpha|\leq m}\|D^{l,\alpha}\|_{C^0(\bar\Omega_T)}.
\end{align*}
For $u\in C^m(\bar\Omega_T)$ and $\gamma\in(0,1)$ we define the \emph{temporal} and \emph{spatial}-Hölder seminorms
\begin{align*} 
[u]_{\gamma}^s:&=\sup_{4l+|\alpha|= m}\sup_{(x,t)\neq (y,t)}\frac{|D^{l,\alpha}u(x,t)-D^{l,\alpha} u(y,t)|}{|x-y|^{\gamma}},\\
[u]_{\gamma}^t:&=\sup_{0<m+\gamma-4l-|\alpha|\leq 4}\sup_{(x,t)\neq (x,s)}\frac{|D^{l,\alpha}u(x,t)-D^{l,\alpha} u(x,s)|}{|s-t|^{\frac{m+\gamma-4l-|\alpha|}4}}
\end{align*}
and put $\|u\|_{C^{m,\gamma}(\bar\Omega_T)}=\|u\|_{C^{m}(\bar\Omega_T)}+[u]_{\gamma}^s+[u]_{\gamma}^t$. Then the parabolic Hölder space $C^{m,\gamma}(\bar\Omega_T)$ is defined as
$$C^{m,\gamma}(\bar\Omega_T):=\set{u\in C^{m}(\bar\Omega_T)\ |\ \|u\|_{C^{m,\gamma}(\bar\Omega_T)}<\infty}.$$
We also refer to this space as $C^{m,[\frac m4]^-,\gamma}$ where $[\cdot]^-$ denotes the floor function. The definition of boundary spaces such as $C^{3,0,\gamma}([0,T]\times \partial \Omega)$ is analogue and also covered in \cite{eidelman}. Lifting the definitions onto a manifold $M$ works as usual and is e.g. covered in \cite{eidelman} for the case when $M=\partial U$ for a sufficiently regular domain $U$.\\

\paragraph{Improved Schauder estimates}\ \\
All Schauder theory except for the decay estimate used e.g. in Theorem \ref{firstestimate} are standard and may be derived by following Simon's scaling argument \cite{simon}. To prove the decay consider a function $\varphi\in X_T$ satisfying $\varphi(t,\cdot)\in X_0^\perp$ (recall Equation (\ref{directdec3})) for all $t\in[0,T]$ and  
\begin{equation}\label{coercivity}
\left\{\begin{aligned}
\dot\varphi+\frac12\Delta(\Delta+2)\varphi&=0\hspace{.2cm}\text{on $[0,T]\times \Sp^2_+$}\\
\varphi(0,\cdot)&=\psi_0\hspace{.2cm}\text{on $\set 0\times \Sp^2_+$}.
\end{aligned}\right.
\end{equation}
Following Lemma 4 in \cite{AK} we see that for all $u_0\in X_0^\perp $ we have $$\int_{\Sp^2_+} u_0\Delta(\Delta+2) u_0d\mu_{\Sp^2}\geq \lambda_2(\lambda_2-2)\int_{\Sp^2_+}u_0^2=24\int_{\Sp^2_+}u_0^2d\mu_{\Sp^2}$$
where $\lambda_2=6$ is the second non-zero eigenvalue of $-\Delta$.  Let $\lambda':=24$, pick $\mu\in(0,\frac{\lambda'}2)$ and define $\varphi_\mu(t,\omega):=e^{\mu t}\varphi(t,\omega)$. Then  $\varphi_\mu$ has zero boundary conditions along $\partial \Sp^2_+$ and solves
$$\left\{\begin{aligned}
\dot\varphi_\mu+\frac12\Delta(\Delta+2)\varphi_\mu-\mu\varphi_\mu&=0\hspace{.2cm}\text{on $[0,T]\times \Sp^2_+$},\\
\varphi_\mu(0,\cdot)&=\psi_0\hspace{.2cm}\text{on $\set 0\times \Sp^2_+$}.
\end{aligned}\right.$$
Standard Schauder theory then gives 
\begin{equation}\label{improvedschauder}
\|\varphi_\mu\|_{C^{4,1,\gamma}_T}\leq C(\Sp^2_+,\gamma,\mu)T^N\left(\|\psi_0\|_{C^{4,\gamma}}+\sup_{t\in[0,T]}\|\varphi_\mu(t)\|_{L^2(\Sp^2_+)}\right)
\end{equation}
for some $N=N(\Sp^2_+,\gamma)\in \N$. The scaling of the Schauder constant with $T$ follows by scaling arguments. Indeed, an estimate of the form (\ref{improvedschauder}) is usually derived and a suitable $L^2$-estimate inserted to derive the usual Schauder estimate. Using the fact that $\varphi(t,\cdot)\in X_0^\perp$ for all $t\in [0,T]$  and remembering (\ref{coercivity}) we compute
\begin{align*} 
\frac d{dt}\frac12\int_{\Sp^2_+} \varphi_\mu^2d\mu_{\Sp^2}&=-\frac12\int_{\Sp^2_+} \varphi_\mu\Delta(\Delta+2)\varphi_\mu d\mu_{\Sp^2}+\mu\int_{\Sp^2_+} \varphi_\mu^2d\mu_{\Sp^2}\\
&\leq \left(-\frac{\lambda'}2+\mu\right)\int_{\Sp^2_+}\varphi_\mu^2d\mu_{\Sp^2}\leq 0.
\end{align*}
This implies $\sup_{t\in[0,T]}\|\varphi_\mu(t,\cdot)\|_{L^2(\Sp^2_+)}^2=\|\psi_0\|_{L^2(\Sp^2_+)}^2$ and therefore 
$$e^{\mu T}\|\varphi(T,\cdot)\|_{C^{4,\gamma}}=\|\varphi_\mu(T,\cdot)\|_{C^{4,\gamma}}\leq \|\varphi_\mu\|_{C^{4,1,\gamma}_T}\leq C(\Sp^2_+,\gamma)T^N\|\psi_0\|_{C^{4,\gamma}}.$$
Choose $\mu:=\frac13\lambda'$ and note $e^{-\frac{\lambda'}{12}}T^{N(\Sp^2_+,\gamma)}\leq C(\Sp^2_+,\gamma)$. Thus 
$$\|\varphi(T)\|_{C^{4,\gamma}}\leq C(\Sp^2_+,\gamma)e^{-\frac{\lambda'}4T}\|\psi_0\|_{C^{4,\gamma}}=C(\Sp^2_+,\gamma)e^{-6T}\|\psi_0\|_{C^{4,\gamma}}.$$

\paragraph{Regularity of meta maps}\ \\
Let $U$ be a bounded domain and $V$ be a bounded convex domain. Put $A:= C^{0,\frac\gamma4}([0,T], C^l(V,\R))$, $B:= C^{0,0,\gamma}([0,T]\times U, V)$ and $C:=C^{0,0,\gamma}([0,T]\times U)$. For $g\in A$ and $f\in B$ put $(g\diamond f)(t,x):=g(t,f(t,x))$ and consider the meta map
$$T: A\times B\rightarrow C,\ (g,f)\mapsto g\diamond f.$$
It is easy to see that $T$ is well defined if $l\geq 1$. $T$ is even continuous if $l\geq 2$. As an example we show that  $[g\diamond f-g\diamond\tilde f]^t_\gamma$ is small if $f\approx\tilde f$. Indeed
\begin{align*}
\Delta:=&|g(t,f(t,x))-g(t,\tilde f(t,x))-g(s,f(s,x))+g(s,\tilde f(s,x))|\\
    \leq & |g(t,f(t,x))-g(t,\tilde f(t,x))-g(t,f(s,x))+g(t,\tilde f(s,x))|\\
       & +|g(t,f(s,x))-g(t,\tilde f(s,x))-g(s,f(s,x))+g(s,\tilde f(s,x))|\\
       \leq & |g(t, \lambda f(t,x)+(1-\lambda) \tilde f(t,x))-g(t, \lambda f(s,x)+(1-\lambda) \tilde f(s,x))\big|_{\lambda=0}^{\lambda=1}|\\
    & +|g(t, \lambda f(s,x)+(1-\lambda) \tilde f(s,x))-g(s, \lambda f(s,x)+(1-\lambda) \tilde f(s,x))\big|_{\lambda=0}^{\lambda=1}|.
    \end{align*}
Convexity of $V$ and the chain rule then readily imply
$$\Delta\leq \left(\|D_2 g\|_{C^0}[f-\tilde f]^t_\gamma+\|D_2^2 g\|_{C^0}\|f-\tilde f\|_{C^0}([f]^t_\gamma+[\tilde f]^t_\gamma)+[D_2 g]^t_\gamma\|f-\tilde f\|_{C^0}\right)|t-s|^{\frac\gamma4},$$
which shows the continuity of $T$. Similarly one shows that for $T\in C^k$ it suffices if $l\geq 2+k $ as differentiating $g\diamond f$ with respect to $f$ produces $Dg\diamond f$ which must be continuous in $f$. Similar arguments allow one so study $g\diamond f \in C^{4,1,\gamma}([0,T]\times U)$ for $f\in C^{4,1,\gamma}([0,T]\times U, V)$. The corresponding meta map is of class $C^k$ as long as $l\geq 6+k$ as four additional derivatives of $g$ are required.

\section{The Riemannian barycenter}\label{Barycenter}
\setcounter{equation}0
Let $\tilde g\in G_0^l(\delta;\epsilon)$ with $l\geq 2$. This appendix serves the purpose of constructing the Riemannian barycenter of the immersion $f_u:\Sp^2_+\rightarrow (\R^3,\tilde g)$ for small enough $u$. The analysis follows the construction in \cite{AK} but changes the euclidean projection $\pi_{\R^2}$ used in \cite{AK} to the Riemannian projection $\pi_{\R^2}[\tilde g,\cdot ]$. We recall 
$D_r:=\set{x\in\R^2\ |\ |x|< r}$ and $Z_r:=\bar D_r\times[-r,r]$.
\begin{lemma}
For $\epsilon>0$ small enough the following is true: For each $p\in Z_{\frac32}$ there exists a unique $x\in D_2$, denoted by $\pi_{\R^2}[\tilde g, p]$, such that $\tilde g_x(p-x,\R^2)=0$. The map $\pi_{\R^2}:G_0^l(\delta;\epsilon)\times Z_{\frac32}\rightarrow D_2,\ (\tilde g, p)\mapsto x$ is of class $C^{l}$.
\end{lemma}
\begin{proof}
First we prove uniqueness. Fix a point $p=(\vec p, p^3)\in Z_{\frac32}$ and consider the map
$$\Phi:D_2(0)\subset \R^2\rightarrow \R^2,\ f(q):=\vec p-(\delta-\tilde g_q)(p-q,e_i)e_i.$$
For small $\epsilon>0$ it is easy to show that $\Phi:D_2\rightarrow D_2 $ defines a contraction which implies the existence and uniqueness. To check the regularity we consider the $C^ l$-map $\tilde \Phi:Z_{\frac32}\times D_2\times G_\epsilon\rightarrow \R^2,\ \tilde f(p, q, \tilde g):=\tilde g_q(p-q,e_i)e_i$. Given any $p_0=(\vec p_0,p^3_0)\in Z_{\frac32}$ we note $\tilde \Phi(p_0,\vec p_0,\delta)=0$ and $D_2\tilde \Phi(p_0,\vec p_0,\delta)=-\operatorname{id}_{\R^2}$. The regularity of the local and hence global inverse follows from the implicit function theorem. 
\end{proof}
We briefly motivate the modified definition of the Riemannian barycenter. For an immersion $f_u:\Sp^2_+\rightarrow (\R^3,\tilde g)$ we wish to define the Riemannian barycenter $x\in\R^2$ by the implicit equation 
$$I:=\avint_{\Sp^2_+}\left(\exp^{\tilde g}_x\right)^{-1}(f_u(\omega))d\mu_g(\omega)\overset!\perp_{\tilde g_x}\R^2.$$
Note that by definition $I\in T_x\R^3$ and hence it is sensible to demand orthogonality with respect to $\tilde g_x$. We may reformulate the condition as $\tilde g_x(x+I-x,\R^2)=0$ which implies $\pi_{\R^2}[\tilde g, x+I]=x$. Hence we study zeros of the map
$$X:U_\epsilon\times G_\epsilon\times D_\epsilon\rightarrow\R^2,\ X[u,\tilde g, x]:=\pi_{\R^2}\left[\tilde g,x+\avint_{\Sp^2_+}\left(\exp^{\tilde g}_x\right)^{-1}(f_u(\omega))d\mu_g(\omega)\right]-x,$$
which reduces to $X'$ (we include the prime to distinguish the map from \cite{AK} to the one studied here) from \cite{AK} if we replace $\pi_{\R^2}[\tilde g,\cdot]$ with $\pi_{\R^2}[\delta,\cdot]$ as is used there. In particular both maps are identical if $\tilde g=\delta$ is inserted. Repeating the analysis from \cite{AK} we get the following Theorem: 

\begin{theorem}[The two dimensional barycenter]\label{barycenterexi}\ \\
There exist $\epsilon>0$ and $\rho>0$ such that for $u\in C^{1}(\Sp^2_+,\R)$ and $\tilde g\in C^2(Z_2,M_3(\R))$ satisfying $\|u\|_{C^1}<\epsilon$ and $\|\tilde g-\delta\|_{C^2}<\epsilon$ there exists a unique point $x=C[u,\tilde g]\in D_\rho(0)\subset\R^2$ such that $X[u,\tilde g, C[u,\tilde g]]=0$ and 
$$|x|\leq C(\|u\|_{C^{1}(\Sp^2_+)}+\|\tilde g-\delta\|_{C^2}).$$
As map from $G_0^l\times C^{4,\gamma}(\Sp^2_+)\rightarrow \R^2$ the map $C$ is of class $C^{l-1}$. 
\end{theorem}
\begin{proof}
As $X'[\cdot,\delta,\cdot]=X[\cdot,\delta,\cdot]$ the proof from \cite{AK} carries over. 
\end{proof}
As $X[\cdot,\delta,\cdot]=X'[\cdot,\delta,\cdot]$ we conclude the same explicit formula for $C[u,\delta]$ as is given in \cite{AK}:
$$C[u,\delta]=\pi_{\R^2}\left(\avint f_ud\mu_{f_u^*\delta}\right)$$
Following \cite{AK} we now compute the $L^2$-gradient of $C^ i$. The only difference in the analysis that is required is to also take the derivative of the projection into account. Let $f(\epsilon)$ be a variation of $f_u$ with $\dot f(0)=\varphi\tilde\nu$ where $\tilde\nu$ is the inner normal of $f_u$. We set $g:=f_u^*\tilde g$, $x(\epsilon):=C[f(\epsilon),\tilde g]$, $I[\tilde g, x, f]:=\avint_{\Sp^2_+}\left(\exp^{\tilde g}_x\right)^{-1}(f)d\mu_g$ and compute 
\begin{align}
    0&=\frac d{d\epsilon}\bigg|_{\epsilon=0}\pi_{\R^2}\left[\tilde g, x+\avint_{\Sp^2_+}\left(\exp^{\tilde g}_x\right)^{-1}(f(\epsilon))d\mu_{f(\epsilon)^*\tilde g}\right]-\frac {dx}{d\epsilon}\bigg|_{\epsilon=0}\nonumber\\
    &=D_2\pi_{\R^2}\left[\tilde g,x+I[\tilde g, x ,f_u]\right]\frac d{d\epsilon}\bigg|_{\epsilon=0}\left(x+\avint_{\Sp^2_+}\left(\exp^{\tilde g}_x\right)^{-1}(f(\epsilon))d\mu_{f(\epsilon)^*\tilde g}\right)-\frac {dx}{d\epsilon}\bigg|_{\epsilon=0}\label{baryepsder}
\end{align}
We investigate the derivative:
\begin{equation}\label{intepsder}
\begin{aligned}
    \frac d{d\epsilon}\bigg|_{\epsilon=0}\avint_{\Sp^2_+}\left(\exp^{\tilde g}_x\right)^{-1}(f(\epsilon))d\mu_{f(\epsilon)^*\tilde g}=&\frac1{\mu_g(\Sp^2_+)^2}\int_{\Sp^2_+}\varphi Hd\mu_g\int_{\Sp^2_+}\left(\exp^{\tilde g}_x\right)^{-1}(f_u(\omega))d\mu_g\\
    &\hspace{-3cm}+\avint_{\Sp^2_+} D\left(\left(\exp^{\tilde g}_x\right)^{-1}\right)(f_u(\omega))\varphi\tilde\nu-\left(\exp^{\tilde g}_x\right)^{-1}(f_u(\omega))H\varphi d\mu_g\\
     &\hspace{-3cm}+\avint_{\Sp^2_+} D_x\left(\left(\exp^{\tilde g}_x\right)^{-1}\right)(f_u(\omega))d\mu_g\frac{dx}{d\epsilon}\bigg|_{\epsilon=0}
\end{aligned}
\end{equation}
Inserting (\ref{intepsder}) into (\ref{baryepsder}) and dropping the arguments on $I$ for spatial reasons gives
\begin{align*}
    \frac {dx}{d\epsilon}\bigg|_{\epsilon=0}=&-\left(D_2\pi_{\R^2}[\tilde g, x+I]\left(\avint_{\Sp^2_+}  D_x\left(\left(\exp^{\tilde g}_x\right)^{-1}\right)(f_u(\omega))d\mu_g +\operatorname{id}_{\R^2}\right)-\operatorname{id}_{\R^2}\right)^{-1}\\
   & D_2\pi_{\R^2}[\tilde g, x+I]    \left(\avint_{\Sp^2_+}\varphi Hd\mu_g\avint_{\Sp^2_+}\left(\exp^{\tilde g}_x\right)^{-1}(f_u(\omega))d\mu_g\right.\\
   &\left.\hspace{3cm}+\avint_{\Sp^2_+} D\left(\left(\exp^{\tilde g}_x\right)^{-1}\right)(f_u(\omega))\varphi\tilde\nu-\left(\exp^{\tilde g}_x\right)^{-1}(f_u(\omega))H\varphi d\mu_g\right).
\end{align*}
As the Riemannian barycenter is invariant under reparameterization the $L^2$-gradient of $C^i$ is normal along $f_u$. We may therefore multiply the actual $L^2$-gradient with $\tilde\nu$ to obtain a scalar function which we denote by $\nabla C^i[u,\tilde g]$. By definition we then have 
$$\frac d{d\epsilon}\bigg|_{\epsilon=0}C^i[f(\epsilon),\tilde g]=\int_{\Sp^2_+}\nabla C^i[u,\tilde g]\tilde g(\frac{\partial f}{\partial\epsilon}\bigg|_{\epsilon=0},\tilde\nu)d\mu_g.$$
The analysis above gives the explicit formula 
\begin{align*} 
\sum_{i=1}^2\nabla C^i[u,\tilde g]e_i=&-\frac1{\mu_g(\Sp^2_+)}\left(D_2\pi_{\R^2}[\tilde g, x+I]\left(\avint_{\Sp^2_+}  D_x\left(\left(\exp^{\tilde g}_x\right)^{-1}\right)(f_u(\omega))d\mu_g + \operatorname{id}_{\R^2}\right)-\operatorname{id}_{\R^2}\right)^{-1}\\
&D_2\pi_{\R^2}[\tilde g, x+I]\left(\avint_{\Sp^2_+}\left(\exp^{\tilde g}_x\right)^{-1}(f_u)d\mu_gH\right.\\
&\hspace{3cm}\left.+ D\left(\left(\exp^{\tilde g}_x\right)^{-1}\right)(f_u(\omega))-\left(\exp^{\tilde g}_x\right)^{-1}(f_u(\omega))H\right).
\end{align*}
Specializing to $\tilde g=\delta$ gives 
$$\nabla C^i[u,\delta]=\frac{-1}{\mu_g(\Sp^2_+)}\pi_{\R^2}\left(\avint_{\Sp^2_+}\left(\exp^{\delta}_x\right)^{-1}(f_u)d\mu_gH+ D\left(\left(\exp^{\delta}_x\right)^{-1}\right)(f_u)-\left(\exp^{\delta}_x\right)^{-1}(f_u)H\right)^i.$$
Note that $\avint_{\Sp^2_+}\left(\exp^{\delta}_x\right)^{-1}(f_u)d\mu_g\perp_\delta \R^2$ by definition of $x=C[u,\delta]$. This eliminates the first term and thus establishes the same formula that is derived in \cite{AK}. In particular we recover 
$$\nabla C^i[0,\delta]=-\frac{3}{2\pi}\omega^i.$$

\paragraph{A note on regularity}\ \\
We consider $\nabla C^i$ as a map $\nabla C^i:G_T^l\times C^{4,1,\gamma}([0,T]\times \Sp^2_+)\rightarrow C^{0,\frac\gamma4}([0,T],\R^2)$ defined on a neighbourhood of $(\delta, 0)$. We use the fact that the map $(x,v,\tilde g)\mapsto \exp_x^{\tilde g}v\in\R^3$ defined on a suitably large neighbourhood of $(0,0,\delta)\in\R^2\times\R^3\times G_0^l$ is of class $C^{l-1}$. This is proven in Appendix 4 in \cite{AK}. Combined with the explicit formula for $\nabla C^i$ we get that $\nabla C^i$ is of class $C^{l-7}$ as long as $l\geq 7$.
\paragraph{Application}\ \\
We now define $C[\phi_{p,u}^\lambda]$ by applying Theorem \ref{barycenterexi} to $f_u$ with the pullback metric $\tilde g^{p,\lambda}$. This gives a point $x=C[u,\tilde g^{p,\lambda}]\in\R^2$ and we wish to define 
$$C[\phi]:=F^\lambda[p, x].$$
We must now prove that this definition does not depend on $p$ and the chosen frame.
\begin{proof}
Abbreviate $\tilde g:=\tilde g^{p,\lambda}$. Rewriting the defining equation for $x$ gives 
\begin{align}
    x=C[f_u,\tilde g]&\Leftrightarrow \int_{\Sp^2_+}\left(\exp^{\tilde g}_x\right)^{-1}(f_u(\omega))d\mu_{f_u^*\tilde g}\perp\R^2 \text{ with respect to $g_x$}\nonumber\\
    &\Leftrightarrow \int_{\Sp^2_+}\left(\exp^{\tilde g}_x\right)^{-1}(f_u(\omega))d\mu_{f_u^*\tilde g}=k\tilde\nu_{\R^2}(x)\text{ for some $k\in\R$}.\label{orthcond}
\end{align}
We note that $F^\lambda[p,\cdot]$ is an isometry up to the scaling factor  $\lambda^2$ in the definition of $\tilde g^{p,\lambda}$. Applying the differential $(DF^\lambda[p,x])^{-1}$ to (\ref{orthcond}) we learn 
\begin{align*}
   x=C[f_u,\tilde g] &\Leftrightarrow (D F^\lambda[p,x])^{-1}\left[\int_{\Sp^2_+}\left(\exp^{\tilde g}_x\right)^{-1}(f_u(\omega))d\mu_{f^*\tilde g}\right]=\lambda k N^S(F^\lambda[p,x])\\
    &\Leftrightarrow\int_{\Sp^2_+}\left(\exp^{\delta}_{F[p,x]}\right)^{-1}(F^\lambda[p, f_u(\omega)])d\mu_{\phi^*\delta}=\lambda kN^S(F^\lambda[p,x])\\
    &\Leftrightarrow \int_{\Sp^2_+}(f_u(\omega)-F[p,x])d\mu_{\phi^*\delta}=\lambda N^S(F^\lambda[p,x]).
\end{align*}
For $\lambda$ small enough we see that $x=C[f_u,\tilde g]$ is equivalent to
    \begin{equation}\label{barysddgfg}
     \pi_S\int_{\Sp^2_+}\phi(\omega)d\mu_{\phi^*\delta}=F^\lambda[p,x],
    \end{equation}
    where $\pi_S$ is the nearest point projection onto $S$. As the left hand side of equation (\ref{barysddgfg}) is independent of the choices for $p$ and the frame the right hand side must be too. 
\end{proof}
An immediate consequence is that for $\phi=F^\lambda[p, f_u]$ we have $C[\phi]=p$ if and only if $C[f_u,\tilde g^{p,\lambda}]=0$. We close by proving that we may always parameterize $\phi$ over its barycenter. 

\begin{theorem}
There exists $\lambda_0>0$, $\theta_0>0$ such that for $\lambda\leq\lambda_0$ each $\phi\in\mathcal S'(\lambda,\theta_0)$ may be parameterized over its barycenter $C[\phi]$. That is, there exists a unique graph function $\tilde u\in C^{4,\gamma}(\Sp^2_+)$ such that 
$$\phi=F^\lambda[C[\phi], f_{\tilde u}].$$
\end{theorem}
\begin{proof}
As $|\varphi[p,\lambda x]|\leq C\lambda^2$ we derive $F^\lambda[p,Z_2]=\operatorname{im}F^\lambda[p,\cdot]\supset B_{\frac32\lambda}(p)$ for small enough $\lambda$. Next we have the following claim:
\begin{claim}{1}{There exist $\epsilon,\delta,\lambda>0$ such that $F^\lambda[p,x]\in\operatorname{im}F^\lambda[q,\cdot]$ for $p,q\in S$ and $x\in Z_2$ satisfying $d(p,q)<\lambda\delta$, $|x|\leq1+\epsilon$.}
For small enough $\lambda$ we have $\operatorname{im}F^\lambda[q,\cdot]\supset B_{\frac32\lambda}(q)$. Now if $|x|<1+\epsilon$ then 
$$|F^\lambda[p,x]-q|\leq |p-q|+\lambda|x|+|\varphi[q,\lambda x]|\leq \lambda\delta+\lambda(1+\epsilon)+C\lambda^2$$
which implies the claim for a sufficiently small choice of $\epsilon,\ \delta$ and $\lambda$.
\end{claim}
Now suppose that $\phi\in \mathcal S'(\epsilon,\lambda)$. Then we may write $\phi=F^\lambda[p, f_u]$ for some $p\in S$ and $\|u\|_{C^{4,\gamma}}<\epsilon$. Let $q\in S$ denote the barycenter of $\phi$. By definition $q=C[\phi]=F^\lambda[p, C[f_u,\tilde g^{p,\lambda}]$ and using the Estimate from Theorem \ref{barycenterexi} we get $| C[f_u,\tilde g^{p,\lambda}|<\delta(\epsilon,\lambda)$ with $\delta\rightarrow 0$ as $\lambda,\epsilon\rightarrow 0$. This gives $d(p,q)\leq C\lambda\delta(\epsilon,\lambda)$. As $|f_u(\omega)|<1+\epsilon$ we may use Claim\# 1 for small enough $\epsilon$ and $\lambda$ to define $\tilde u$ by 
$$f_{\tilde u}(\omega)=\left(F^\lambda[q,\cdot]\right)^{-1}\left(F^\lambda[p, f_u(\omega)]\right).$$
Uniqueness is easily established.  
\end{proof}

\section{The constraint space}\label{Kspace}
\setcounter{equation}0
Let $\tilde g\in C^5(Z_2,M_3(\R))$ be a metric that is close to the euclidean metric $\|\tilde g-\delta\|_{C^5}\leq \epsilon$. Also consider a function $u\in C^{4,\gamma}(\Sp^2_+)$ that satisfies $\|u\|_{C^{4,\gamma}}\leq \epsilon$. We consider the $L^2$-gradients of the area functional $A[u,\tilde g]$ and the barycenter components $C^i[u,\tilde g]$ and put
\begin{equation}\label{kspacedef}
    K[u,\tilde g]:=\operatorname{span}_\R\set{\nabla A[u,\tilde g], \nabla C^i[u,\tilde g]}.
\end{equation}
In the Appendix of \cite{AK} the following formulas are derived for  $\tilde g=\delta$ and $u=0$:
\begin{equation}\label{Kforround}
\nabla A[0,\delta]=-2\hspace{.5cm}\text{and}\hspace{.5cm}\nabla C^i[0,\delta]=-\frac 3{2\pi}\omega^i
\end{equation}
There it is also shown that the maps $(\tilde g,u)\mapsto \nabla A[u,\tilde g],\,\nabla C^i[u,\tilde g]\in C^{0,\gamma}(\Sp^2_+)$ are of class $C^2$. Hence we learn that for $\epsilon$ small enough the following quantities are well defined:
$$\psi_0[u,\tilde g]:=\frac{\nabla A[u,\tilde g]}{\|\nabla A[u,\tilde g]\|_{L^2(f_u^*\tilde g)}}\hspace{.5cm}\textrm{and}\hspace{.5cm}\psi_i[u,\tilde g]:=\frac{\nabla C^i[u,\tilde g]}{\|\nabla C^i[u,\tilde g]\|_{L^2(f_u^*\tilde g)}}$$
Clearly $(\psi_\mu[u,\tilde g])_{\mu=0}^2$ constitutes a generating system for $K[u,\tilde g]$. (\ref{Kforround}) implies that for $u=0$ and $\tilde g=\delta$ the functions $(\psi[0,\delta])_{\mu=0}^2$ even provide an $L^2(\Sp^2_+)$-orthonormal basis of $K[0,\delta]$. Let 
$$A_{\mu\nu}[u,\tilde g]:=\langle \psi_\mu[u,\tilde g],\psi_\nu[u,\tilde g]\rangle_{L^2 (f_u^*\tilde g)}.$$
As $A_{\mu\nu}[0,\delta]=\delta_{\mu\nu}$ we learn that for $\epsilon$ small enough $(\psi_\mu[u,\tilde g])_{\mu=0}^2$ constitutes a basis of $K[u,\tilde g]$ and that the matrix $A_{\mu\nu}[u,\tilde g]$ is invertible. As a consequence we may write the $L^2(f_u^*\tilde g)$-projection of any $X$ onto the space $K[u,\tilde g]$ as 
\begin{equation}\label{projectionKspacedefinition}
P_{K[u,\tilde g]}(X):=\sum_{\mu,\nu=0}^2A^{\mu\nu}[u,\tilde g]\langle \psi_\mu[u,\tilde g], X\rangle_{L^2 (f_u^*\tilde g)}\psi_\nu[u,\tilde g].
\end{equation}
We denote the complementary projection onto the $L^2(f_u^*\tilde g)$-orthogonal complement $K[u,\tilde g]^\perp$ by $P_{K[u,\tilde g]}^\perp$. As $u$ and $\tilde g$ are usually clear from context we often drop them in the notation as simply write $P_{K}$ and $P_{K}^\perp$. \\

The metric $\tilde g^{p,\lambda}$ satisfies $\|\tilde g^{p,\lambda}\|_{C^{n-1}}\leq C(\Omega)\lambda$ if $\Omega$ is of class $C^n$. Thus we learn that for $\lambda\leq\lambda_0(\Omega, \epsilon)$ we may always achieve $\|\tilde g-\delta\|_{C^5}<\epsilon$.

\section{Computation of metric derivative}\label{metricderivative}
\setcounter{equation}0
We abbreviate $f:=f_u$, $\tilde g(t):=\tilde g^{\xi(t),\lambda}$ and denote the $L^2(f^*\tilde g)$ scalar product by $\langle\cdot,\cdot\rangle$. We also put $\phi(t):=F^\lambda[\xi(t),f(t)]$. We have to compute $D_2 C^ i[u,\tilde g]\dot{\tilde g}$ as well as $D_2 A[u,\tilde g]\dot{\tilde g}$. 
\paragraph{Barycenter}\ \\
We begin by investigating the barycenter. As $C^i[f(t),\tilde g(t)]=0$ we may write 
\begin{equation}\label{11111}
D_2 C[f(t),\tilde g(t)]\dot{\tilde g}(t)=-\langle \nabla C^i[f(t),\tilde g(t)],\tilde g(\dot f(t),\tilde\nu)\rangle.
\end{equation}
The geometric object the evolves in time is $\phi$. Its evolution is then decomposed into two parts. $f$ only represents the evolution of the `sphere-shape' while the movement of the barycenter is not contained in $f$. However, to exploit that $C[\phi(t)]=\xi(t)$ we must take the full evolution of $\phi$ into account. For that purpose we fix $t$ and consider small time displacements $\epsilon$ from $t$. We may then write 
$$\phi(t+\epsilon)=F^\lambda[\xi(t), h(t,\epsilon)]$$
where $h(t,0)=f(t)$. Unlike $f$ the quantity $h$ encodes the full evolution of $\phi$ as $t$ is a fixed time and the dynamical variable is now $\epsilon$. Expressing $\xi(t+\epsilon)=C[\phi(t+\epsilon)]$ inside the chart centered at $\xi(t)$ as $\xi(t+\epsilon)=f[\xi(t),(\xi^1(t+\epsilon),\xi^2(t+\epsilon))$ and differentiating at $\epsilon=0$ gives
\begin{align}
    \lambda^{-1}\dot\xi^i(t)&=\frac d{d\epsilon}\bigg|_{\epsilon=0}\lambda^{-1}\xi^i(t+\varepsilon)\nonumber\\
    &=\frac d{d\epsilon}\bigg|_{\epsilon=0}C^i[h(t,\epsilon),\tilde g^{\xi(t),\lambda}]\nonumber\\
    &=\langle\nabla C^i[f(t),\tilde g(t)],\tilde g(\frac{\partial h}{\partial\epsilon}\bigg|_{\epsilon=0},\tilde\nu)\rangle.\label{dotciiapp}
\end{align}
Here $\tilde\nu$ is the inner normal of $h$. Next we must relate $\partial_\epsilon h(t,0)$ to $\dot f(t)$. For that purpose we use the relation $F^\lambda[\xi(t+\epsilon), f(t+\epsilon)]=\phi(t)=F^\lambda[\xi(t),h(t,\epsilon)]$ to obtain  
\begin{equation}\label{22222}
\begin{aligned}
    D_2 F^\lambda[\xi(t), h(t,0)]\frac{\partial h}{\partial\epsilon}\bigg|_{\epsilon=0}&=\frac d{d\epsilon}\bigg|_{\epsilon=0}F^\lambda[\xi(t),h(t,\epsilon)]\\
    &=\frac d{d\epsilon}\bigg|_{\epsilon=0}F^\lambda[\xi(t+\epsilon), f(t+\epsilon)]\\
    &=D_1F^\lambda[\xi(t), f(t)]\dot\xi(t)+D_2 F^\lambda[\xi(t), f(t)]\dot f(t).
\end{aligned}
\end{equation}
Remembering $h(t,0)=f(t)$, multiplying with $\left(D_2 F^\lambda[\xi(t), f(t)]\right)^{-1}$, recalling $X$ from Equation (\ref{equivequation}) and inserting into equations (\ref{11111}) and (\ref{dotciiapp}) yields 
\begin{align}
    D_2 C^i[f(t),\tilde g(t)]\overset{(\ref{11111})}=&-\langle \nabla C^i[f(t),\tilde g(t)],\tilde g(\dot f(t),\tilde\nu)\rangle\nonumber\\
    =&-\langle \nabla C^i[f(t),\tilde g(t)],\tilde g(\frac{\partial h}{\partial\epsilon}\bigg|_{\epsilon=0},\tilde\nu)-X\rangle\nonumber\\
    \overset{(\ref{dotciiapp})}=&-\frac1\lambda\dot\xi^i(t)+\langle\nabla C^i[f(t),\tilde g(t)],X\rangle.\label{33333}
\end{align}

\paragraph{Area}\ \\
We copy the derivation from above. This time we use that $A[f(t),\tilde g(t)]=2\pi$. Taking the derivative then gives 
\begin{equation}\label{11111a}
D_2 C[f(t),\tilde g(t)]\dot{\tilde g}(t)=-\langle \nabla C^i[f(t),\tilde g(t)],\tilde g(\dot f(t),\tilde\nu)\rangle.
\end{equation}
Next we must use that $A[\phi(t)]=2\pi\lambda^2$ and again use $h(t,\epsilon)$ for small $\epsilon$ to obtain
$$0= \langle\nabla A[f(t),\tilde g(t)],\tilde g(\frac{\partial h}{\partial\epsilon}\bigg|_{\epsilon=0},\tilde\nu)\rangle.$$
We reuse Equation (\ref{22222}) to derive the analogue of Equation (\ref{33333}) given by
\begin{equation}\label{33333n}
D_2 A[f(t),\tilde g(t)]\dot{\tilde g}=\langle\nabla A[f(t),\tilde g(t)],X\rangle.
\end{equation}

\paragraph{Conclusion}\ \\
We may now substitute Equations (\ref{33333}) and (\ref{33333n}) into the definition of $\tau$ in Equation (\ref{atomheartmother5}) to get
\begin{align}
\tau[u,\tilde g]\overset{(\ref{atomheartmother5})}=&-\sum_{\mu,\nu=0}^3A^ {\mu\nu}[u,\tilde g]\frac{D_2 C^\mu[u,\tilde g]\dot{\tilde g}}{\|\nabla C^\mu[u,\tilde g]\|_{L^2(f_u^*\tilde g)}}\frac{\nabla C^\nu[u,\tilde g^{\xi,\lambda}]}{\|\nabla C^\nu[u,\tilde g]\|_{L^2(f_u^*\tilde g)}}\label{tauequationappendix}\\
\overset{\substack{(\ref{33333}),\\(\ref{33333n})}}=&-\sum_{\mu,\nu=0}^3A^ {\mu\nu}[u,\tilde g]\frac{\langle\nabla C^\mu[u,\tilde g], X\rangle}{\|\nabla C^\mu[u,\tilde g]\|_{L^2(f_u^*\tilde g)}}\frac{\nabla C^\nu[u,\tilde g^{\xi,\lambda}]}{\|\nabla C^\nu[u,\tilde g]\|_{L^2(f_u^*\tilde g)}}\\
&+\sum_{\nu=0}^2\sum_{i=1}^2A^{i\nu}[u,\tilde g]\frac{\dot\xi^i}{\lambda\|\nabla C^i[u,\tilde g]\|_{L^2(f_u^*\tilde g)}}\frac{\nabla C^\nu[u,\tilde g^{\xi,\lambda}]}{\|\nabla C^\nu[u,\tilde g]\|_{L^2(f_u^*\tilde g)}}\nonumber\\
=&-P_K(X)+\sum_{\nu=0}^2\sum_{i=1}^2A^{i\nu}[u,\tilde g]\frac{\dot\xi^i}{\lambda\|\nabla C^i[u,\tilde g]\|_{L^2(f_u^*\tilde g)}}\frac{\nabla C^\nu[u,\tilde g^{\xi,\lambda}]}{\|\nabla C^\nu[u,\tilde g]\|_{L^2(f_u^*\tilde g)}}.\nonumber
\end{align}

\section*{Acknowledgements}
The author would like to thank Ernst Kuwert for the suggestion of this interesting topic and the helpful guidance as well as Marius Müller for the many helpful discussions.

\bibliographystyle{plain}
\bibliography{quellen}

\end{document}